\documentclass[11pt]{article}

\title{Embedded contact homology of the unit cotangent bundle of the Klein bottle}
\author{Marcelo Miranda and Vinicius G. B. Ramos}
\date{}

\usepackage{amssymb}
\usepackage{latexsym}

\usepackage{amsmath}
\usepackage{amsthm}
\usepackage{amscd}
\usepackage{color}
\usepackage{graphicx}
\usepackage{pgfplots}
\usepackage{tikz-cd}
\usepackage{fdsymbol}
\usepackage{nicefrac}
\usepackage{tikz}

\numberwithin{equation}{section}

\newtheorem{theorem}{Theorem}[section]

\newtheorem{proposition}[theorem]{Proposition}
\newtheorem{corollary}[theorem]{Corollary}
\newtheorem{lemma}[theorem]{Lemma}
\newtheorem{lemma-definition}[theorem]{Lemma-Definition}

\theoremstyle{definition}
\newtheorem{definition}[theorem]{Definition}

\newtheorem{remark}[theorem]{Remark}

\newtheorem{example}[theorem]{Example}

\newcommand{\floor}[1]{\left\lfloor #1 \right\rfloor}
\newcommand{\ceil}[1]{\left\lceil #1 \right\rceil}

\newcommand{\bpm}{\begin{pmatrix}}
\newcommand{\epm}{\end{pmatrix}}

\renewcommand{\epsilon}{\varepsilon}

\begin{document}
\maketitle

\begin{abstract}
We give a combinatorial description of the embedded contact homology chain complex of the unit cotangent bundle of the Klein bottle with the standard flat Riemannian metric. Using pseudoholomorphic curves coming from the associated differential, we find an obstruction theorem for symplectic embeddings of toric domains $X_{\Omega}\subset \mathbb{C}^2$ into the unit disk cotangent bundle $D^*K$. As an application we compute the Gromov width of $D^*K$.
\end{abstract}

\setcounter{tocdepth}{2}

\section{Introduction}

\subsection{Overview of ECH in the toric setting}

Let $Y$ be a closed, connected and oriented three-manifold. The embedded contact homology (ECH), introduced by Hutchings \cite{Hutchings3} is a powerful invariant that has been computed for various manifolds. It is a homology theory associated with each nondegenerate contact form $\lambda$ on $Y$. Most of computations of ECH were carried out in the presence of toric symmetries as in \cite{Hutchings4}, see also  \cite{Hutchings3, Keon}.

In this paper, we consider $Y=U^*(K)$, the unit cotangent bundle of the flat Klein bottle. We first observe that this is topologically a $T^2$-bundle over the circle; however, these tori are not invariant by the standard contact form $\lambda_{std}$ in $D^*(K)$ (see Remark \ref{remark1}). We provide an ``almost toric" description of $(U^*(K),\lambda_{std})$ with invariant tori $T^2$. Based on that, we give a combinatorial description of the ECH generators, differential and capacities.

An important example is the torus $T^3$ with its canonical contact form was considered in \cite{Hutchings4}. In this setting, a combinatorial description of the ECH chain complex was given using ``lattice paths" as generators and a ``rounding the corners" operation as the differential. The main idea was to use the fact that $T^3$ is a $T^2$-fibration with invariant tori. Then using Morse-Bott pertubations as in \cite{Bourgeois}, up to a fixed action $L>0$, a finite number of tori foliated by Reeb orbits splits in only two Reeb orbits: one elliptic and other positive hyperbolic. For each orbit set $\alpha$ with zero homology in $H_{1}(T^3,\mathbb{Z})$ one can associate a lattice path defined as the concatenation of vectors in $\mathbb{Z}^2$ corresponding to each Reeb orbit that appearing in $\alpha$.

For the unit cotangent bundle of Klein bottle $U^*(K)$ with its canonical contact form, the situation is similar, but has a significant difference from the setting in $T^3$. Instead, the structure is ``almost toric". As we will see in Section \ref{almosttoric}, there exists an open subset in $D^*(K)$ that is strictly contactomorphic to
$$
((-\pi/2,\pi/2)\times T^2, \cos(\theta)dx+\sin(\theta)dy),
$$
whose complement consists of two disjoint Klein bottles. The key difference from the toric case is that, after perturbing the contact form near these Klein bottles, two negative hyperbolic Reeb orbits emerge for each Klein bottle.

\subsection{Convex K-lattice path}\label{intro: K-lattice}

To describe the ECH generators, we start with a definition. A convex K-lattice path $\Lambda$ is a piecewise linear path in $\mathbb{R}^2$ satisfying:

\begin{itemize}
    \item The initial point is $(0,0)$ and the final point is $(x(\Lambda),0)$.
    \item $\Lambda$ is a graph of a piecewise linear convex function $f:[0,x(\Lambda)]\rightarrow [0,+\infty)\times (-\infty,0]$, such that the non-smooth points are in $\mathbb{Z}^2$, possibly together vertical lines after the starting point or before the end point.
    \item The edges are labeled by ``e" (elliptic), ``h" (positive hyperbolic) or ``h$^-$" (negative hyperbolic).
    \item Vertical edges cannot be labeled by ``h".
\end{itemize}

On top of those, we ask for $\Lambda$ to satisfy one extra condition. Depending on which one of them it satisfies, we say that the K-convex lattice path is of one of the following four types:

\begin{itemize}
    \item \textbf{Type I}. (Toric generator). In this case $\Lambda$ is a lattice path in the plane as above, but the end point $(x(\Lambda),0)$ satisfies $x(\Lambda)\equiv 0\pmod{2}$

    %%%%%%%%%%%%%%%%%%%%%%%%%%%%%example
    \begin{center}
    \begin{tikzpicture}
    % Draw lattice points inside the polygon
    \foreach \x in {0,1,2} {
    \foreach \y in {0,-1,-2} {
        \fill (\x,\y) circle (1pt);
    }
    }
    % Draw lattice points on the boundary of the polygon
    \foreach \x/\y in {0/0, 0/-2, 2/-1, 2/0} {
    \fill (\x,\y) circle (1pt);
    }
    % Draw the polygon with orientation arrows
    \draw[->] (0,0) -- node[left]{$e$} (0,-1);
    \draw[->](0,-1) -- node[left]{$e$} (0,-2);
    \draw[->] (0,-2) -- node[pos=0.6, below]{$h$} (2,-1);
    \draw[->] (2,-1) -- node[right]{$e$} (2,0);
    \end{tikzpicture}
    \end{center}
    %%%%%%%%%%%%%%%%%%%%%%%%%%%%%example
    
    \item \textbf{Type II}. In this case $\Lambda$ starts with two vertical down-arrows of lenght $1/2$, each labeled by $h^-$. We require that the end point $(x(\Lambda),0)$ satisfies $x(\Lambda)\equiv 1 \pmod{2}$.

    %%%%%%%%%%%%%%%%%%%%%%%%%%%%%example
    \begin{center}
    \begin{tikzpicture}
    % Draw lattice points inside the polygon
    \foreach \x in {0,1,2} {
    \foreach \y in {0,-1,-2} {
        \fill (\x,\y) circle (1pt);
    }
    }
    % Draw lattice points on the boundary of the polygon
    \foreach \x/\y in {0/0, 0/-2, 2/-1, 2/0} {
    \fill (\x,\y) circle (1pt);
    }
    % Draw the polygon with orientation arrows
    \draw[->] (0,0) -- node[left]{$h^-$} (0,-1/2);
    \draw[->] (0,-1/2) -- node[left]{$h^-$} (0,-1);
    \draw[->](0,-1) -- node[left]{$e$} (0,-2);
    \draw[->] (0,-2) -- node[right]{$h$} (1,0);
    \end{tikzpicture}
    \end{center}
    %%%%%%%%%%%%%%%%%%%%%%%%%%%%%example

    \item \textbf{Type III}. In this case $\Lambda$ ends with two vertical up-arrows of lenght $1/2$, each labeled by $h^-$. We require that the end point $(x(\Lambda),0)$ satisfies $x(\Lambda)\equiv 1 \pmod{2}$.

    %%%%%%%%%%%%%%%%%%%%%%%%%%%%%example
    \begin{center}
    \begin{tikzpicture}
    % Draw lattice points inside the polygon
    \foreach \x in {0,1,2,3} {
    \foreach \y in {0,-1,-2} {
        \fill (\x,\y) circle (1pt);
    }
    }
    % Draw lattice points on the boundary of the polygon
    \foreach \x/\y in {0/0, 0/-2, 2/-1, 2/0} {
    \fill (\x,\y) circle (1pt);
    }
    % Draw the polygon with orientation arrows
    \draw[->] (0,0) -- node[left]{$e$} (0,-1);
    \draw[->](0,-1) -- node[left]{$e$} (0,-2);
    \draw[->] (0,-2) -- node[pos=0.55, below]{$h$} (3,-1);
    \draw[->] (3,-1) -- node[right]{$h^-$} (3,-1/2);
    \draw[->] (3,-1/2) -- node[right]{$h^-$} (3,0);
    \end{tikzpicture}
    \end{center}
    %%%%%%%%%%%%%%%%%%%%%%%%%%%%%example

    \item \textbf{Type IV}. In this case $\Lambda$ start and finish with vertical arrows of lenght $1/2$, but we require that the end point $(x(\Lambda),0)$ satisfies $x(\Lambda)\equiv 0 \pmod{2}$.

    %%%%%%%%%%%%%%%%%%%%%%%%%%%%%example
    \begin{center}
    \begin{tikzpicture}
    % Draw lattice points inside the polygon
    \foreach \x in {0,1,2,3} {
    \foreach \y in {0,-1,-2} {
        \fill (\x,\y) circle (1pt);
    }
    }
    % Draw lattice points on the boundary of the polygon
    \foreach \x/\y in {0/0, 0/-2, 2/-1, 2/0} {
    \fill (\x,\y) circle (1pt);
    }
    % Draw the polygon with orientation arrows
    \draw[->] (0,0) -- node[left]{$h^-$} (0,-1/2);
    \draw[->] (0,-1/2) -- node[left]{$h^-$} (0,-1);
    \draw[->](0,-1) -- node[left]{$e$} (0,-2);
    \draw[->] (0,-2) -- node[pos=0.6, below]{$h$} (2,-1);
    \draw[->] (2,-1) -- node[right]{$h^-$} (2,-1/2);
    \draw[->] (2,-1/2) -- node[right]{$h^-$} (2,0);
    \end{tikzpicture}
    \end{center}
    %%%%%%%%%%%%%%%%%%%%%%%%%%%%%example
    
\end{itemize}

\begin{remark}
   As expect, the ``half arrows'' with label $h^-$, correspond to negative hyperbolic Reeb orbits. Sometimes we say "toric part" of $P$ as the lattice path considering only the edges which are not labeled by $h^-$. The conditions of parity in the $x(\Lambda)$ coordinate is important to guarantee that the correspondent ECH generator has zero homology in $H_1(U^*(K),\mathbb{Z})$. We also abbreviate Convex K-lattice path to K-lattice path for simplicity.
\end{remark}

\begin{remark}
    We say ``toric arrow" by any arrow labeled with ``e" or ``h".
\end{remark}

\subsection{The combinatorial grading}

Let $P$ a K-lattice path, its grading $I(P)$ is defined as:
$$
I(P)=2\textrm{Area}(P)+m(P)-h(P),
$$
where the Area($P$) is the area enclosed by $P$ and the $x$-axis, $m(P)$ denotes the number of toric edges of $P$ (counted with multiplicities) and $h(P)$ denotes the number of edges labeled by ``h" of $P$.

\begin{remark}
    It is easy see that $I(P)$ is always nonnegative, and the only index zero generators are points and the unique K-lattice path with no toric part of \textbf{Type IV}.
\end{remark}

\subsection{Generators and differential}\label{descrip: differential}

Consider $C_k$ the $\mathbb{Z}_2$ vector space generated by $K$-lattice path of index $k\in \mathbb{Z}$. We can define a differential $\delta:C_{*}\rightarrow C_{*-1}$ using three operations, which we describe as follows.

\begin{itemize}
    \item \textbf{Interior rounding the corners}. Let $\Lambda$ an $K$-lattice path, then $R_{int}(P)$ is defined as
    $$
    R_{int}(P)=\sum_{Q}Q
    $$
    where the sum is taken over all $Q$ which is obtain from $P$ by doing rounding at some corners which is not on the $x$-axis, or if $P$ is not of \textbf{Type I}, we do not do rouning at a vertex of $P$ which touch arrows with label $h^-$. We label the new edges in the following rule, if both edges around the vertex are labeled by $``e"$, then all new edges are labeled by $``e"$. Otherwise we consider all ways to labeled a new edge with exactly one $``h"$ and the others labeled by $``e"$.

    \item \textbf{C operation}. Suppose that $P$ is a $K$-lattice path such that its first arrow is toric and labeled by $h$, that is, $P$ is not of \textbf{Type II}, such that the slope of its first arrow is in the interval $(-\infty,-1]$. Then the $C$ operation correspond to doing rounding in the start point of $P$, but the first arrow is two vertical down-arrows labeled by $h^-$, all other new edges are labeled with $``e"$. Similarly if $P$ has last edges labeled by $h$ with slope in $[1,\infty)$, then the $C$ operations correspond to doing rounding in the end point of $P$ but, the last arrow now is two vertical up-arrows labeled by $h^-$ with all other new edges labeled with $``e"$.

    Suppose now that the slope of the first arrow of $P$ is labeled by $``h"$, and its the slope lies in the interval $(-1,0]$, the $C$ operation correspond to take the $K$-lattice path obtained from $P$ exluding the starting point $(0,0)$ and $(1,0)$ in the $x$-axes and then taking the convex hull of the remaining points. The created new edges are labeled by $``e"$. Similarly, if the last edge of $P$ is labeled by $``h"$ and its slope is in the interval $[0,1)$, then the $C$ operation correspond to take the $K$-lattice path obtained from $P$ by excluded the last point of $P$ and $(x(P)-1,0)$ and take the convex hull, all new edges are lebeled by $``e"$.

    \item \textbf{D operation}. Suppose that $P$ is a $K$-lattice path of \textbf{Type II} and its first toric edge is labeled by $``h"$. The $D$ operation correspond to forget the vertex in the down-arrows labeled by $h^-$ and labeled all new edges by $``e"$. Similarly, if $P$ is of \textbf{Type III} and its last toric edge is labeled by $``h"$, then the $D$ operation is forget the vertex in the up-arrows labeled by $h^-$ and labeled all new edges by "e". For a generator of \textbf{Type IV} its differential follows the same rules.
\end{itemize}

\begin{definition}
    The combinatorial $ECH$ differential, denote by $\delta$, is defined on $K$-lattice path using the interior rounding and the C, D operations, that is,
    $$
    \delta(\Lambda)=C(\Lambda)+D(\Lambda)+R_{\textrm{int}}(\Lambda).
    $$

    We observe that the $C(\Lambda)$ is obtained by doing $C$ operation in the $K$-lattice path $\Lambda$ in the first edge and in the last edge, that is, $C(\Lambda)$ can be a sum of two generators or one, the same to the $D$ operation.
\end{definition}

\begin{example}
    We compute below, the differential of a generator which has granding equal to two.

\begin{center}
\begin{tikzpicture}
    % Left lattice path
    \begin{scope}
        % Draw lattice points inside the polygon (smaller points)
        \foreach \x in {0,1,2} {
            \foreach \y in {0,-1} {
                \fill (\x,\y) circle (1pt);
            }
        }
        % Draw lattice points on the boundary of the polygon, excluding (0,-2)
        \foreach \x/\y in {0/0, 2/-1, 2/0} {
            \fill (\x,\y) circle (1pt);
        }
        % Draw the polygon with orientation arrows
        \draw[->] (0,0) -- node[left]{$h$} (1,-1);
        \draw[->] (1,-1) -- node[right]{$h$} (2,0);
        % Place the \partial symbol to the left of the lattice path
        \node[left] at (-0.5,-0.5) {\large $\delta:$};
        % Place the mapsto symbol
        \node[right] at (2.5,-0.5) {\large $\mapsto$};
    \end{scope}

    % Right lattice path (translated to the right)
    \begin{scope}[shift={(4,0)}]
        % Draw lattice points inside the polygon
        \foreach \x in {0,1} {
            \foreach \y in {0,-1} {
                \fill (\x,\y) circle (1pt);
            }
        }
        % Draw lattice points on the boundary
        \foreach \x/\y in {0/0, 1/-1} {
            \fill (\x,\y) circle (1pt);
        }
        % Draw the polygon with orientation arrows
        \draw[->] (0,0) -- node[left]{$h^-$} (0,-1/2);
        \draw[->] (0,-1/2) -- node[left]{$h^-$} (0,-1);
        \draw[->] (0,-1) -- node[right]{$h$} (1,0);
        % Place the plus symbol
        \node[right] at (1.6,-0.5) {\large $+$};
    \end{scope}

    % Second right lattice path (translated to the right)
    \begin{scope}[shift={(7,0)}]
        % Draw lattice points inside the polygon
        \foreach \x in {0,1} {
            \foreach \y in {0,-1} {
                \fill (\x,\y) circle (1pt);
            }
        }
        % Draw lattice points on the boundary
        \foreach \x/\y in {0/0, 1/-1} {
            \fill (\x,\y) circle (1pt);
        }
        % Draw the polygon with orientation arrows
        \draw[->] (0,0) -- node[left]{$h$} (1,-1);
        \draw[->] (1,-1) -- node[right]{$h^-$} (1,-1/2);
        \draw[->] (1,-1/2) -- node[right]{$h^-$} (1,0);
        % Place the plus symbol
        \node[right] at (1.8,-0.5) {\large $+$};
    \end{scope}

    % Third right lattice path (translated to the right)
    \begin{scope}[shift={(10,0)}]
        % Draw lattice points along the path
        \foreach \x in {0,1,2} {
            \fill (\x,-0.5) circle (1pt);
        }
        % Draw the polygon with orientation arrows
        \draw[->] (0,-0.5) -- node[midway, above]{$e$} (1,-0.5);
        \draw[->] (1,-0.5) -- node[midway, above]{$h$} (2,-0.5);
    \end{scope}
\end{tikzpicture}
\end{center}

where the second generator was obtained by apply the $C$ operation in the first arrow and the third by apply $C$ operation in the second arrow, the last generator was obtained by doing interior rounding in the vertex $(1,-1)$.
\end{example}

Using Morse-Bott pertubation techniques from \cite{Bourgeois}, we can perturb the standard contact form of $U^*(K)$ as explained in the section \ref{pertubation}. Let $P$ be a K-lattice path, we defined its action, denoted by $A(P)$ to be sum of the euclidean lenghts of its edges, and extend by linearity in $(C,\delta)$. It is not hard to see that the differential $\delta$ decreases action, that is, if $\langle \delta(P),Q\rangle=1$, then $A(P)\geq A(Q)$. Therefore we can consider the filtred subcomplex
$$
(C_{*}^L,\delta),
$$
where we consider only generators with action less than $L$. The Stokes theorem also implies that the ECH differential descreases action, therefore we also have a filtred subcomplex on ECH. Our main result is the following.

\begin{theorem} 
    For every $L>0$, the perturbed contact form $\lambda$ and almost complex structure $J$ can be chosen such that $(ECC_{*}^L(U^*(K),\lambda,0,J),\partial)$ is isomorphic to $(C_{*}^L,\delta)$.
\end{theorem}

\subsection{ECH spectrum and applications}

We compute the ECH spectrum of $(U^*(K),\lambda_{st})$ in terms of $K$-lattice paths.

\begin{theorem}\label{combcapacity}
    The ECH spectrum of $U^*(K)$ with the standard contact form is given by
    \begin{equation}\label{formulaspecECH}
    c_{k}(U^*(K),\lambda_{st})=\min \{A(P)\mid |P|=2k\},
    \end{equation}
    where the minimum is taken over $K$-lattice path.
\end{theorem}

We also obtain an ``almost toric" description of $D^*(K)$. Recall the definition of toric domain. Let $\Omega$ be a domain in the first quadrant of the plane, define the toric domain
$$
X_{\Omega}=\{(z_1,z_2)\in \mathbb{C}^2\mid \pi(|z_1|^2,|z_2|^2)\in \Omega\},
$$
equipped with the standar symplectic form from $\mathbb{C}^2$. For example if $\Omega$ is a triangle with vertices $(0,0)$, $(a,0)$ and $(0,b)$, then $X_{\Omega}$ is the ellipsoid $E(a,b)$, in particular we consider the ball $B^4(a)=E(a,a)$. If $\Omega$ is the retangle with vertices $(0,0), (a,0),(0,b)$ and $(a,b)$, then $P(a,b):=X_{\Omega}$ is the polydisk. 

\begin{proposition}\label{semidisk}
    There is a symplectic embedding from the toric semidisk $D^2_{+}\times T^2$ into $D^*K$, where
    $$
    D^2_{+}=\{(I_1,I_2)\in \mathbb{R}^2\mid I_1^2+I_2^2<1\textrm{ and }I_1>0\},
    $$
    with its standard symplectic form $\omega_0=dx\wedge dI_1+dy\wedge dI_2 $
\end{proposition}

\begin{proof}
    Consider the following open set of $U^*(K)$
    $$
    A=\{(p,\eta)\in D^*(K)\mid x(\eta)\neq 0\},
    $$
    where $x(\eta)$ is the $x$-coordinate in $\mathbb{R}^2$ of any lift of the covector $\eta$ under the derivative of the natural projection $\pi:\mathbb{R}^2\rightarrow K$, is true that this is not a well defined number, but the condition that this is zero is well defined, $A$ is the complement of this set. Now consider the following map $F:(0,1)\times (-\pi/2,\pi/2)\times T^2\rightarrow D^*(K)$, given by
    $$
    F(r,\theta,[x,y])=(\pi(x,y),rd\pi_{(x,y)}(\cos(\theta),\sin(\theta))).
    $$
    \textbf{Claim.} The map $F$ is an symplectic embedding.

    Indeed, we will prove that $F^*(\omega_{st})=\omega_0$. In a almost global coordinates for $D^*(K)$ we can write
    $$
    \omega_{st}=dx\wedge dz+dy\wedge dw
    $$
    where $(x,y)$ are local coordinates in $K$ (but in an open dense set) from the projection $\pi$ and $(z,w)$ are induced coordinates in the fiber. Consider the following ``action" coordinates
    $$
    I_1=r\cos(\theta)\quad and\quad I_2=r\sin(\theta),
    $$
    since $r\in (0,1)$ and $\theta\in (\pi/2,\pi/2)$, this retangle is transformed in $I_1^2+I_2^2<1$ with $I_1>0$, which is exactly the toric semidisk $D²_{+}$.
    Moreover, we have in this coordinates $F^*(\omega_{st})=dx\wedge dI_1+dy\wedge dI_2$.
\end{proof}

\begin{remark}\label{remark; Gromovwidth}
    This embedding together with the ``Traynor Trick" show that we have a symplectic embedding of toric domains into $D^*(K)$,
    $$
    \textrm{int}(B^4(1-\epsilon))\hookrightarrow D^*(K)\quad\textrm{and} \quad \textrm{int}(E(1-\epsilon,2-\epsilon))\hookrightarrow D^*(K);\hspace{0.2cm}\forall \epsilon>0.
    $$ 
\end{remark}

In particular we have that the Gromov width of $D^*K$ satisfies $c_{Gr}(D^*K,\omega_{std})\geq 1$. The ECH capacities give not enough obstruction to compute the Gromov width of $(D^*(K),\omega_{std})$. We prove a combinatorial obstruction theorem (see Theorem \ref{thm: combbeyond}) for symplectic embeddings 
$$
X_{\Omega}\rightarrow D^*K,
$$
which can be applied to the ball to obtain the following theorem.

\begin{theorem}\label{thm: GROMOVWDITH}
    The Gromov width of $(D^*K,\omega_{std})$ is 
    $$
    c_{Gr}(D^*K,\omega_{std})=1.
    $$
\end{theorem}

\subsection{Outline of the proof}

If we combine the Taubes isomorphism \ref{Taubesiso} with computation results of the thesis of Eli. B. Lebow \cite{Eli}, we can compute the ECH homology in the classe zero.

\begin{theorem}\label{Homologythem}
    Let $\xi_{std}$ be the contact structure defined by the canonical contact form in $U^*(K)$, then
    $$
ECH(U^*(K),\xi_{st},0) = \left\{
\begin{array}{rcl}
\mathbb{Z}_2, & *=0,1,2,...\\
0, & \textrm{otherwise}
\end{array}
\right.
$$
\end{theorem}

\begin{proof}
    We compute the difference of the first Chern class of both contact structures $\xi_{st}$ and $\xi_{0}$ and obtain:
    $$
    c_1(\xi_{st})-c_1(\xi_0)=2(\mathfrak{s}(\xi_{st})-\mathfrak{s}(\xi_0)),
    $$
    where the contact structure $\xi_0$ was considered in \cite{Eli} and $\mathfrak{s}(\xi)$ is the spin-c structure associated to the oriented $2$-plane bundle $\xi$. Because $U^*(K)$ also can be seen as a $T^2$-torus bundle over the circle with monodromy matrix $A=-Id_{2\times 2}$. Follows from the proposition \ref{chernclass} that $c_1(\xi_{st})=0$ and cleary $c_{1}(\xi_0)=0$ (see lemma \ref{torictrivial}). Therefore, $\mathfrak{s}(\xi_{st})-\mathfrak{s}(\xi_0)$ is a torsion element in $H^2(U^*(K),\mathbb{Z})\simeq \mathbb{Z}\oplus \mathbb{Z}_2\oplus \mathbb{Z}_2$. Write
    $$
    \mathfrak{s}(\xi_{st})=\mathfrak{s}(\xi_0)+PD(\Gamma),
    $$
    where $\Gamma$ is a torsion element in $H_1(U^*(K),\mathbb{Z})$. We have the commutative diagram 
    %%%%%%%%%%%%%%%%%%DIAGRAMA
\[
\begin{tikzcd}
ECH_{*}(U^*(K),\xi_{st},0) \arrow[r, ""] \arrow[d, "\simeq"'] & ECH_{*}(U^*(K),\xi_0,\Gamma) \arrow[d, "\simeq"] \\
\widehat{HM}^
{-*}(U^*(K),\mathfrak{s}(\xi_{st}))\arrow[r, "\simeq"]                  & \widehat{HM}^{-*}(U^*(K),\mathfrak{s}(\xi_0)+PD(\Gamma))
\end{tikzcd}
\]
    %%%%%%%%%%%%%%%%%%DIAGRAMA
    where the vertical maps comming from the Taubes isomorphism between ECH and the Seiberg-Witten cohomology (see Theorem \ref{Taubesiso}) and the horizontal bellow arrow is only the identity. Therefore we have an isomorphism
    $$
    ECH_{*}(U^*(K),\xi_{st},0)\simeq ECH_{*}(U^*(K),\xi_0,\Gamma).
    $$
    Using that $\Gamma$ is torsion we have a well defined relative $\mathbb{Z}$-grading for this homology which can be refined to a absolute $\mathbb{Z}$-grading with the fact that this isomorphism sends $[\emptyset]$ to $i_0$ as in \cite{Eli}, because $(D^*(K),\omega_{st})$ is a exact symplectic filling of $(U^*(K),\lambda_{st})$ so that the ECH-class of the empty set is a non-trivial ECH generator in degree zero, moreover there is no ECH generator with grading strictly less than zero, the unique option is that the class of the empty set is send to $i_0$, this implies the Theorem \ref{Homologythem}.
\end{proof}

However, the ECH homology is a topological invariant of the three-manifold, meaning that it alone is insufficient for computing capacities. To address this, we combine homological information with existence results for pseudoholomorphic curves from \cite{Taubes} to compute the ECH differential. Throughout the proof, we also use the Weyl law of the ECH spectrum (see \ref{asymptotic}) to show that the D operation plays a role in the ECH differential.

To compute the Gromov width of $(D^*K,\omega_{std})$, we prove a combinatorial obstruction theorem for symplectic embeddings of toric domains $X_{\Omega}\subset \mathbb{C}^2\hookrightarrow D^*K$ similar to the Theorem 1.19 in \cite{Hutchings0}, and using that we concluded the computation.
%%%%%%%%%%%%%%%%%%%%%%%%%%%%%%%%%%%%%%%%%%%

\begin{flushleft}
\textbf{Acknowledgments.} The first author would like to thank Brayan Ferreira  for helpful discussions and hospitality during his visit in UFES.
\end{flushleft}

\section{Embedded contact homology theory}\label{ECHtheory}

Let $Y$ be a closed oriented three-manifold endowed of a contact form $\lambda$ on $Y$, which is a $1$-form such that $\lambda\wedge d\lambda>0$ with respect the orientation of $Y$. The contact form $\lambda$ define the contact structure $\xi=\ker\lambda$, which is an oriented two-plane distribution, and the Reeb vector field $R$ is characterized by $d\lambda(R,\cdot)=0$ and $\lambda(R)=1$.

A Reeb orbit is a closed orbit of $R$, i.e. a curve $\gamma:\mathbb{R}/T\mathbb{Z}\rightarrow Y$ for some $T>0$, module reparametrization, such that $\gamma'(t)=R(\gamma(t))$. A Reeb orbit can be embedded in $Y$, or a $m$-fold cover of an embedded Reeb orbit for some integer $m>1$.

We assume that $\lambda$ is nondegenerate, i.e. given a Reeb orbit $\gamma$, the linearized return map is the symplectic automorphism $P_{\gamma}$ on the symplectic vector space $(\xi_{\gamma(0)},d\lambda)$, which is defined as the derivative of the time $T$ flow of $R$. The Reeb orbit $\gamma$ is nondegenerate if $1$ is not an eigenvalue of $P_{\gamma}$, the contact form $\lambda$ is nondegenerate if all Reeb orbits are nondegenerate.

Since $P_{\gamma}$ is a linear symplectic automorphism in dimension two, the Reeb vector field $R$ admits only three types of Reeb orbits.
\begin{itemize}
\item Elliptic: The eigenvalues of $P_{\gamma}$ are on the unit circle.
\item Positive hyperbolic: The eigenvalues of $P_{\gamma}$ are real and positive.
\item Positive hyperbolic: The eigenvalues of $P_{\gamma}$ are real and negative.
\end{itemize}

An orbit set is a finite set of pairs $\alpha=\{(\alpha_i,m_i)\}$, where $\alpha_i$ are distinct embedded Reeb orbits and $m_i$ are positive integers. We said that $\alpha$ is admissible if $m_i=1$, whenever $\alpha_i$ is hyperbolic.

For a fixed homology class $\Gamma\in H_1(Y)$, the embedded contact homology $ECH_{*}(Y,\xi,\Gamma)$ is the homology of a chain complex $ECC(Y,\lambda,\Gamma,J)$. The chain complex is freely generated over $\mathbb{Z}/2$ by admissible orbit sets whose the total homology class is $\Gamma$
$$
[\alpha]=\sum_im_i[\alpha_i]=\Gamma.
$$

The chain complex differential is defined roughly as follows. An almost complex structure $J$ on the symplectization $\mathbb{R}\times Y$ is admissible if $J$ is $\mathbb{R}$-invariant, $J(\partial_s)=R$ where $s$ denotes the $\mathbb{R}$ coordinate, and $J$ restricts to an endomorphism on the contact structure $\xi$, satisfying $d\lambda(v,Jv)>0$ for all $v\in \xi\setminus \{0\}$.

If $\alpha=\{(\alpha_i,m_i)\}$ and $\beta=\{(\beta_j,n_j)\}$ are chain complex generators, then the differential $\langle \partial \alpha,\beta\rangle\in \mathbb{Z}/2$ is a mod $2$ count of $J$-holomorphic currents of "ECH index" equal to $1$. The definition of the ECH index is the key nontrivial part of the definition of ECH.

\subsection{The ECH index}

Now we will introduce each piece to define the ECH index.

\subsubsection{The Conley-Zehnder index}

If $\gamma:\mathbb{R}/T\rightarrow Y$ is a Reeb orbit, let $\mathcal{T}(\gamma)$ denote the set of homotopy classes of symplectic trivializations of the $2$-plane bundle $\gamma^*(\xi)$ over $S^1=\mathbb{R}/T$. This is an affine space over the integers. Our sign convention is that if $\tau_1,\tau_2:\gamma^*(\xi)\rightarrow S^1\times \mathbb{R}^2$ are two trivializations, then
$$
\tau_1-\tau_2=\textrm{deg}(\tau_2\circ \tau_1^{-1}\longrightarrow Sp(2,\mathbb{R})\simeq S^1).
$$

Let $\gamma:\mathbb{R}/T\rightarrow Y$ be a Reeb orbit and let $\tau$ be a trivialization of $\gamma^*(\xi)$. Given any $t\in \mathbb{R}$, the linearized Reeb flow along $\gamma$ from time $0$ to time $t$, defines a symplectic map $\xi_{\gamma(0)}\rightarrow \xi_{\gamma(t)}$, which with respect to the trivialization $\tau$ is a symplectic matrix $\psi(t)$. In particular, $\psi(0)$ is the identity and $\psi(T)$ is the linearized return map. Since $\gamma$ is assumed to be nondegenerate, the path of symplectic matrices $\{\psi(t) \mid 0\leq t\leq T\}$ has a well-know \textbf{Conley-Zehnder index}, which we denote by
$$
CZ_{\tau}(\gamma)\in \mathbb{Z}.
$$
%%%%%%%%%
In the three-dimensional case, this can described explicitly as follows.
\begin{itemize}
    \item If $\gamma$ is hyperbolic, then there is an integer $n$ such that the linearized Reeb flow along $\gamma$ rotates the eigenspace of the linearized return map by angle $n\pi$ with respect to $\tau$. In this case
    \begin{equation}\label{eqn:CZhyp}
    CZ_{\tau}(\gamma^k)=kn.
    \end{equation}
    
    The integer $n$ is even when $\gamma$ is positive hyperbolic and odd when $\gamma$ is negative hyperbolic.
    \item When $\gamma$ is elliptic, the trivialization $\tau$ is homotopic to a trivialization in which the linearized Reeb flow along $\gamma$ rotates by angle $2\pi \theta$. The number $\theta$ is called \textbf{monodromy angle}, is necessarily irrational because $\gamma$ and all of its iterated areassumed nondegenerate. In this case
    \begin{equation}\label{eqn:CZelip}
    CZ_{\tau}(\gamma^k)=2\floor{k\theta}+1.
    \end{equation}
\end{itemize}

The Conley-Zehnder index depends only on the Reeb orbit $\gamma$ and the homotopy class of the trivialization $\tau$. If $\tau'\in \mathcal{T}(\gamma)$ is another trivializations, then the Conley-Zehnder changes in the following way
$$
CZ_{\tau}(\gamma^k)-CZ_{\tau'}(\gamma^k)=2k(\tau'-\tau).
$$

\subsubsection{The relative first Chern class}

Let $\alpha=\{(\alpha_i,m_i)\}$ and $\beta=\{(\beta_j,n_j)\}$ be two orbit sets with $[\alpha]=[\beta]\in H_1(Y)$. Fix trivializations $\tau_i^+\in \mathcal{T}(\alpha_i)$ for each $i$ and $\tau_j^-\in \mathcal{T}(\beta_j)$ for each $j$, and denote this set of trivializations choices by $\tau$. Let $Z\in H_2(Y,\alpha,\beta)$.

\begin{definition}
    Define the \textbf{relative first Chern class}
    $$
    c_{\tau}(Z):=c_1(\xi|_{Z},\tau)\in \mathbb{Z},
    $$
    as follows. Take a representative of $Z$ as a smooth map $f:S\rightarrow Y$, where $S$ is a compact oriented surface with boundary. Choose a section $\psi$ of $f^*(\xi)$ over $S$ such that $\psi$ is transverse to the zero section, and over each boundary component of $S$, the section $\psi$ is nonvanishing and has winding number zero with respect to $\tau$. Define
    $$
    c_{\tau}(Z):=\#\psi^{-1}(0),
    $$
    where '\#' denotes the signed count with respect the orientations.
\end{definition}

If $Z'$ is another relative homology class in $H_2(Y,\alpha,\beta)$. Then
\begin{equation}\label{eq: difchern}
c_{\tau}(Z)-c_{\tau}(Z')=\langle c_1(\xi),Z-Z'\rangle,
\end{equation}
where $c_1(\xi)\in H^2(Y,\mathbb{Z})$ is the ordinary first Chern class. Moreover, if $\tau'$ is another collections of trivializations choices, then
    $$
    c_{\tau}(Z)-c_{\tau'}(Z)=\sum_{i}m_i(\tau_i^+-\tau_i^{+'})-\sum_{j}n_j(\tau_{j}^{-}-\tau_j^{-'}).
    $$

\subsection{Braids around Reeb orbits}

Let $\gamma$ be an embedded Reeb orbit and let $m$ be a positive integer.
\begin{definition}
    A braid around $\gamma$ with $m$ strands is an oriented link $\zeta$ contained in a tubular neighborhood $N$ of $\gamma$ such that the tubular neighborhood projection $\zeta\rightarrow \gamma$ is an orientation-preserving degree $m$ submersion.
\end{definition}

Choose a trivialization $\tau$ of $\gamma^*(\xi)$. Extend the trivialization to identify the tubular neighborhood $N$ with $S^1\times D^2$ such that the projection of $\zeta$ to $S^1$ is a submersion. Identify $S^1\times D^2$ with a solid torus in $\mathbb{R}^3$ via the orientation-preserving diffeomorphism
$$
(\theta,(x,y))\longmapsto (1+x/2)(\cos\theta,\sin\theta,0)-(0,0,t/2).
$$
We have now an embedding $\phi_{\tau}:N\rightarrow \mathbb{R}^3$. Now $\phi_{\tau}(\zeta)$ is an oriented link in $\mathbb{R}^3$ with no vertical tangents.
\begin{definition}
    The \textbf{writhe} of the braid $\zeta$, denote by $w_{\tau}(\zeta)$, is the signed count of crossings in the projection $\mathbb{R}^2\times \{0\}$ of $\phi_{\tau}(\zeta)$.
\end{definition}

This number depends only on the isotopy class of $\zeta$ and the homotopy class of $\tau$ in $\mathcal{T}(\gamma)$. If $\tau'\in \mathcal{T}(\gamma)$ is another trivializations, and if $\zeta$ has $m$ strands, then
$$
w_{\tau}(\zeta)-w_{\tau'}(\zeta)=m(m-1)(\tau'-\tau).
$$

\begin{definition}
    If $\zeta_1$ and $\zeta_2$ are disjoint braids around $\gamma$, define the \textbf{linking number}
    $$
    lk_{\tau}(\zeta_1,\zeta_2)\in \mathbb{Z},
    $$
    to be one half the signed count of crossings of a strand of $\phi_{\tau}(\zeta_1)$ with a strand of $\phi_{\tau}(\zeta_2)$ in the projection to $\mathbb{R}^2\times \{0\}$.
\end{definition}

If $\zeta_i$ has $m_i$ strands, for $i=1,2$, then
$$
lk_{\tau}(\zeta_1,\zeta_2)-lk_{\tau'}(\zeta_1,\zeta_2)=m_1m_2(\tau'-\tau).
$$

\begin{remark}
    The sign convetion of crossings is that counterclockwise twists contribute positively.
\end{remark}

\subsubsection{The relative intersection pairing}

\begin{definition}
    Let $\alpha=\{(\alpha_i,m_i)\}$ and $\beta=\{(\beta_j,n_j)\}$ be orbit sets with $[\alpha]=[\beta]\in H_1(Y,\mathbb{Z})$, and let $Z\in H_2(U,\alpha,\beta)$. An \textbf{admissible representative} of $Z$ is a smooth map $f:S\rightarrow [0,1]\times Y$, where $S$ is a compact oriented surface with boundary, such that:
    \begin{itemize}
        \item The restriction of $f$ to $\partial S$ are positively oriented covers of $\{1\}\times \alpha_i$ with total multiplicity $m_i$ and negatively oriented covers of $\{-1\}\beta_j$ with total multiplicty $n_j$.
        \item The composition of $f$ with the projection $[0,1]\times Y\rightarrow Y$ represents the class $Z$.
        \item The restriction of $f$ to the interior of $S$ is an embeddings, and $f$ is transverse to $\{0,1\}\times Y$.
    \end{itemize}
\end{definition}

If $S$ is an admissible representative of $Z$, let $\epsilon>0$ be sufficiently small, then $S\cap (\{1-\epsilon\}\times Y)$ consists of braids $\zeta_i^+$ with $m_i$ strands in disjoint tubular neighborhoods of the Reeb orbits $\alpha_i$, which are well defined up to isotopy. Likewise $S\cap(\{\epsilon\}\times Y)$ consists of disjoint braids $\zeta_j^-$ with $n_j$ strands in disjoint tubular neighborhoods of the Reeb orbits $\beta_j$. Give $S'$ an admissible representative of $Z'\in H_2(Y,\alpha',\beta')$, such that the interior of $S'$ does not intersect the interior of $S$ near the boundary, with braids $\zeta_i^{+'}$ and $\zeta_{j}^{-'}$. Define the linking number
$$
lk_{\tau}(S,S'):=\sum_ilk_{\tau}(\zeta_i^+,\zeta_i^{+'})-\sum_jlk_{\tau}(\zeta_j^-,\zeta_{j}^{-'}).
$$

\begin{definition}
    If $Z\in H_2(Y,\alpha,\beta)$ and $Z'\in H_2(Y,\alpha',\beta')$, define the \textbf{relative intersection number}
    $$
    Q_{\tau}(Z,Z')\in \mathbb{Z},
    $$
    as follows. Choose admissible representatives $S$ of $Z$ and $S'$ of $Z'$ whose interiors $\dot{S}$ and $\dot{S'}$ are transverse and do not intersect each other near the boundary. Define
    $$
    Q_{\tau}(Z,Z'):=\#(\dot{S}\cap \dot{S'})-lk_{\tau}(S,S').
    $$
\end{definition}

For $\tau'$ another collection of trivializations choices we have
$$
Q_{\tau}(Z,Z')-Q_{\tau'}(Z,Z')=\sum_im_im_i'(\tau_i^{+}-\tau_{i}^{+'})-\sum_j n_jn'_{j}(\tau_j^{-}-\tau_{j}^{-'}).
$$

\begin{remark}\label{rem: quadratic}
    Suppose that $Z\in H_2(Y,\alpha,\beta)$ and $Z'\in H_2(Y,\alpha',\beta')$, then $Z+Z'\in H_2(Y,\alpha\alpha',\beta\beta')$, where the product of two orbit sets is defined by adding the multiplicities of the Reeb orbits if there are some repetition of Reeb orbit, and
    \begin{equation}
    Q_{\tau}(Z+Z')=Q_{\tau}(Z)+2Q_{\tau}(Z,Z')+Q_{\tau}(Z').       
    \end{equation}
\end{remark}

\subsection{The Fredholm index in symplectizations}

We start with the following proposition that give the smooth structure on the moduli of $J$-holomorphic curve for a generic almost complex structure $J$.

\begin{proposition}
    Suppose that $J$ is a generic admissible almost complex structure on the symplectizations $\mathbb{R}\times Y$. Then every somewhere injective $J$-holomorphic curve $C$ in $\mathbb{R}\times Y$ is regular, so that the moduli space of $J$-holomorphic curve near $C$ is a manifold. Its dimension is given by the Fredholm index below.
\end{proposition}

If $C$ has $k$ positive ends at Reeb orbits $\gamma_1^+,...,\gamma_k^+$ and $l$ negative ends at Reeb orbits $\gamma_1^-,...,\gamma_l^-$. The \textbf{Fredholm index} of $C$ is defined by
$$
\textrm{ind}(C)=-\chi(C)+2c_{\tau}(C)+\sum_{i=1}^k CZ_{\tau}(\gamma_i^+)-\sum_{i=1}^lCZ(\gamma_{i}^-).
$$

\subsection{The ECH index}

The key nontrivial part of the definition of ECH is define the relative index.

\begin{definition}
    For any $Z\in H_2(Y,\alpha,\beta)$, the ECH index is given by
    $$
    I(\alpha,\beta,Z)=c_{\tau}(Z)+Q_{\tau}(Z)+CZ_{\tau}^{I}(\alpha,\beta),
    $$
    where $CZ_{\tau}^{I}$ is the Conley-Zehnder term defined as
    $$
    CZ_{\tau}^{I}(\alpha,\beta)=\sum_{i}\sum_{k=1}^{m_i}CZ_{\tau}(\alpha_i^{k})-\sum_{j}\sum_{k=1}^{n_j}CZ_{\tau}(\beta_i^{k}).
    $$
    If $C\in \mathcal{M}(\alpha,\beta)$, define $I(C)=I(\alpha,\beta,[C])$.
\end{definition}

The following proposition give us some basic properties of the ECH index.

\begin{proposition}\label{prop: ECHPRO}{(See \cite{Hutchings3})}
    Let $\alpha$ and $\beta$ be orbit sets in the homology class of $\Gamma$ and $Z\in H_2(Y,\alpha,\beta)$. The ECH index has the following properties.
    \begin{itemize}
        \item The ECH index $I(Z)$ is well defined, i.e, does not depend on the choice of trivialization $\tau$.
        \item The ECH index is additive. Let $\delta$ another orbit set in the homology class $\gamma$, and if $W\in H_2(Y,\beta,\delta)$, then $Z+W\in H_2(Y,\alpha,\delta)$ is defined and
        $$
        I(Z+W)=I(Z)+I(W).
        $$
        \item The parity of $I(Z)$ is equal to the parity of the number os positive hyperbolic orbits in $\alpha$ and $\beta$.

        \item If $Z'\in H_2(Y,\alpha,\beta)$ is another relative homology class, then
        $$
        I(Z)-I(Z')=\langle Z-Z',c_1(\xi)+2\textrm{PD}(\Gamma)\rangle.
        $$
    \end{itemize}
\end{proposition}

The following proposition give us the general structure of a $J$-holomorphic current of low ECH index.

\begin{proposition}{(See \cite{Hutchings3})}
    Suppose that $J$ is generic. Let $\alpha$ and $\beta$ be orbit sets and $C\in \mathcal{M}(\alpha,\beta)$ be any $J$-holomorphic current in $\mathbb{R}\times Y$. Then:
    \begin{enumerate}
        \item $I(C)\geq 0$, with equality if and only if $C$ is a union of trivial cylinders with multiplicities.
        \item If $I(C)=1$, then $C=C_0\sqcup C_1$, where $I(C_0)=0$, and $C_1$ is embedded and has ind$(C_1)=I(C_1)=1$.
        \item If $I(C)=2$, and if $\alpha$ and $\beta$ are generators of the chain complex $ECC_{*}(Y,\lambda,\Gamma,J)$, then $C=C_0\sqcup C_2$, where $I(C_0)=0$, and $C_2$ is embedded and has ind$(C_2)=I(C_2)=2$.
    \end{enumerate}
\end{proposition}

\subsection{The ECH differential}

We define now the ECH differential $\partial$ on the chain complex $ECC_*(Y,\lambda,\Gamma,J)$. If $\alpha$ and $\beta$ are orbit sets and $k$ is an integer, we define
$$
\mathcal{M}_k^J(\alpha,\beta)=\{\mathcal{C}\in \mathcal{M}^J(\alpha,\beta)\mid I(\mathcal{C})=k\}.
$$
%%%%%%%%%%%%%%%%%%%%%%%%%%%%%%%%%%%%%
If $\alpha$ is a chain complex generator, define
$$
\partial \alpha=\sum_{\beta}\#(\mathcal{M}_1(\alpha,\beta)/\mathbb{R})\beta,
$$
where the sum is taken over chain complex generators $\beta$, and $``\#"$ denotes the mod $2$ count. By the $\mathbb{R}$-invariance of our almost complex structure, we have a natural action on $J$-holomorphic curves by translation on the $\mathbb{R}$-coordinate. As explained in \cite{Hutchings3} the quotient $\mathcal{M}_1(\alpha,\beta)$ is a finite set, moreover there are only finitely many $\beta$ with $\mathcal{M}(\alpha,\beta)$ nonempty, so that $\partial \alpha$ is well defined. Moreover we have $\partial^2=0$ (See \cite{Hutchings3}). Therefore, we have a well defined homology theory.

\subsection{The grading of the complex}

We now define the a relative grading in the chain complex $ECC_*(Y,\lambda,\Gamma,J)$, and hence in its homology. Let $d$ the divisibility of $c_1(\xi)+2\textrm{PD}(\Gamma)$ in $H^2(Y,\mathbb{Z})$ mod torsion. If $\alpha$ and $\beta$ are two chain complex generators, we define their ''index difference'' by choosing an arbitrary $Z\in H_2(Y,\alpha,\beta)$ and put
$$
I(\alpha,\beta)=[I(\alpha,\beta,Z)]\in \mathbb{Z}/d.
$$
This defines a relatively $\mathbb{Z}/d$ grading. This is well defined by the index-ambiguity formula. When $\Gamma=0$ we can further define an absolute $\mathbb{Z}/d$ grading declaring the empty set $\emptyset$ to be a zero grading generator, so that the grading of any other generator $\alpha$ with zero homology class is
$$
|\alpha|=I(\alpha,\emptyset).
$$
By the Additivity property of the ECH index, the differential decreases this absolute grading by $1$.

\subsection{The isomorphism with Seiberg-Witten cohomology}

The definition of ECH a priori depends on the choice of a admissible almost complex structure $J$ in the symplectisation of $Y$, the way to proof that the homology does not depend of $J$ follows from a Theorem of Taubes \cite{Taubes1} asserting that if $Y$ is connected, there is a canonical isomorphism between embedded contact homology and the Seiberg-Witten cohomology defined by Kronheimer-Mrowka in \cite{Mrowka}.

\begin{theorem}\label{Taubesiso}
    Let $(Y,\lambda)$ a connected contact three manifold. Then there is a canonical isomorphism of relatively graded $\mathbb{Z}_2$ vector spaces
    $$
    ECH_{*}(Y,\lambda, \Gamma,J)\simeq \widehat{HM}^{-*}(Y,\mathfrak{s}_{\xi}+\textrm{PD}(\Gamma)),
    $$
    where $\mathfrak{s}_{\xi}$ denotes a spin-c structure determined by the oriented $2$-plane field $\xi$, see \cite{Mrowka} and $\widehat{HM}^{*}$ denotes the "from" version of Seiberg-Witten Floer cohomology.
\end{theorem}

\subsection{The $U$ map and ECH capacities}    

The $U$ map is a degree $-2$ map in ECH induced by a map in the chain complex defined by the count $J$-holomorphic currents of ECH index $2$ passing through a specified point, more precisely, fixe a $z\in Y$ which is not in any Reeb orbit. Let $\alpha$ and $\beta$ be generators of $ECC_{*}(Y,\lambda,\Gamma,J)$, the set of ECH index $2$ passing through $(0,z)$
$$
\mathcal{M}_{2,z}(\alpha,\beta)=\{\mathcal{C}\in \mathcal{M}_2(\alpha,\beta)\mid (0,z)\in \mathcal{C}\},
$$
is a finite set. Define $U_z:ECC_{*}(Y,\lambda,\Gamma,J)\rightarrow ECC_{*-2}(Y,\lambda,\Gamma,J)$ the chain map complex by
$$
U_z(\alpha)=\sum_{\beta}\#_2\mathcal{M}_{2,z}(\alpha,\beta) \beta
$$
The $U_{z}$ map has the property that it commutes with the ECH differential, that is, $U\partial=\partial U$, therefore induces a well defined map in homology which does not depend of the choice of the base point $z$, see \cite{Hutchings3}.

Let $Y$ be a connected three manifold with a nondegenerate contact form $\lambda$. The ECH spectrum is a sequence of numbers
$$
0=c_{0}(Y,\lambda)<c_1(Y,\lambda)\leq c_{2}(Y,\lambda)\leq \cdots \leq \infty.
$$
%%%%%%%%%%%%%%%%%%%%
It is defined by the following formula
$$
c_{k}(Y,\lambda)=\inf \{A(\alpha)\mid U^k[\alpha]=[\emptyset] \}, \hspace{0.5cm} \forall k\geq 1,
$$
where the infimum is taken over all ECH generators $\alpha$ with zero homology in $H_1(Y,\lambda)$.

If $\lambda$ is nondegenerate, one defines
$$
c_{k}(Y,\lambda)=\lim_{n\rightarrow +\infty} c_{k}(Y,f_n\lambda),
$$
where $f_n$ are smooth positive functions such that $f_n\lambda$ is a nondegenerate contact form and $f_n\rightarrow 1$ in the $C^0$-topology. Now suppose that $(W,\omega)$ is a symplectic filling of $(Y,\lambda)$, that is, $\partial W=Y$ and $\omega|_{Y}=d\lambda$. Using cobordism map in ECH it is possible to show that $[\emptyset]$ is a nonzero element in the ECH. In this case we define the ECH capacities of $(W,\omega)$ as
$$
c_{k}(W,\omega):=c_{k}(Y.\lambda)\in [0,+\infty],
$$
it was shown in \cite{Hutchings3}, that theses are all symplectic capacities. 
%%%%%%%%%%%%%%%
For each $\sigma$ a non-zero element in $ECH(Y,\xi, \Gamma)$ we define
$$
c_{\sigma}(Y,\lambda)=\inf\{L>0\mid \sigma \in \textrm{im}(i:ECH^{L}(Y,\lambda,\Gamma)\rightarrow ECH(Y,\xi,\Gamma)) \}.
$$
%%%%%%%%%%%%%%%%
These sequences of numbers satisfies a "Weyl law" proved in \cite{Vinicius}.
\begin{theorem}
    Let $Y$ be a closed connected contact three-manifold with a contact form $\lambda$ and let $\Gamma\in H_1(Y,\mathbb{Z})$. Suppose that $c_1(\xi)+2\textrm{PD}(\Gamma)$ is torsion in $H^2(Y,\mathbb{Z})$ and let $I$ be an absolute $\mathbb{Z}$-grading of $ECH(Y,\xi,\Gamma)$. Let $(\sigma_k)_{k\geq 1}$ a sequence of nonzero $I$-homogeneous elements in $ECH(Y,\xi,\Gamma)$ with $\lim_{k\rightarrow +\infty} I(\sigma_k)=+\infty$. Then
    $$
    \lim_{k\rightarrow +\infty} \frac{c_{\sigma_k}(Y,\lambda)}{I(\sigma_k)}=\textrm{vol}(Y,\lambda).
    $$
\end{theorem}

As corollary on ECH capacities we obtain.
\begin{corollary}\label{asymptotic}
    Let $(Y,\lambda)$ be a closed connected contact three-manifold. Then
    $$
    \lim_{k\rightarrow +\infty} \frac{c_{k}(Y,\lambda)^2}{k}=2\textrm{vol}(Y,\lambda)=\int_{Y}\lambda \wedge d\lambda.
    $$
\end{corollary}

\subsection{Partition conditions}

Let $\alpha=\{(\alpha_i,m_i)\}$ and $\beta={(\beta_j,n_j)}$ be orbit sets. Let $C\in \mathcal{M}^{J}(\alpha,\beta)$ be a smowhere injective $J$-holomorphic curve. For each $i$, the curve $C$ has ends at covers of $\alpha_i$ whose sum covering multiplicities is $m_i$, in particular these multiplicities covers form a partition of the positive integer $m_i$, which we denote by $p_{i}^+(C)$, similarly the covering multiplicities of the negative ends of $C$ at covers of $\beta_j$ form a partition of $n_j$, which we denote by $p_{j}^{-}(C)$.

For each Reeb orbit $\gamma$ and each positive integer $m$ we define the ''positive partition'' $p_{\gamma}^+(m)$ and the ''negative partition'' $p_{\gamma}^{-}(m)$ as follows.
\begin{itemize}
    \item If $\gamma$ is positive hyperbolic, then
    $$
    p_{\gamma}^+(m)=p_{\gamma}^-(m)=(1,....,1).
    $$
    \item If $\gamma$ is negative hyperbolic, then
    $$
    p_{\gamma}^+(m)=p_{\gamma}^-(m)= \left\{
    \begin{array}{rcl}
    (2,...,2),& \mbox{se} & m\textrm{ is even},\\
    (2,...,2,1), & \mbox{se} & m\textrm{ is odd}.
    \end{array}
    \right.
    $$
    \item If $\gamma$ is an elliptic orbit with monodromy angle $\theta$ with respect to some trivialization $\tau$ of the contact structure over $\gamma$ . Then $p_{\gamma}^{\pm}$ is defined as folows.

    Let $\Lambda_{\theta}^+$ be the maximal concave polygon path in the plane with vertices at lattice points, starting at the origin and ends at $(m\theta, \floor{m\theta})$, and lies below the line $y=\theta x$. Then $p_{\gamma}^{\pm}$ is given by the horizontal displacements of the segments of $\Lambda_{\theta}^+$ connecting consecutive points.

    To define $p_{\gamma}^{-}$, let $\Lambda_{\theta}^-$ be the minimal convex lattice path starting in the origin with end point at $(m\theta,\ceil{m\theta})$, which lies above the line $y=\theta x$.
\end{itemize}
\begin{remark}
    Both partitions depends only on the class of $\theta\in \mathbb{R}/\mathbb{Z}$.
\end{remark}

The following theorem, proved in \cite{Hutchings3} show how the parition conditions of a $J$-holomorphic curve is determined by the orbit and its multiplicity.

\begin{theorem}{(See \cite{Hutchings3})}
    Suppose that $C$ is the nontrivial component of a $J$-holomorphic current that contributes to the ECH differential or the $U$ map. Then $p_{i}^+(C)=p_{\alpha_i}^+(m)$ and $p_{j}^-(C)=p_{\beta_j}^{-}(n_j)$.
 \end{theorem}

\section{The Klein bottle and closed geodesics}
The Klein Bottle is the quotient of $\mathbb{R}^2$ by the group G generated by two diffeomorphism $f(x,y)=(x+1,y)$ and $g(x,y)=(1-x,y+1/2)$. We will denote by $K=\mathbb{R}^2/G$. Since both maps $f$ and $g$ are isometries in $\mathbb{R}^2$ with the flat metric, we also have a flat metric in $K$ such that the natural projection $\pi:\mathbb{R}^2\rightarrow K$ is a local isometry.

\begin{remark}
    We observe that the vector field $\partial_y$ is globally well defined over $K$.
\end{remark}

\begin{remark}
    Sometimes we write $[x,y]=\pi(x,y)$.
\end{remark}

It is well know that if $(\Sigma,g)$ is a compact Riemannian surface without boundary then the Reeb vector field $R$ of the canonical contact form in $U^*(\Sigma)$ is dual to the geodesic vector field, in the sense that the isomorphism of vector bundles induces by $g$ has the property that the pushfoward of the geodesic vector field is the Reeb vector field. Therefore, we can study the Reeb orbits looking at closed geodesics. Using this idea we can find good coordinates to work in the case that $\Sigma=K$ the Klein bottle.

\subsection{Toric coordinates}\label{almosttoric}

We consider the unit cotangent bundle $U^*(K)$ where we endowed $K$ with its flat Riemannian metric. This is an orientable $3$-manifold, it is well known that $U^*(K)$ has a canonical contact form $\lambda_{std}$ which can describe in coordinates as follows. If $(x,y)$ are local coordinates to $K$, then in the induced coordinates for $T^*(K)$, which we denote by $(x,y,z,w)$, we have
$$
\lambda_{std}=zdx+wdy.
$$

Let $y:T^*(K)\rightarrow \mathbb{R}$ be the smooth function defined as follows. For $\xi \in T_{p}^*K\simeq T_{p}K$, take any $(x,y)\in \mathbb{R}^2$, such that $\pi(x,y)=p$. The map $d\pi_{(x,y)}:\mathbb{R}^2\rightarrow T_{p}K$ is an linear isometry, then $y(p,\xi)$ is the $y$-coordinate of any preimage of $\xi$ by the map $d\pi_{(x,y)}$. Follows from the definition of $K$ that this does not depend on the point $(x,y)$ which projects to $p$ by the natural projection $\pi$.

Consider the open set $U\subset U^*(K)$, defined as 
$$
U=y^{-1}((-1,1))\cap U^*(K),
$$
roughly these are covectors that are not "vertical". The proposition below describe the toric symmetry of the Reeb dynamics into the open set $U$.

\begin{proposition}\label{toricprop}
    There exist a strictly contactmorphism between $(U,\lambda_{std})$ and
    %%%%%%%%%%%%%%%%%%%%%%
    $$
    (I\times T^2,\lambda_{can}),
    $$ 
    where $I=(-\pi/2,\pi/2)$, and the contact form $\lambda_{can}=\cos(\theta)dx+\sin(\theta)dy$. 
\end{proposition}

\begin{proof}
    One define the map $F:I\times T^2\rightarrow U^*(K)$ by
    $$
    F(\theta,[x,y])=(\pi(x,y),d\pi_{(x,y)}(\cos(\theta),\sin(\theta))),
    $$ 
    where $\pi:\mathbb{R}^2\rightarrow K$ is the natural quotient map. The image is just $U$ and the pullback of the standard contact form is as claimed, and $F$ is cleary an embedding.
\end{proof}

The complement of $U$ is a disjoint union of two Klein bottles, namely 
$$
K_{+}=y^{-1}(1)\cap U^*(K) \hspace{0.6cm} and \hspace{0.6cm}
K_{-}=y^{-1}(-1)\cap U^*(K).
$$
%%%%%%%%%%%%%%%%%%%%%%%%
In fact, the map $K\rightarrow U^*(K)$ defined as $p\mapsto (p,\pm dy|_{p})$ are embeddings whose image is exactly $K_{\pm}$.

We can define the map $F$ more general as 
$$
\begin{array}{cccc}
F \ : & \! \mathbb{R}/2\pi \mathbb{Z}\times T^2 & \! \longrightarrow
& \! S^{*}(K) \\
& \! (\theta,[x,y])  & \! \longmapsto
& \! (\pi(x,y),d\pi_{(x,y)}(\cos(\theta),\sin(\theta)))
\end{array}
$$

This map is a quotient map. Therefore, the unit cotangent bundle of Klein bottle can be obtain from the three dimensional torus $T^3=\mathbb{R}/2\pi \mathbb{Z}\times T^2$, by the following identification
$$
{(\theta,[x,y])\sim (\pi-\theta,[1-x,y+1/2])}.
$$

\begin{remark}
Observe that the map $(\theta,[x,y])\in S^{*}(K)\mapsto [e^{2\pi y}]\in \mathbb{R}P^1\simeq S^1$, is a $T^2$-fibration over $S^1$, but the contact form is not invariant in this torus, and the Reeb vector field is not tangent to the fibers of this fibration. This implies that this contact manifold as a \textbf{Unit Cotangent Bundle of Klein Bottle} is different of this manifold with a contact form $a(t)dx+b(t)dy$ invariant under in each torus of the torus fibration. The following lemma give us some information about contact manifolds which has a contact form invariant under such fibrations.    
\end{remark}

\begin{lemma}\label{torictrivial}
    Let $Y^3$ a orientable, closed three-manifold. Suppose that $Y$ is a $T^2$-fibration over the circle $S^1$. Let $\lambda=a(t)dx+b(t)dy$ be a $T^2$-invariant contact form in $Y$. Then the contact structure is trivial.
\end{lemma}

\begin{proof}
    The vector field $X=\partial_t$ is in contact structure and has no zeros, therefore $\xi$ is trivial as a $\mathbb{C}$-vector bundle, and therefore as a symplectic vector bundle.
\end{proof}

\begin{remark}\label{remark1}
In this case we have proved that the unit cotangent bundle of Klein Bottle with the canonical contact structure is not of this type because our contact structure is not trivial.    
\end{remark}

\section{Reeb vector field and Reeb orbits}

By proposition \ref{toricprop}, our manifold has a open set which the Reeb Vector field is tangent to the fibers $U\rightarrow I$, whose the complement are the Klein bottles $K_{\pm}$. In the toric part
$$
(-\pi/2,\pi/2)\times T^2,
$$
the Reeb vector field is given by $R=\cos(\theta)dx+\sin(\theta)dy$, therefore for each $\theta\in (-\pi/2,\pi/2)$ fixed, $R$ is tangent to the torus $\{\theta\}\times T^2$ and its flow are "straight lines" with slope $\tan(\theta)$. Therefore, if $\tan(\theta)$ is rational, the flow of $R$ are all closed with homology class in the torus $(q,p)$, where $\tan(\theta)=p/q$ with gcd$(q,p)=1$. If $\tan(\theta)$ is irrational, there is no Reeb orbits in the torus $\{\theta\}\times T^2$.

However, in the sets $K_{\pm}$ the Reeb vector field is just the "vertical". Therefore, there are two $[0,1/2]$-family of Reeb orbits one in each Klein bottle $K_{\pm}$. 

\begin{remark}
Note that the $C^0$ limit of orbits when the parameter of the family tends to $0$ or $1/2$ are double covering of the Reeb orbits that pass through $[(0,0)]$ and $[(1/2,0)]$ in $K_{\pm}$, the picture is that they lived in a Mobius band neighborhood of each orbit in the extremes of the family converging to the central circle.  
\end{remark}

\section{Pertubations and Conley-Zehnder index}\label{pertubation}

First of all we show how to perturb the contact form $\lambda_{st}$ near the Klein bottle $K_+$, with out loss of generality.

\subsection{Pertubations of $\lambda_{st}$ in the ``Toric part"}

Fix a real number $L>0$, then there exists finitely many torus in $(-\pi/2,\pi/2)\times T^2$ which is foliated by Reeb orbits of action less than $L$. For each circle of orbits parametrized by $S^1$, we choose a Morse function $f:S^1\rightarrow \mathbb{R}$ with only two critical points, as in Bourgeois thesis \cite{Bourgeois}, we can extend this function to a neighborhood of this torus. Consider this pertubations along all tori with is foliated by Reeb orbits of action less than $L$, call $\overline{f}$ the function which realize all such pertubations and consider the following contact form
$$
\lambda_{\epsilon}=(1+\epsilon\overline{f})\lambda.
$$
The next lemma shows how the Reeb orbits behave after pertubations in terms of action.

\begin{lemma}\label{actionless}(See \cite{Bourgeois})
    For each such $L>0$, there exist $\epsilon=\epsilon(L)>0$ such that all Reeb orbits of $\lambda_{\epsilon}$ of action less than $L$ in $U^*(K)\setminus K_{+}\cup K_{-}$ are nondegenerate and correspond to critical points of the functions $f$ that parametrize the families of Reeb orbits of action less than $L$.
\end{lemma}

Thus each pertubed torus foliated by Reeb orbits of action less than $L$ splits into two nondegenerate Reeb orbit, as observed by Hutchings in \cite{Hutchings4}, one of them is elliptic and other is positive hyperbolic, it the pertubed torus is $\{\theta\}\times T^2$ with $\tan(\theta)=p/q$, we will denote these orbits by $e_{(q,p)}$ and $h_{(q,p)}$. 

The trivialization over Reeb orbits in the toric part $(-\pi/2,\pi/2)\times T^2$, which we call by $\tau$, is given by 
$$
\xi=span\{\partial_{\theta},-\sin(\theta)\partial_x+\cos(\theta)\partial_y\}.
$$

Given $k\in \mathbb{N}$, we can choose the Morse-Bott pertubation so that the rotation angle of $e_{(q,p)}$ satisfies $0<\phi<1/k$ and the linearized flow around the positive hyperbolic orbit do not rotate the eigenspace so much, so that by equation \ref{eqn:CZhyp} and \ref{eqn:CZelip}, we have
$$
CZ_{\tau}(e_{(q,p)}^k)=1\hspace{0.5cm} \textrm{and} \hspace{0.5cm } CZ_{\tau}(h_{(q,p)})=0.
$$

\subsection{Pertubations of $\lambda_{st}$ near $K_{\pm}$}

Let $f:\mathbb{R}\rightarrow \mathbb{R}_{\geq 0}$ be a smooth function with the following properties.

\begin{itemize}
    \item $f$ is $1/2$ periodic;
    \item In the interval $[0,1/2]$ $f$ has exactly one maximum at $x=1/4$ and two minimums at the extremes of the interval.
    \item The function $f$ has symmetry with respect to $x=1/4$, that is,  $f(x)=f(1/2-x)$, observe that this is equivalent to ask $f$ be even.
    \item There is no other critical points diferent from the maximum and minimum and the function $f$ is Morse.
\end{itemize}
%%%%%%%%%%%%%%%%%%%%%%%%%%%%
In this case, we can define a function in the Klein bottle, say $K_+$ by the formula
$$
\overline{f}([x,y],dy|_{[x,y]})=f(x).
$$

It is well defined by the properties of $f$ and is smooth. Finally we can extend $\overline{f}$ to a tubular neighborhood of $K_+$ inside $U^{*}(K)$ such that it is constant in the fiber of the normal bundle of $K_+$. To finish use a cut-off function depending on the distance to $K_+$, we can take a small radius in the fiber of the normal bundle such that the Reeb vector field $R_{st}$ is $C^1$-close to its value in $K_+$, which is $dy$.

Now consider the contact form 
$$
\lambda_{\epsilon}=(1+\epsilon \overline{f})\lambda_{st}.
$$

\begin{proposition}
    Given $L>0$, there is $\epsilon=\epsilon(L)$ such that every Reeb orbit of $\lambda_{\epsilon}$ that is close to $K_+$ and has action less than $L$ are over critical points of $\overline{f}$.
\end{proposition}

\begin{proof}
    The new Reeb vector field $R_{\epsilon}$ is given by
    $$
    R_{\epsilon}=\frac{1}{1+\epsilon \overline{f}}R_{\lambda}+\frac{\epsilon}{(1+\epsilon \overline{f})^2}X,
    $$
    where $X$ is the unique vector field with $X\in \xi=\ker \lambda$ such that 
    $$
    d\lambda(X,\cdot)|_{\xi}=d\overline{f}|_{\xi}(\cdot).
    $$
    In particular, if $p\in \textrm{Crit}(\overline{f})$, then $X(p)=0$ and $R_{\epsilon}(p)=\frac{1}{1+\epsilon\overline{f}(p)}$. This implies that the following curves are Reeb orbits
    \begin{eqnarray*}
    h_{(0,1)}^1(t)=([0,t],dy|_{[0,t]}); \quad 0\leq t\leq 1/2;\\
    h_{(0,1)}^2(t)=([1/2,t],dy|_{[1/2,t]}); \quad 0\leq t\leq 1/2;\\
    e_{(0,1)}(t)=([1/4,t],dy|_{[1/4,t]}); \quad 0\leq t\leq 1.
    \end{eqnarray*}
    Note that the action must change, but it is $\epsilon$-close to the lenght of the corresponding unperturbed Reeb orbit. If $\epsilon>0$ is sufficiently small, there is no Reeb orbits of action less than $L$ create other than these ones, this follows as in the Lemma \ref{actionless}.
\end{proof}

\subsection{Trivializations over $h_{(0,1)}^1$, $h_{(0,1)}^2$ and $e_{(0,1)}$}

We defined symplectic trivializations of the contact structure $\xi_{std}$ as follows. In $h_{(0,1)}^1(t)$ with $0\leq t\leq 1/2$, we consider
$$
\xi_{std}(h_{(0,1)}^1(t))=span\{\cos(2\pi t)\partial_{x}-\sin(2\pi t)\partial_{\theta}, \cos(2\pi t)\partial_{\theta}+\sin(2\pi t)\partial_{x}\}.
$$
Similarly, for $h_{(0,1)}^2(t)$ with $0\leq t\leq 1/2$,
$$
\xi_{std}(h_{(0,1)}^2(t))=span\{\cos(2\pi t)\partial_{x}-\sin(2\pi t)\partial_{\theta}, \cos(2\pi t)\partial_{\theta}+\sin(2\pi t)\partial_{x}\},
$$
%%%%%%%%%%%%%%%%%%%
where each vector field $\partial_x$ and $\partial_{\theta}$ is over the point $h_{(0,1)}^1(t)$ in the first case and over $h_{(0,1)}^2(t)$ in the second one. In the case of the Reeb orbit $e_{(0,1)}(t)$ we just consider as in the toric case
$$
\xi_{std}(e_{(0,1)}(t))=span\{\partial_{\theta},-\sin(\theta)\partial_x+\cos(\theta)\partial_y\}.
$$

The reason why we need take the trivialization with cossines and sines over the Reeb orbits $h_{(0,1)}^1(t)$ and $h_{(0,1)}^2(t)$ is that $\partial_x$ and $\partial_{\theta}$ are no more well defined vector field, because at final point $t=1/2$ they return back with a minus sign, we can overcome this difficult rotating a ``half" more to close up and obtain a well defined section of $\xi$ over these Reeb orbits.

\subsection{Computations of Conley-Zehnder index}

We will compute all for $h_{(0,1)}^1(t)$ and is similar to $h_{(0,1)}^2(t)$. Let us call
$$
e_1(t)=\cos(2\pi t)\partial_{x}-\sin(2\pi t)\partial_{\theta}\quad \textrm{and}\quad  e_2(t)=\cos(2\pi t)\partial_{\theta}+\sin(2\pi t)\partial_{x},
$$
the trivialization $\tau$ over the Reeb orbit $h_{(0,1)}^1(t)$. Let $\phi_{\epsilon}$ the flow of $R_{\epsilon}$ and $\phi$ the flow of $R_{\lambda}$. The linearized flow of $R_{\epsilon}$ of $h_{(0,1)}^1(t)$ restrict to the contact structure satisfies
$$
d\phi_{\epsilon}^t|_{\xi}=d\phi^t|_{\xi}\cdot \textrm{exp}(-\epsilon tJ_{0}Hess(\overline{f})|_{\xi}),
$$

where $J_0$ is the standard contact structure of $\mathbb{R}^2$. Therefore, we have to compute the linearized flow of the old flow and the Hessian of the function $\overline{f}$. The linearized old flow $d\phi^{t}(h_{(0,1)}^1(0))|_{\xi}$ along the contact structure is given by 
$$
d\phi^{t}(h_{1/0}(0))|_{\xi}=
\begin{bmatrix}
\cos(2\pi t) & -\sin(2\pi t)-t\cos(2\pi t) \\
\sin(2\pi t) & \cos(2\pi t)-t\sin(2\pi t) \\
\end{bmatrix}.
$$

A straightforward computation give that the Hessian of the function $\overline{f}$ in the contact structure over the orbit $h_{(0,1)}^1(t)$ is the following
$$
\textrm{Hess}(\overline{f})|_{\xi}=
\begin{bmatrix}
k\cos(2\pi t)^2 & k\sin(2\pi t)\cos(2\pi t) \\
k\sin(2\pi t)\cos(2\pi t) & k\sin(2\pi t)^2 
\end{bmatrix},
$$
where $k=f''(0)$. Recall by definition of $f$ that $k>0$. Therefore, the new linearized Reeb flow is given by
$$
d\phi_{\epsilon}^{1/2}|_{\xi}=
\begin{bmatrix}
    -1-\frac{k\epsilon}{4} & 1/2\\
    \frac{k\epsilon }{2} & -1
\end{bmatrix}
$$
its eigenvalues are given by
$$
\frac{-(2+\frac{k\epsilon}{4})\pm \sqrt{(2+\frac{k\epsilon}{4})^{2}-4}}{2}.
$$
Since $k>0$, both eigenvalues are negative and $h_{(0,1)}^1$ is a nondegenerate negative hyperbolic Reeb orbit.

\begin{proposition}
    The Conley-Zehnder index of $h_{(0,1)}^1$ in the trivialization $\tau$ is given by $CZ_{\tau}(h_{(0,1)}^1)=-1$.
\end{proposition}

\begin{proof}
    We just see that with the computations of the old linearized Reeb flow and Hessian of $\overline{f}$ we can compute the linearized Reeb flow over $h_{(0,1)}^1(t)$ in any time $t$, a straightforward computation show that this rotates a eigenvector by a angle of $-\pi$, using the equation \ref{eqn:CZhyp} we are done.
\end{proof}

The Reeb orbit $h_{(0,1)}^2$ also satifies $CZ_{\tau}(h_{(0,1)}^2)=-1$. In the case of Reeb orbit $e_{(0,1)}(t)$ its rotation angle is positive and very close to zero, therefore $CZ_{\tau}(e_{(0,1)})=1$.

In the klein bottle $K_{-}$ we consider the same with trivializations as in $K_{+}$ and obtain two negative hyperbolic Reeb orbits $h_{(0,-1)}^1$ and $h_{(0,-1)}^2$ both with Conley-Zehnder index equal to $-1$ and a elliptic Reeb orbit $e_{(0,-1)}$ which has Conley-Zehnder index equal to $1$.

\section{Singular Homology and Chern Class}

We start this section showing that the standard contact structure $\xi_{st}$ in the unit cotangent bundle of the flat Klein bottle has zero chern class as a complex bundle.

\begin{proposition}\label{chernclass}
    The first Chern class of the contact structure $c_1(\xi)$ is zero.
\end{proposition}

\begin{proof}
    Consider the vector field $\partial_{\theta}$ in $(-\pi/2+\epsilon,\pi/2-\epsilon)\times T^2\subset U^*(K)$ we can extend this vector field in the contact structure $\xi$ putting zeroes along the torsion generator in each Klein bottle $K_+$ and $K_-$, but these curves are homologous, therefore
    $$
    PD(c_1(\xi))=[\psi^{-1}(0)]=2[\gamma],
    $$
    where $\psi$ is the extension of the vector field $\partial_{\theta}$ and $\gamma$ is the torsion generator of the homology of $K_+$ (the projection of horizontal line in the Klein bottle). Since $2[\gamma]=0$, we conclude that $PD(c_1(\xi))=0$.
\end{proof}

\subsection{Homology classes of Reeb orbits}
In this section, we will compute the singular homology of our three manifold, with this information we can understand better our orbit sets when is fixed a homology class $\Gamma$ necessary to define ECH.

Using the Mayer-Vietoris sequence, we compute the singular homology of $U^*(K)$ and describe the homology classes of all Reeb orbits. The answer is $H_1(U^*(K),\mathbb{Z})=\mathbb{Z}\oplus \mathbb{Z}_2\oplus \mathbb{Z}_{2}$, and the Reeb orbits have homology class
\begin{table}[h]
    \centering
    \begin{tabular}{c|c}
    \textrm{Reeb orbit} & homology class \\
    \hline
     $h_{(0,1)}^1$    & $(1,\overline{0},\overline{0})$  \\
     \hline
     $h_{(0,1)}^2$   & $(1,\overline{0},\overline{1})$ \\
     \hline
     $e_{(0,1)}$   & $(2,\overline{0},\overline{0})$ \\
     \hline
     $h_{(0,-1)}^1$   & $(-1,\overline{1},\overline{0})$ \\
     \hline
     $h_{(0,-1)}^2$   & $(-1,\overline{1},\overline{1})$ \\
     \hline
     $e_{(0,-1)}$   & $(-2,\overline{0},\overline{0})$ \\
    \end{tabular}
    \caption{Homology of Reeb orbits in $K_+$ and $K_{-}$}
    \label{tab:my_label}
\end{table}

For Reeb orbits in the torus $tan(\theta)=p/q$, and therefore with homology $(q,p)$ in this torus, will have singular homology class in $U^*(K)$ given by $(2p,\overline{0},\overline{q})$.

\begin{remark}
    For any integer $n$, the overline on it $\overline{n}$ means its class in the group $\mathbb{Z}_2$.
\end{remark}

The next lemma is used to ensures in computations of the ECH index that the relative intersection number $Q_{\tau}(Z)$ does not depend on $Z\in H_2(U^*K,\alpha,\beta)$.

\begin{lemma}
    The embedded torus $\{\theta\}\times T^2\subset U^*K$ generates $H_2(U^*K,\mathbb{Z})$ for any $\theta\in (-\pi/2,\pi/2)$. More precisely, if $i:\{\theta\}\times T^2\rightarrow U^*(K)$, denotes the inclusion and $[\{\theta\}\times T^2]$ is the fundamental class of this torus, then $i_{*}[\{\theta\}\times T^2]$ generates $H_2(U^*(K),\mathbb{Z})\simeq \mathbb{Z}$.
\end{lemma}

\begin{proof}
    Follows by apply the long exact sequence in singular homology to the pair $\{\theta\}\times T^2\subset U^*K$.
\end{proof}

\subsection{Types of orbit sets}

Let $\alpha$ be an admissible orbit set with homology class zero. Then $\alpha$ has one of the following types.

\textbf{Type I.} The generator has only orbits from the toric coordinates given by Proposition \ref{toricprop}, together the eliptic orbits from $K_{\pm}$. We write 
$$
\alpha=e_{(0,-1)}^m\gamma_{(q_1,p_1)}\cdots \gamma_{(q_k,p_k)}e_{(0,1)}^n, 
$$
where $p_i/q_i\leq p_j/q_j$ for $i\leq j$. Subject to the conditions
$$
\sum_{i=1}^k p_i+n-m=0 \hspace{0.6cm} and \hspace{0.6cm} \sum_{i=1}^k
q_i\textrm{ is even}.
$$

\textbf{Type II.} The generator $\alpha$ has negative hyperbolic orbits from $K_{+}$, but not from $K_{-}$. For homological reasons, $\alpha$ can not have only one negative hyperbolic orbit. We write this orbit set in the multiplicative notations as follows.
%%%%%%%%%%%%
$$
\alpha=e_{(0,-1)}^m\gamma_{(q_1,p_1))}\cdots \gamma_{(q_k,p_k)}e_{(0,1)}^n h_{(0,1)}^1h_{(0,1)}^2, 
$$
%%%%%%%%%%%%%%%%%%%%%%%%%%%%
subject to the conditions
%%%%%%%%%%%%%%%%%%%%%%%%%%%%
$$
\sum_{i=1}^k p_i+n-m+1=0 \hspace{0.6cm} and \hspace{0.6cm} \sum_{i=1}^k
q_i\textrm{ is odd}.
$$

\textbf{Type III.} The generator $\alpha$ has negative hyperbolic orbits from $K_{-}$, but not from $K_{+}$. For homological reasons, $\alpha$ can not have only one negative hyperbolic orbit. Therefore we can write 

$$
\alpha=h_{(0,-1)}^1h_{(0,-1)}^2 e_{(0,-1)}^m\gamma_{(q_1,p_1)}\cdots \gamma_{(q_k,p_k)}e_{(0,1)}^n , 
$$
%%%%%%%%%%%%%%%%%%%%%%%%%%%%
subject to the conditions
%%%%%%%%%%%%%%%%%%%%%%%%%%%%
$$
\sum_{i=1}^k p_i+n-m-1=0 \hspace{0.6cm} and \hspace{0.6cm} \sum_{i=1}^k
q_i\textrm{ is odd}.$$
%%%
\textbf{Type IV.} In this case, $\alpha$ has all negative hyperbolic orbits, we write
$$
\alpha=h_{(0,-1)}^1h_{(0,-1)}^2 e_{(0,-1)}^m\gamma_{(q_1,p_1)}\cdots \gamma_{(q_k,p_k)}e_{(0,1)}^n h_{(0,1)}^1h_{(0,1)}^2,
$$
%%%%%%%%%%%%%%%%%%%%%%%%%%%%
subject to the conditions
%%%%%%%%%%%%%%%%%%%%%%%%%%%%
$$
\sum_{i=1}^k p_i+n-m=0 \hspace{0.4cm} \sum_{i=1}^k
n_iq_i\textrm{ is even}.$$

\subsection{Bijection with K-lattice paths}\label{bijectioncomb}

We can define now a map $ECC^L(U^*(K),\lambda_{\epsilon},0,J)$ to $C^{L}$, by taking the concatenation in the plane of the toric homology classes and in the case with negative hyperbolic arrows just take ``half arrows" as explained in the introduction \ref{intro: K-lattice}. This map is a bijection and we write to $\alpha$ an ECH generator, $P_{\alpha}$ the corresponding K-lattice path. The next section shows that this map preserves the $\mathbb{Z}$-grading in both homologies.

\begin{remark}
    We can extended this map in the following way, consider $\alpha$ be an admissible orbit set not necessary with total homology equal to zero, we associated $P_{\alpha}$ to the concatenation of its homology in $\mathbb{Z}^2$, possibly with half-arrows in the order given by the $y$-function as described in section \ref{almosttoric}.
\end{remark}

\section{ECH index computations}

\subsection{Admissible surface in $\mathbb{R}\times U^*(K)$}

In this section, we construct admissible surfaces $S$ necessary to compute the relative ECH index. We describe four operations to construct such admissible surfaces.

Let $\alpha$ and $\beta$ be orbit sets with same homology class, we will describe three operations which are useful in the construction of our surfaces.

\textbf{Operation I}. Consider in the open set $(-\pi/2,\pi/2)\times T^2\subset U^*(K)$ two curves $\gamma_1$ and $\gamma_2$ in the torus $\{\theta_0\}\times T^2$ that are projection of horizontal streight lines of $\mathbb{R}^2$ to $T^2$. It is clear that in this open set these two curves are not homologous to zero, but in our three manifold $U^*(K)$ they are, the Operation $I$ is the following, we move these two curves using the $\theta$-coordinate until $\theta=\pi/2$, we have now in the Klein bottle $K_{+}$ two curves such that in the Klein bottle there a cylinder wich closed the two cylinders $[\theta_0,\pi/2]\times \gamma_1$ and $[\theta_0,\pi/2]\times \gamma_2$ in a orientable way.

\begin{figure}[!htb]
\centering
\includegraphics[width=0.60\textwidth]{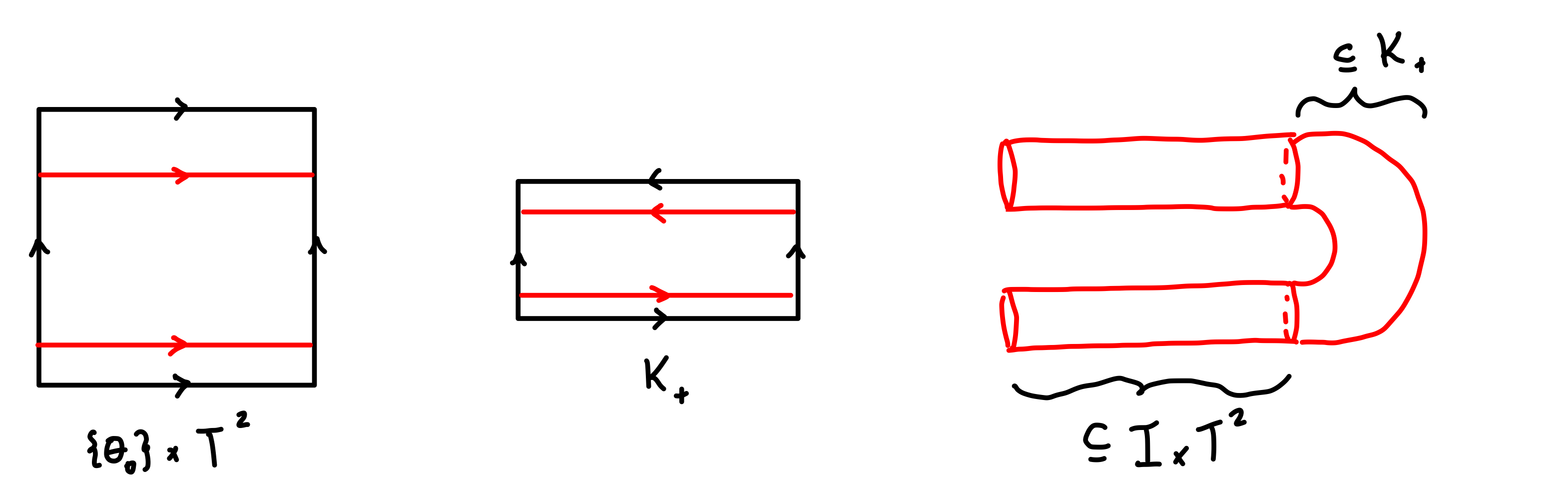}
\caption{\label{fig1} The operation I construction}
\end{figure}

\begin{remark}
    The same procedure can be done moving $\theta$ to $-\pi/2$, and we will still call this of Operation I.
\end{remark}

\textbf{Operation II}. When the orbits sets contain some negative hyperbolic orbits we in general will working with the new orbit sets which are obtain by double the multiplicity with the same Reeb orbits, this make easier all computations on the ECH index.

The operation II is the construction of a surface near the ends at negative hyperbolic orbits. We will describe this to the orbit $h_{1/0}^1$, but the construction is entirely analogous for the others negative hyperbolic orbits.

For $s$ descreasing from $1$, the slice $S\cap \{s\}\times Y$ near $h_{1/0}^1$ consist of the curve in $\{s\}\times\{-\pi/2+1-s\}\times T^2$ with homology $(0,-1)$ in the torus we move that curve as $s$ varies and when $s$ tends to $1$ we move this curve to $[0,1]\times K_{+}$ and continues moving $s$ until we colapse this curve as the double covering of the Reeb orbit $h_{\nicefrac{1}{0}}^1$, we draw below this procedure.
\begin{figure}[!htb]
\centering
\includegraphics[width=0.90\textwidth]{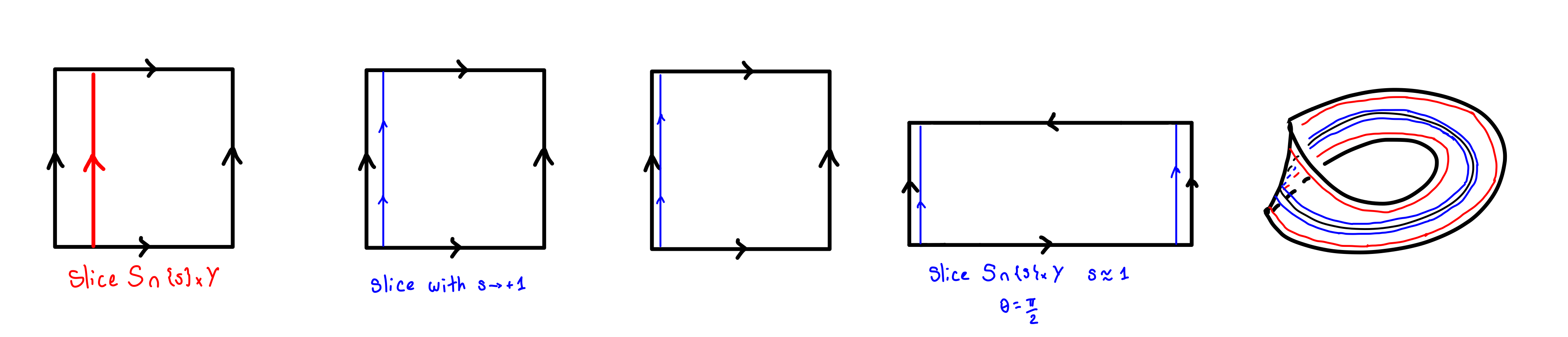}
\caption{\label{fig1} Surface S near negative hyperbolic orbits}
\end{figure}

\textbf{Operation III}. When we construct a surface with ends at $m$-coverings of a same Reeb orbit, we take $m$-copies of this same orbit, the operation III consist of moving the trivial cylinder over each copie of these orbit to the same orbit, in such way that in the end we have an admissible surface with ends at covering of the Reeb orbit. 

\begin{remark}
    Observe that the operation I does not add to the relative chern class, in general we will use the vector field $\partial_{\theta}$ plus a modification of this vector field in the surface nearby the negative hyperbolic orbits which are given by the operation II, to have a well defined section of the contact structure which has winding number zero with respect the trivialization $\tau$. We observe that, we can interpolate between $\partial_{\theta}$ and $-\partial_{\theta}$ in the surface of the operation I in the portion that it intersect the Klein bottle $K_{+}$ using the vector field $\partial_{x}$. This therefore create a smooth section of the contact structure along the surface $S$.
\end{remark}

\textbf{Operation IV}. (Surgeries) Let $\alpha=\{(\gamma_i,m_i)\}$ and $\beta=\{(\beta_j,n_j)\}$ be two toric orbit sets. Suppose that $[\alpha]=[\beta]\in H_1(U^*(K),\mathbb{Z})$, we will constuct a surface $S$ embedded (except possibly at the boundary) with $[S]\in H_2(U^*(K),\alpha,\beta)$ in $[0,1]\times I\times T^2$. Denote $\theta(\gamma_i)$ the $\theta$-coordinate which $\gamma_i\subset \{\theta(\gamma_i)\}\times T^2$. 

We start with half-cylinders 
$$
[1/2,1]\times \{\theta(\gamma_i)+\epsilon/k\}\times \pi_{T^2}(\gamma_i)
$$
for each $\gamma_i$ and $1\leq k\leq m_i$. Also we consider half-cylinders 
$$
[0,1/2]\times \{\theta(\beta_j)-\epsilon/k\}\times \pi_{T^2}(\beta_j)
$$
for each $\beta_j$ and $1\leq k\leq n_j$. The $\pm \epsilon/k$ are chosen to ensure that the half-cylinders are pairwise disjoint. We construct $S$ to be the union of such half-cylinders and a movie of curves $S(\theta)$ in $\{1/2\}\times I\times T^2$. For $\theta$ less than min$\{\theta(\gamma_i),\theta(\beta_j)\}_{i,j}$, $S(\theta)$ is empty.

Suppose that there only one half-cylinder between $\theta_{-}$ and $\theta_{+}$, say $[1/2]\times \{\theta(\gamma_i)\}\times \pi_{T^2}(\gamma_i)$, we obtain $S(\theta_{+})$ from $S(\theta_{-})$ as follows.
\begin{itemize}
    \item If $S(\theta_{-})$ and $\gamma_i$ are parallel, then we add or remove a component to/from $S(\theta_{+})$ depending if we can closed up the surface with a cylinder in $\{1/2\}\times \{\theta(\gamma_i)\}\times T^2$ in a orientable fashion way or not.
    \item If they are not parallel, we perform a "surgerie": the boundary of the half cylinder at $\{1/2\}\times \{\theta(\gamma_i)\}\times T^2$ is $\{\theta(\gamma_i)\}\times \pi_{T^2}(-\gamma_i)$. We resolve each intersection of $S(\theta_{-})$ and $-\gamma_i$ and linearly interpolate $S(\theta)$ between $\theta_{\gamma_i}$ and $\theta_{+}$ until achived the straight line in the same homology of each closed component of the curve after resolve all intersections.
\end{itemize}

The case of a half-cylinder in $\beta_j$ is similar. Thus we construct surface $S$ with boundaries in $\{1\}\times \{\theta(\gamma_i)+\epsilon/k\}$ and $\{0\}\times \{\theta(\beta_j)-\epsilon/k\}\times \pi_{T^2}(\beta_j)$ and in a even number of horizontal circles with homology $(1,0)\in \mathbb{Z}^2$ in some torus $\{1/2\}\times \{\theta_{\textrm{max}}\}\times T^2$. As a final step we deform this surface so that it has boundaries in $\{1\}\times \gamma_i^{m_i}$ and $\{0\}\times \beta_j^{n_j}$ and in a bunch of "horizontal" circles we apply the \textbf{Operation I} because there are a even number of such circles.

\begin{remark}
    Of course, not every orbit set $\alpha$ with zero homology is toric, but the idea in computations of the ECH index is take the orbit set $2\alpha$ which is defined with the same orbits, but the multiplicities are multiplied by $2$. In this case, we can apply this constructions, once we can use the \textbf{Operation II}.
\end{remark}

\subsection{The relative first Chern class}\label{prop: relchern}

The following proposition give us a formula for the relative first Chern class.

\begin{proposition}\label{relchern}
    Let $\alpha$ and $\beta$ be admissible orbit sets such that $[\alpha]=[\beta]$. For any $Z\in H_2(Y,\alpha,\beta)$
    $$
    c_{\tau}(Z)=\frac{n_{\alpha}-n_{\beta}}{2}
    $$
    where $n_{\alpha}$ is the number of negative hyperbolic orbits at $\alpha$.
\end{proposition}

\begin{proof}
    We construct a admissible surface $S$, such that $[S]\in H_2(Y,2\alpha,2\beta)$. Using the operation II we know the surface $S$ nearby each double covering of the negative hyperbolic orbit.
    Consider the vector field $\partial_{\theta}$ given by the $\theta$ coordinate in the toric part, we consider this vector field along $S$ except in a small neighborhood of each neagtive hyperbolic end. We observe the following, the winding number with respect the trivialization of $\partial_{\theta}$ in a slice of $S$ very close to the negative hyperbolic orbits on $\alpha$ is $-1$, therefore we need add $+1$ zeroes for each end of $\alpha$ at negative hyperbolic orbits to have a extension of $\partial_{\theta}$ to all positive end of $S$ such that the vector field in the end has winding number zero with respect $\tau$.
    
    In the same way, the vector field $\partial_{\theta}$ has winding number $+1$ at the negative ends of $S$ in negative hyperbolic orbits, therefore we need add $-1$ zeros to each such end, and from this we obtain the formula in the proposition.
\end{proof}

\begin{remark}
    The equation \ref{eq: difchern} implies that the relative first Chern class does not depend on the relative homology class $Z\in H_2(U^*(K),\alpha,\beta)$.
\end{remark}

\begin{remark}
    The equality $[\alpha]=[\beta]$ implies necessary that $n_{\alpha}-n_{\beta}$ is an even number.
\end{remark}

\subsection{Fredholm index formula}\label{Fredholm1}

Suppose that $\alpha$ and $\beta$ are orbit sets with same homology $[\alpha]=[\beta]$. Let $C$ be a regular pseudo-holomorphic curve in $M^J(\alpha,\beta)$. By definition
$$
\textrm{ind}(C)=-\chi(C)+2c_{\tau}([C])+CZ_{\tau}^{ind}(\alpha,\beta).
$$
By proposition \ref{relchern} and Conley-Zehnder computations from the section \ref{pertubation}, we obtain an explicit formula to the Fredholm index of a $J$-holomorphic curve with finite energy in the symplectization of $U^*(K)$ with a perturbed contact form $\lambda_{\epsilon(L)}$.

\begin{equation}\label{eq: fredind}
\textrm{ind}(C)=2(g(C)+e(\alpha)-1)+n_{\alpha}+n_{\beta}+h(\alpha)+h(\beta).    
\end{equation}

\subsection{ECH index for $U^*(K)$}

We will compute the ECH index of each type of admissible orbit set. A first observation is that if $\alpha$ and $\beta$ are nullhomologous ECH generators by the Proposition \ref{prop: ECHPRO} the ECH index $I(\alpha,\beta,Z)$ does not depend on the relative homology class $Z$. We write $I(\alpha):=I(\alpha,\emptyset)$. Moreover by additive property of the ECH index, we have $I(\alpha,\beta)=I(\alpha)-I(\beta)$. We will compute $I(\alpha,\emptyset)$ for $\alpha$ orbit set of any type. 
\smallskip

\textbf{ECH index of type I orbit sets}. Let $\alpha=\gamma_{(q_1,p_1))}\cdots \gamma_{(q_k,p_k)}$ be a orbit set of \textbf{type I} such that we order the factors so that $p_i/q_i\leq p_j/q_j$ if $i<j$. Suppose that $[\alpha]=0\in H_1(U^*(K),\mathbb{Z})$. We will compute its grading $I(\alpha,\emptyset)$.

Let $S$ be admissible surface $S\subset [0,1]\times U^*(K)$ with $[S]\in H_2(U^*(K),\alpha,\emptyset)$ as follows. We apply the \textbf{Operation IV} together with the \textbf{Operation III} to obtain a surface that has boundary in $\alpha$ and boundary in a disjoint union of simple curves in the torus $\{1/2\}\times \{\theta_0\}\times T^2$ for some $\pi/2>\theta_0>\theta_{max}(\alpha)$, where
$$
\theta_{max}(\alpha)=\max\{\theta\mid \textrm{there is a Reeb orbit $\gamma$ in $\alpha$ that is in the torus }\{\theta\}\times T^2\}.
$$
Now we have a bunch of parallel simple curves with  homology in the torus $\pm (q,p)$ such that the total homology in $U^*(K)$ is zero. After "canceling" pairs with homology with a change o sign, we have a family of disjoint parallel curves with homology, say, $(q,p)$. Let $n$ be the number of such curves, then by homology
$$
2np=0\quad \textrm{and}\quad nq\equiv 0\pmod{2}.
$$
Therefore, or $n=0$ and we are done the construction of the surface, or $p=0$ and therefore each curve is "horizontal" with homology $(q,0)$, but because each curve is simple we would have $q=\pm 1$, in particular $n$ is an even number. Therefore, there is a even number of "horizontal" curves with same homology within in the torus, to finish we apply the \textbf{operation I} to finish the construction of the surface $S$.

By Proposition \ref{prop: relchern} we already know compute the relative chern class, in this case is zero because there are no negative hyperbolic orbits in $\alpha$. We can assume that on the complement of the surgery points, the circles used to construct $S$ as a movies of curves never point in the $y$-direction in the each torus.

To compute the term $Q_{\tau}([S])$ we perturb $S$ as follows. Consider the section of the normal bundle of $S$ given by $\psi=\pi_{NS}(\partial_{y})$, defined in $S\cap \{\theta<\pi/2-\epsilon\}$, where $\pi/2-\epsilon>\theta_0$. The portion of $S$ with $\theta>\pi/2-\epsilon$ we perturb the parallel circles in the $y$-direction in the torus and applying the \textbf{operation I} we construct the surface $S'$.   

First observe that $S$ and $S'$ have no link number between then because $\psi$ is $\tau$-trivial, therefore to compute $Q_{\tau}$ we just need compute the zeros of $\psi$ with sign, but $\psi$ vanishes only at the surgery points, as in \cite{Hutchings5} all zeros are positive, thus

$$
Q_{\tau}([S])=-\sum_{i<j}det
\begin{bmatrix}
p_i & p_j \\
q_i & q_j \\
\end{bmatrix}.
$$

In this case, we observe that $Q_{\tau}([S])$ coincides with $2Area(P_{\alpha})$, this follows from the fact that this determinant counts a area of a edge compared with other, because this counts is make twice we obtain the formula. Therefore we have $I(\alpha)=I(P_{\alpha})$ for each $\alpha$ of type I.
\smallskip

\textbf{ECH index of type II orbit set}. Let $\alpha$ be a ECH generator of type II, that is, of the form $\alpha=\alpha_1 h_{(0,1)}^1 h_{(0,1)}^2$, where $\alpha_1=\gamma_{(q_1,p_1)}\cdots \gamma_{(q_k,p_k)}$ such that $i<j$ implies $p_i/q_i\leq p_j/q_j$. As in the remark \ref{rem: quadratic} we consider the orbit set $2\alpha:=\alpha\alpha$. Therefore each negative hyperbolic orbit in $2\alpha$ has multiplicity two. We apply the \textbf{Operation II} together with the construction of cylinders and apply the \textbf{Operation IV} doing surgeries to obtain a surface which has boundary in $2\alpha$ and as in the case of toric orbit sets, a bunch of "horizontal" curves in a torus $\theta_0$ very close to $\pi/2$ and we can close the surface in a orientable manner as in the toric case.

By Proposition \ref{prop: relchern}, for any $Z\in H_2(U^*(K),\alpha,\emptyset)$, we have $c_{\tau}(Z)=1$. To compute the relative intersection pairing we use the fact that $Q_{\tau}$ is quadratic in the sense of Remark \ref{rem: quadratic}. First observe that $Q_{\tau}(Z)$ does not depend on because any tori $\{\theta\}\times T^2\hookrightarrow U^*(K)$ is the generator of $H_2(U^*(K),\mathbb{Z})$, therefore for any another relative class $Z'$ we have
$$
Q_{\tau}(Z)-Q_{\tau}(Z')=([Z]-[Z'])\cdot [\alpha],
$$
where ``$\cdot$" denotes the usual intersection number of an oriented manifold, but all Reeb orbits can always be perturbed to have zero intersection of this tori. Therefore, $Q_{\tau}(Z)=Q_{\tau}(Z')$.

Fix $Z\in H_2(U^*(K),\mathbb{Z})$, then $2Z\in H_2(U^*(K),\mathbb{Z})$, therefore
$$
Q_{\tau}(Z)=\frac{Q_{\tau}(2Z)}{4}=\frac{Q_{\tau}([S])}{4}.
$$
Now $Q_{\tau}([S])=\# \textrm{int}(S)\cap \textrm{int}(S')-lk_{\tau}(S,S')$. The intersection part is compute as follows, as in the case o toric orbit sets, we use the section $\psi=\pi_{NS}(\partial_y)$ to perturb $S$ outside from the portion of $S$ in $\theta>\pi/2-\epsilon$, where in this portion we perturb just taking parallel circles to the curves on slices and interpolate this pertubations which do not add any intersection point other than the zeros of $\psi$ which are in the surgery points. Because we take the double of multiplicities the number of intersection is multiply by four, therefore
$$
\#\textrm{int}(S)\cap \textrm{int}(S')=4(-\sum_{i<j}det
\begin{bmatrix}
p_i & p_j \\
q_i & q_j \\
\end{bmatrix}
+\sum_{i=1}^k q_i),
$$
%%%%%%
but now we have non-trivial linking number between the braids in near the ends of these two surfaces, which can compute as follows. The trivialization of the negative hyperbolic orbits $h_{(0,1)}^1$ and $h_{(0,1)}^2$, give us the slice $S\cap \{s\}\times Y$ and $S'\cap \{s\}\times Y$, near each negative hyperbolic orbit two curves which are boundary of a mobius band with central circle exactly the negative hyperbolic orbit. Therefore, we can compute this link number, computing the self link number of the braid which is the boundary of this mobius band aroung the central circle as in the Figure  \ref{fig: braids}.
\begin{figure}[!htb]
\centering
\includegraphics[width=0.30\textwidth]{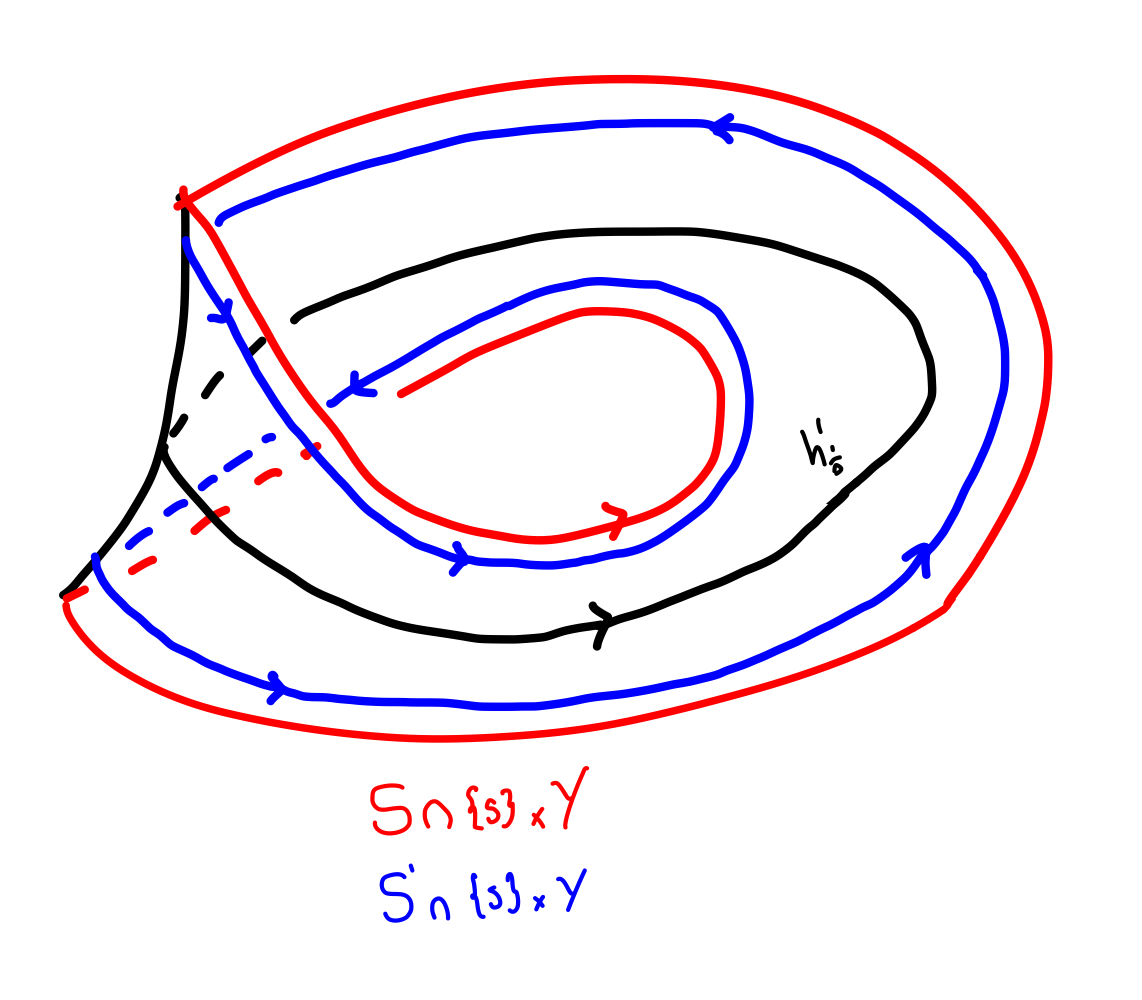}
\caption{\label{fig: braids} Braids aroung negative hyperbolic orbits}
\end{figure}

Thus we have two braids, each one around $h_{(0,1)}^1$ or $h_{(0,1)}^2$ which has self-link number $-2$. Therefore, $lk_{\tau}(S,S')=-4$, which implies that
$$
Q_{\tau}([S])=4(-\sum_{i<j}det
\begin{bmatrix}
p_i & p_j \\
q_i & q_j \\
\end{bmatrix}
+\sum_{i=1}^k q_i)+4,
$$
and consequently
$$
Q_{\tau}(Z)=-\sum_{i<j}det
\begin{bmatrix}
p_i & p_j \\
q_i & q_j \\
\end{bmatrix}
+\sum_{i=1}^k q_i+1,
$$
%%%%%%
for any $Z\in H_2(Y,\alpha,\emptyset)$. We compute the ECH index using these formulas and obtain
$$
I(\alpha)=-\sum_{i<j}det
\begin{bmatrix}
p_i & p_j \\
q_i & q_j \\
\end{bmatrix}
+\sum_{i=1}^k q_i+m(\alpha_1)-h(\alpha_1)
$$

As in the case of toric orbit, the first and second sum is exactly $2$Area$(P_{\alpha})$. Therefore, we obtain the equality $I(\alpha)=I(P_{\alpha})$.

\begin{remark}
    Although we derive the same formula as in the toric case, we do so through a different approach. In this case, the $+1$ of the $Q_{\tau}$ term and the $+1$ of the $c_{\tau}$ term cancel out the $-2$ from Conley-Zehnder index of both orbits $h_{(0,1)}^1$ and $h_{(0,1)}^2$.
\end{remark}

\begin{remark}
    The reason that we can compute the link number from the picture of a Mobius band neighborhood of each negative hyperbolic orbit is the following. The trivialization $\tau$ generates an embedding of a solid torus neighborhood of each negative hyperbolic orbit into $\mathbb{R}^3$, but using "coordinates" $(x,y)$ in the neighborhood of each negative hyperbolic orbit give us also an embedding of this Mobius band into $\mathbb{R}^3$, the key fact is that both embedding are isotopic because the $\tau$ trivialization just rotate "half" more the old embedding.
\end{remark}

\textbf{ECH index of type III orbit set}. Let $\alpha=h_{(0,-1)}^1\cdot h_{(0,-1)}^2\cdot \alpha_1$ be a generatorof type III. We can compute its ECH index as in the case of a generator of type II and obtain the formula
$$
I(\alpha)=2Area(P(\alpha))+l(\alpha_1)-h(\alpha_1),
$$
which implies equality between the ECH-index and the combinatorial index in this case.

\textbf{ECH index of type IV orbit set.} Now we compute that the ECH granding of $\alpha=h_{(0,-1)}^1 h_{(0,-1)}^2\alpha_1 h_{(0,1)}^1 h_{(0,1)}^2$, where $\alpha_1$ is a toric orbit set with homology zero. As in the last cases, we take the double orbit $2\alpha$ and construct the surfaces $S$ representing a relative homology class in $H_2(U^*(K),2\alpha,\emptyset)$. We use the equation from the Remark \ref{rem: quadratic} and the pertubations of $S$ as in the last cases to show that
$$
Q_{\tau}(Z)=2\textrm{Area}(P_{\alpha})+2,
$$
the intersection terms are given by the sum of determinant as usual, an each braid around the negative hyperbolic orbits given $-2$, by the formula summing up this number we obtain $8$, but by the Remark \ref{rem: quadratic} we have to divide this by $4$, and that's why appears $+2$ in the $Q_{\tau}$ formula in this case.

The relative chern class is computed from the Proposition \ref{prop: relchern}, we obtain $c_{\tau}(Z)=2$, because each negative hyperbolic orbit has Conley-Zehnder index $-1$. Therefore, we derive the same formula in this case $I(\alpha)=I(P_{\alpha})$, and this finish the proof that the bijection to the combinatorial complex preserves grading.

We can resume all formulas in a unique formula given by the next proposition.

\begin{proposition}
    Let $\alpha$ be an ECH generator with $[\alpha]=0$. Then
    \begin{equation}\label{eq: combECHindex}
    I(\alpha)=2(L(P_{\alpha})-1)-n_{\alpha}/2-x(P_{\alpha})-h(\alpha)
    \end{equation}
    where $L(P_{\alpha})$ is the number of lattice point in the region enclosed by $P_{\alpha}$ and the $x$-axis.
\end{proposition}

\begin{proof}
    It is a automatic consequence of the Pick formula for the area of lattice polygons in the plane.
\end{proof}

\section{$J$-holomorphic curves}\label{sec: Jhol}

\subsection{Automatic transversality for almost complex structures}

The standard almost complex structure $J_{st}$ in $\mathbb{R}\times T^3$ given by
$$
J(\partial_s)=\cos(\theta)\partial_x+\sin(\theta)\partial_y,\hspace{0.5cm} J(\partial_{\theta})=-\sin(\theta)\partial_x+\cos(\theta)\partial_y,
$$
%%%%%%%%%%%%%%%%%
descends to $\mathbb{R}\times U^*(K)$, and we continue to denote it by $J_{st}$. Moreover, this almost complex structure is invariant under the map 
$$
R(s,\theta,[x,y])=(s,-\theta,[x,-y]).
$$ 
Such almost complex structure are highly restrictive to achieving transversality. However, an argument with automatic transversality provides useful information about the moduli space of certain pseudoholomorphic curves that appear in the differential. To transition to a generic pertubation of $J_{st}$ while preserving this information, we must first establish a stability result (see \ref{Prop: stabcurve}). This allows us to utilize data obtained in the "R-symmetric" setting (see equation \ref{bijmoduli}). 

We now briefly recall some key facts about automatic transversality. For more general results see \cite{Wendl}.

Let $\lambda$ a nondegeneate contact form on a closed three-manifold $Y$ and $J$ be a $\lambda$-compatible almost complex structure in $\mathbb{R}\times Y$. If $u:(\Sigma,j)\rightarrow (\mathbb{R}\times Y, J)$ is a $J$-holomorphic immersion, with normal bundle $N$, then it has a deformation operator
$$
D_u:W^{1,2}(\Sigma,N)\rightarrow L^2(\Sigma, T^{0,1}\Sigma\otimes_{\mathbb{C}}N).
$$
The moduli space of $J$-holomorphic curves near $u$ is cut out transversely if $D_u$ is surjective. We use the following result of automatic transversality. Let $h_{+}(u)$ denote the number of ends of a $J$-holomorphic curve $u$ at positive hyperbolic orbits, including even covers of negative hyperbolic orbits.

\begin{proposition}\label{automatic}
    Let $\lambda$ be a nondegenerate contact form on a closed three-manifold $Y$ and $J$ be a $\lambda$-compatible almost complex structure in $\mathbb{R}\times Y$. Let $u$ be a $J$-holomorphic curve which is a immersion. If the domain $\Sigma$ of $u$ is connected with genus $g$, and if
    $$
    2g-2+h_+(u)<ind(u),
    $$
    Then the linearized operator $D_u$ is surjective, without any genericity assumption on $J$.
\end{proposition}
\begin{proof}
    See Lemma 4.1 in \cite{Jo Nelson}.
\end{proof}

Now we can proof the following, which is enough to describe the ECH differential in Morse-Bott pertubations of $(U^*(K),\lambda)$.

\begin{proposition}\label{Prop: curv0}
    Let $J$ be a pertubation of $J_{st}$ in $\mathbb{R}\times U^*(K)$ which is symmetric with respect to the map $R$. Then any $J$-holomorphic punctured sphere with Fredholm index $1$ with at most two positive hyperbolic ends is cut out transversely. 
\end{proposition}
\begin{proof}
    This follows directly from the Proposition $\ref{automatic}$.
\end{proof}

In our case, we will use certain pertubations of $J_{st}$ and of the contact form as in the Section \ref{pertubation}. The proposition above does not say that an almost complex structure $J$ with $R$ symmetry is regular to $J$-holomorphic curves with three positive hyperbolic ends, which arise in the case of embedded Fredholm index one curves (see Section \ref{Fredholm1}). However we make the following observation: for $J_{st}$, there is a bijection of the moduli spaces
\begin{equation}\label{bijmoduli}
\mathcal{M}_{J_{st}}(\alpha,\beta)\simeq \mathcal{M}_{J_{st}}(R(\alpha),R(\beta)),    
\end{equation}
%%%%%%%
where $R(\alpha)$ is the orbit set when we apply in each Reeb orbit in $\alpha$ the map $R$ which is a contactmorphism for $\lambda_{st}$ and conserves the multiplicities. In fact, this holds for any $J$ which is invariant and the map $id_{\mathbb{R}}\times R$ in the symplectization of $U^*(K)$. We show that the equation \ref{bijmoduli} holds for more general $J$ than only with "R-symmetry".

\begin{proposition}\label{Prop: stabcurve}
    Let $J_0$ be a generic pertubation of $J_{st}$ and $\alpha$ and $\beta$ two orbit sets for $\lambda_{\epsilon}$ such that $\alpha$ has at most two positive hyperbolic orbits and $\beta$ only elliptic orbits or negative hyperbolic with multiplicity one and suppose that there are no broken pseuholomorphic curves between $\alpha$ and $\beta$ for any $J$. Then there is a bijection of the Fredholm index one moduli spaces
    $$
    \mathcal{M}^{J_{st}}_{1}(\alpha,\beta)\simeq \mathcal{M}^{J_0}_{1}(\alpha,\beta),
    $$
    In particular, we conserve the $R$-symmetry bijection \ref{bijmoduli} for a generic $J$ near $J_{st}$ which is not necessary symmetric with respect $R$.
\end{proposition}

\begin{proof}
    Follows from automatic transversality that for any $J$ compatible almost complex structure $\mathcal{M}^{J}_{1}(\alpha,\beta)$ is cut out transversely.  It is well know that the space of compatible almost complex structure is contractible. We can take a smooth path $J_t$ with $0\leq t\leq 1$ connecting $J_{st}$ to $J_0$. In this case we note that the moduli spaces
    $$
    W=\{(t,C)\mid C\in \mathcal{M}^{J_t}_{1}(\alpha,\beta)/\mathbb{R}\},
    $$
    %%%%%%%
    is a compact cobordism between the moduli spaces, because there are never any broken pseudoholomorphic curves between $\alpha$ and $\beta$ for any $J$. So the moduli spaces $\mathcal{M}^{J_{st}}_1(\alpha,\beta)/\mathbb{R}$ and $\mathcal{M}^{J_0}_1(\alpha,\beta)/\mathbb{R}$ are diffeomorphic.
\end{proof}

\begin{remark}\label{remark: nodoublerounding}
    This proposition is crucial because we require the bijection to maintain that "$R$-symmetry" and to exclude the case of "double rounding". By following the approach in the appendix of \cite{Hutchings1} we can take a
    generic pertubation of $J_{st}$ to excluded the case of ``double rounding" at same time that we preseve the "$R$-symmetry" of the Moduli spaces of $J$-holomorphic curves (in the sense of equation \ref{bijmoduli}) that will be counted in the ECH differential. The bijection \ref{bijmoduli} will be important in low index computation of the ECH differential as we will see in the section \ref{sec: lowind}.
\end{remark}

\begin{remark}
    We also observe when the $C$, $D$ operations or interior rounding are applied, the orbit sets formed by the positive end and negative end of the embedded part of the ECH index one current satisfy the hypotesis of Proposition \ref{Prop: stabcurve}. Therefore, in all of these cases, the bijection \ref{bijmoduli} holds.
\end{remark}

\begin{lemma}\label{lem: oneorbit}
    Suppose that $\gamma$ is a toric Reeb orbit, and let $\beta$ be an orbit set such that $n[\gamma]=[\beta]$ for some natural number $n\geq 1$. If there exists a $J$-holomorphic curve connecting $\{(\gamma,n)\}$ and $\beta$ in $(-\pi/2,\pi/2)\times T^2$. Then $\beta$ has only one orbit in $\theta(\gamma)$-torus, and $C$ is a Morse-Bott cylinder. 
\end{lemma}

\begin{proof}
    Let $[\gamma]=(q,p)$ and $\beta=\{(\beta_i,m_i)\}$, with $[\beta_{i}]=(q_i,p_i)$. Then, we have
    $$
    (nq,np)=(\sum_{i}m_iq_i,\sum_i m_ip_i),
    $$
    but if $C$ has most of two negative ends, we have $nA(\gamma)>\sum_{i}m_iA(\gamma_i)$, that is,
    $$
    n\sqrt{q^2+p^2}>\sum_{i}m_i\sqrt{q_i^2+p_i^2}.
    $$
    However, using the homology restriction and this inequality, we obtain
    $$
    \sqrt{(\sum_i m_iq_i)^2+(\sum_i m_ip_i)^2}>\sum_i m_i\sqrt{q_i^2+p_i^2}.
    $$
    Now we apply the triangular inequality in the vectors $v_i=(m_iq_i,m_ip_i)$, we get
    \begin{eqnarray*}
    \sqrt{(\sum_i m_iq_i)^2+(\sum_i m_ip_i)^2}&=& ||\sum_i v_i||\\
    &\leq & \sum_i||v_i||\\
    &=& \sum_i m_i \sqrt{q_i^2+p_i^2}.
    \end{eqnarray*}
    Leading to a contradiction. Therefore, $C$ has at most one negative end, but this implies that $\beta=\{(\beta_1,m)\}$ where $\beta_1$ is the other Reeb orbit with action less than $L$ in the torus $\theta(\gamma)$. In this case the unique possibility is that $C$ is a Morse-Bott cylinder.
\end{proof}

\subsection{Local energy inequality and positivity}

In this section we obtain constrains to the existence of pseudoholomorphic curves in our manifold inspired in ideas from Hutchings \cite{Hutchings4}.

Let $C$ a $J$-holomorphic curve in $\mathbb{R}\times Y$. If $C$ is transverse to $S_{\theta}=\mathbb{R}\times \{\theta_0\}\times T^2$, (which occurs for generic $\theta$), and $C$ has no ends at $\theta_0$-torus, then $C_{\theta_0}=C\cap S_{\theta_0}$ is a (possibly empty) compact $1$-dimensional submanifold of $C$. When non-empty, it is the boundary of the subdomain given by $C\cap \{\theta\leq \theta_0\}$. Thus, the orientation of $C$ induces an orientation on $C_{\theta_0}$. Our convention is to use the ''outer normal first'' convention. Consequently, we obtain a well-defined class $[C_{\theta_0}]\in H_1(S_{\theta_0})=\mathbb{Z}^2$, which we refer to as the slice class.

\begin{lemma}\label{lemma: localenergy}
    Let $C$ be a $J$-holomorphic curve and $\theta_0\in (-\frac{\pi}{2},\frac{\pi}{2})$ such that $C$ intersects $\mathbb{R}\times \{\theta_0\}\times T^2$ transversally. If $[C_{\theta_0}]=(q,p)$, then
    $$
    det\begin{bmatrix}
    \cos(\theta) & q \\
    \sin(\theta) &  p\\
    \end{bmatrix}\geq 0
    $$
%%%%   
    with equality only if $C$ does not intersect $\mathbb{R}\times \{\theta_0\}\times T^2$.
\end{lemma}

\begin{proof}
    By definition, $d\lambda$ is pointwise nonnegative on $C$, with equality only if the tangent space is the span of $\partial_s$ and $R$. Consider the function $\rho$ defined for $\theta>\theta_0$, close to $\theta_0$, by the formula
    $$
    \rho: \theta\mapsto \int_{C\cap \mathbb{R}\times [\theta_0,\theta] \times T^2}d\lambda.
    $$
%%%%
    This function is non-decreasing as $\theta$ increases. Hence, its derivative with respect to $\theta$ is nonnegative. By Stokes' theorem, we have
    $$
    \rho(\theta)=\int_{C_{\theta}}\lambda-\int_{C_{\theta_0}}\lambda.
    $$
%%%%
    For $\theta$ sufficiently close to $\theta_0$, we have the equality in slice classes $[C_{\theta}]=[C_{\theta_0}]$. Using that $\lambda=\cos(\theta)dx+\sin(\theta)dy$, we obtain
    $$
    \rho'(\theta)=-\sin(\theta)q+\cos(\theta)p.
    $$
%%%%
    Therefore, we must have
    $$
    \det\begin{bmatrix}
    \cos(\theta) & q \\
    \sin(\theta) &  p\\
    \end{bmatrix}\geq 0
    $$
%%%%
    which extends by continuity to $\theta_0$. Suppose now that we have equality, an argument with the second derivative shows that the only possibility if that $C_{\theta_0}=\emptyset$.
\end{proof}

\subsection{Path can not cross}

In this section we compute the slice class $[C_{\theta}]$ in terms of the combinatorics of its positive and negative asymptotic ends. Let $\alpha$ be an admissible orbit set, consider $n_{\alpha}=n_{\alpha}^-+n_{\alpha}^+$, where $n_{\alpha}^-$ and $n_{\alpha}^+$ is the number of negative hyperbolic orbits arising from $K_-$ and $K_+$ of $\alpha$ respectively. We assume that both $n_{\alpha}^-$ and $n_{\alpha}^+$ are even numbers. For any $\theta_0\in (-\pi/2,\pi/2)$, denote $P_{\alpha}^{<\theta_0}$ the vector obtained by summing all of vectors in the $K$-lattice path $P_{\alpha}$ which correspond to \textbf{toric Reeb orbits} arising from $\theta<\theta_0$.

\begin{proposition}
    Let $\alpha$ and $\beta$ be admissible orbit set. Suppose that $C\in \mathcal{M}^J(\alpha,\beta)$. Then there exist a nonnegative integer $r_{-}$ (which does not depend on $\theta_0$), such that
    \begin{equation}\label{eq: sliceclasscomb}
    [C_{\theta_0}]=P_{\beta}^{<\theta_0}-P_{<\alpha}^{\theta_0}+(r_{-}, \frac{n_{\alpha}^{-}-n_{\beta}^{-}}{2}+n_1-m_1)
    \end{equation}
    where $n_1$ and $m_1$ are the multiplicity of $e_{(0,-1)}$ in $\alpha$ and $\beta$ respectively. Moreover $r$ has the same parity of the number of the orbit $h_{(0,-1)}^2$ in $\alpha$ and $\beta$.
\end{proposition}

\begin{proof}
    Let $\epsilon>0$ be small enough such that $\alpha$ and $\beta$ has no Reeb orbits in 
    $$
    (-\pi/2,-\pi/2+\epsilon)\times T^2.
    $$
    Consider $C_{[-\pi/2+\epsilon,\theta_0]}$ the portion of $C$ in $\mathbb{R}\times[-\pi/2+\epsilon,\theta_0]\times T^2$, the boundary of $C_{[-\pi/2+\epsilon,\theta_0]}$ has components corresponding to the curves $C_{\theta_0}$, $C_{-\pi/2+\epsilon}$ and the positive and negative ends of $C$ with $-\pi/2+\epsilon<\theta<\theta_0$.

    Let $\alpha'$ and $\beta'$ be the orbit sets which $C$ has ends in the portion $\pi/2+\epsilon<\theta<\theta_0$. By the choice of orientation in the boundary we have
    $$
    [C_{\theta_0}]-[C_{-\pi/2+\epsilon}]=[\beta']-[\alpha']
    $$
    in $H_2(\mathbb{R}\times \{-\pi/2+\epsilon<\theta<\theta_0\}\times T^2,\mathbb{Z})\simeq \mathbb{Z}^2$.

    If $i:\mathbb{R}\times (-\pi/2,\pi/2)\times T^2\rightarrow U^*K$ is the inclustion, then we can compute $i_{*}[C_{-\pi/2+\epsilon}]$, to obtain that
    $$
    [C_{-\pi/2+\epsilon}]=(r_{-},\frac{n_{\alpha}^{-}-n_{\beta}^{-}}{2}+n_1-m_1),
    $$
    where $r_{-}$ is a integer which is has the parity of the number of the orbits $h_{(0,-1)}^2$ in $\alpha$ and $\beta$, $n_1$ and $m_1$ are the multiplicity of $e_{(0,-1)}$ in $\alpha$ and $\beta$ respectively.

    Note that by local energy inequality $r_{-}\geq 0$. Therefore, we write the toric part of $\alpha$ and $\beta$ by
    $$
    \prod_{i}\gamma_{(q_i,p_i)}\quad \textrm{and}\quad \prod_{j}\gamma_{(r_j,s_j)},
    $$
    then, we obtain
    $$
    [C_{\theta_0}]=(\sum_{s_j/r_j<\tan(\theta_0)}r_j-\sum_{p_i/q_i<\tan(\theta_0)}q_i+r_{-},\sum_{s_j/r_j<\tan(\theta_0)}s_j-\sum_{p_i/q_i<\tan(\theta_0)}p_i+\frac{n_{\alpha}^{-}-n_{\beta}^{-}}{2}+n_1-m_1).
    $$
\end{proof}

\begin{remark}
    We can also give a formula in terms of $P_{\alpha}^{>\theta}$, in this case we have
    \begin{equation}
    [C_{\theta_0}]=P_{\alpha}^{>\theta_0}-P_{\beta}^{>\theta_0}+(r_+,\frac{n_{\alpha}^{+}-n_{\beta}^{+}}{2}+n_1^{+}-m_1^{+})
    \end{equation}
    Moreover, by local energy inequality $r_+\leq 0$.
\end{remark}

\begin{definition}
    Let $P$ and $Q$ be $K$-lattice path. We say that $Q$ is above $P$ is $Q$ is contained in the region enclosed by $P$, the $x$-axis and the line $x=x(P)$, where $x(P)$ is the $x$-coordinate of the end point of $P$.
\end{definition}

\begin{proposition}\label{prop: pathcantcross}
    Let $\alpha$ and $\beta$ be admissible orbit sets with same total homology. Suppose that $\mathcal{M}^J(\alpha,\beta)\neq \emptyset$. Then, there exist $r^-\geq 0$ and $r^+\leq 0$ such that $P_{\beta}+(r_-,0)$ is above $P_{\alpha}$, moreover
    $$
    x(P_{\beta})-r_++r_-=x(P_{\alpha}).
    $$
\end{proposition}

\begin{proof}
    Consider $r_-$ and $r_+$ given by the combinatorial formula of the slice class of the $J$-holomorphic curve $C\in \mathcal{M}^J(\alpha,\beta)$. Denote by $P_{\beta}^r$ the translation of $P_{\beta}$ by the vector $(r,0)$. Denote $y_{\min}(P_{\alpha})$ the lowest $y_0\leq0$ such that $P_{\alpha}$ intersects the line $y=y_0$. We have to show that
    $$
    P_{\beta}^{r_-}\subset [0,x(P_{\alpha})]\times [y_{\min}(P_{\alpha}),0].
    $$
    Now taking $\theta_0\rightarrow \pi/2$ in the equation \ref{eq: sliceclasscomb}, we have
    $$
    -(x(P_{\beta})-x(P_{\alpha})+r_-)\geq 0.
    $$
    Therefore, $P_{\beta}^{r_-}\subset [0,x(P_{\alpha})]\times (-\infty,0]$. In other hand, using the equation \ref{eq: sliceclasscomb} with $\theta_0\rightarrow 0^+$, we obtain
    $$
    y_{\min}(P_{\beta})\geq y_{\min}(P_{\alpha}),
    $$
    so that $P_{\beta}^{r_-}\subset [0,x(P_{\alpha})]\times [y_{\min}(P_{\alpha}),0]$.

    \textbf{Claim.} The $K$-lattice path $P_{\alpha}$ do not intersect $P_{\beta}^{r_-}$.

    Indeed, by convexity of the the $K$-lattice path there exist at least two intersect points, which we suppose to be the first and the second $(a,b)$ and $(c,d)$ with $a<c$ and $b,d\leq 0$.

    Consider $\theta_0\in (-\pi/2,\pi/2)$ with $\tan(\theta_0)=d-b/c-a$. For $\epsilon>0$ sufficiently small, because $P_{\beta}^{r-}$ is above $P_{\alpha}$ before $(a,b)$ and bellow between $x=a$ and $x=c$ we must have
    $$
    \{e^{i\theta_0},[C_{\theta_0+\epsilon}]\} \textrm{ is a negative basis of }\mathbb{R}^2.
    $$
    However, by local energy inequality we have
    $$
    \det(e^{i\theta_0}, [C_{\theta_{\theta_0+\epsilon}}])\geq 0
    $$
    This is a contradiction, therefore $P_{\alpha}$ and $P_{\beta}^{r_-}$ are $K$-lattice path without intersections, this implies that $P_{\beta}^{r_-}$ is above $P_{\alpha}$.
\end{proof}

\section{Low index computations on ECH}\label{sec: lowind}

\subsection{Differential in index $1$ and $2$}

In this section, we use the information about ECH homology from Theorem \ref{Homologythem} to show the existence of certain $J$-holomorphic curves. This will be the first induction-step in the proof of differential, demonstrating that the $C$ operation occurs the ECH differential. From this point forward, we assume that $J$ is a small pertubation of $J_{st}$ such that the equation \ref{bijmoduli} holds, and there are no "double rounding" as explained in the Remark \ref{remark: nodoublerounding}.

We begin by noting that all generators of index less than two is in the table below.

\begin{table}[h!]
\centering
\begin{tabular}{|c|c|c|}
\hline
$I=0$ & $I=1$ & $I=2$ \\
\hline
$\emptyset$ & $h_{(1,-1)}h_{(0,1)}^1h_{(0,1)}^2$ & $e_{(1,0)}^2$ \\
\hline
$h_{(0,-1)}^1h_{(0,-1)}^2h_{(0,1)}^1h_{(0,1)}^2$ & $h_{(0,-1)}^1h_{(0,-1)}^2h_{(1,1)}$ & $h_{(1,-1)}h_{(1,1)}$ \\
\hline
 & $h_{(1,0)}e_{(1,0)}$ & $e_{(0,-1)}e_{(0,1)}$ \\
\hline
 & & $e_{(1,-1)}h_{(0,1)}^1h_{(0,1)}^2$\\
 \hline
 & & $h_{(0,-1)}^1h_{(0,-1)}^2e_{(1,1)}$\\
 \hline
 & & $e_{(0,-1)}e_{(0,1)}h_{(0,-1)}^1h_{(0,-1)}^2h_{(0,1)}^1h_{(0,1)}^2$\\
 \hline
\end{tabular}
\caption{Low index ECH generators}
\end{table}

\begin{theorem}\label{thm:1}
    We claim that the diferential ECH in grading one is
    \begin{align*}
    \partial(e_{(1,0)}h_{(1,0)})&=0\\
    \partial(h_{(1,-1)}h_{(0,1)}^1h_{(0,1)}^2)&=\partial(h_{(0,-1)}^1h_{(0,-1)}^2h_{(1,1)})=h_{(0,-1)}^1 h_{(0,-1)}^2 h_{(0,1)}^1 h_{(0,1)}^2   
    \end{align*}

\end{theorem}

\begin{proof}
    By the equation \ref{bijmoduli}, there is a bijection
    $$
    \mathcal{M}^J_1(h_{(1,-1)},h_{(0,-1)}^1 h_{(0,-1)}^2)\longrightarrow  \mathcal{M}^J_1(h_{(1,1)},h_{(0,1)}^1 h_{(0,1)}^2).
    $$

    It Follows immediately from index considerations that $\partial(\alpha)=0$ for any ECH index two generator $\alpha$, except possibly for $h_{(1,-1)}h_{(1,1)}$. Suppose both manifolds in the bijection above contain an even number of elements. This would imply that 
    $$
    \partial(h_{(1,-1)}h_{(1,1)})=e_{(1,0)}h_{(1,0)},
    $$ 
    and 
    $$
    \partial(h_{(1,-1)}h_{(0,1)}^1h_{(0,1)}^2)=\partial(h_{(0,-1)}^1h_{(0,-1)}^2h_{(1,1)})=0.
    $$
%%%%%%%%%%%%%%%
    However, in this case, the homology in grading one would be isomorphic to $\mathbb{Z}_2\oplus \mathbb{Z}_2$, which contradicts the Theorem \ref{Homologythem}. Therefore, both manifolds must have an odd number of elements, which completes the proof.
\end{proof}

\begin{corollary}
        The ECH differential in grading two is zero for all generators, except $h_1h_1$ whose differential is  
    $$
    \partial(h_{(1,-1)}h_{(1,1)})=h_{(1,-1)}h_{(0,1)}^1h_{(0,1)}^2+h_{(0,-1)}^1h_{(0,-1)}^2h_{(1,1)}+h_{(1,0)}e_{(1,0)}.
    $$
\end{corollary}

\begin{proof}
    This is an easily from the proof of Theorem \ref{thm:1}, which establishes the parity of the number of elements of both moduli spaces:
    $$
    \mathcal{M}^J_1(h_{(1,-1)},h_{(0,-1)}^1 h_{(0,-1)}^2)\simeq  \mathcal{M}^J_1(h_{(1,1)},h_{(0,1)}^1 h_{(0,1)}^2).
    $$
    which is odd. Using the unique index one curve in $\mathcal{M}_1^J(h_{(1,-1)}h_{(1,1)},e_{(1,0)}^2)$ given by Taubes in \cite{Taubes}, we conclude the differential of $h_{(1,-1)}h_{(1,1)}$. The fact that all other generators has ECH differential zero follows from the observation that, in any case, a nontrivial embedded Fredholm index one curve can not exist as part of a current of ECH index one.
\end{proof}

\begin{remark}
    In the Theorem \ref{thm:1}, we deduce $ECH^{L}(U^*(K),\lambda_{\epsilon(L)},0)$ isomorphic to $\mathbb{Z}_2\oplus \mathbb{Z}_2$. This contradicts the Theorem \ref{Homologythem}, because by taking the direct limit with $L\rightarrow +\infty$, we recover the ECH homology. However, for any fixed $L$ it is isomorphic to two copies of $\mathbb{Z}_2$. Therefore, in the limit, we preserve this property, leading to a contradiction.
\end{remark}

\subsection{U map in $e_{(0,-1)}^ke_{(0,1)}^k$}

In this section we compute the $U_{J,z}$ in the generator $e_{(0,-1)}^ke_{(0,1)}^k$.

\begin{proposition}\label{U-map-ee}
    Let $U$ the $U$-map of the ECH homology of the unit cotangent bundle of Klein bottle. Then $U(e_{(0,-1)}^ke_{(0,1)}^k)=e_{(0,-1)}^{k-1}e_{(0,1)}^{k-1}$.    
\end{proposition}

\begin{proof}
    Let $\beta$ be a ECH generator such that $\langle U(e_{(0,-1)}^ke_{(0,1)}^k),\beta \rangle=1$. Then there is $J$-holomorphic current of index $2$ between $e_{(0,-1)}^ke_{(0,1)}^k$ and $\beta$, so there is $1\leq i,j\leq k$ such that the current is of the form $C=C_0\cup C_1$ where $C_0$ are trivial cylinders in $e_{(0,-1)}$ with multiplicity $i$, a trivial cylinder in $e_{(0,1)}$ with multiplicity $j$, and $C_1$ is an index two embedded $J$-holomorphic curve in $\mathcal{M}^J(e_{(0,-1)}^i e_{(0,1)}^j,\beta_0)$ as in the figure \ref{fig2}. 
    
    We compute the Fredholm index for $C_1$ by partition condition we obtain
    $$
    2=\textrm{ind}(C_1)=2g(C_1)+2+n_{\beta_0}+h_{\beta_0}.
    $$
     \begin{figure}[!htb]
    \centering
    \includegraphics[width=0.40\textwidth]{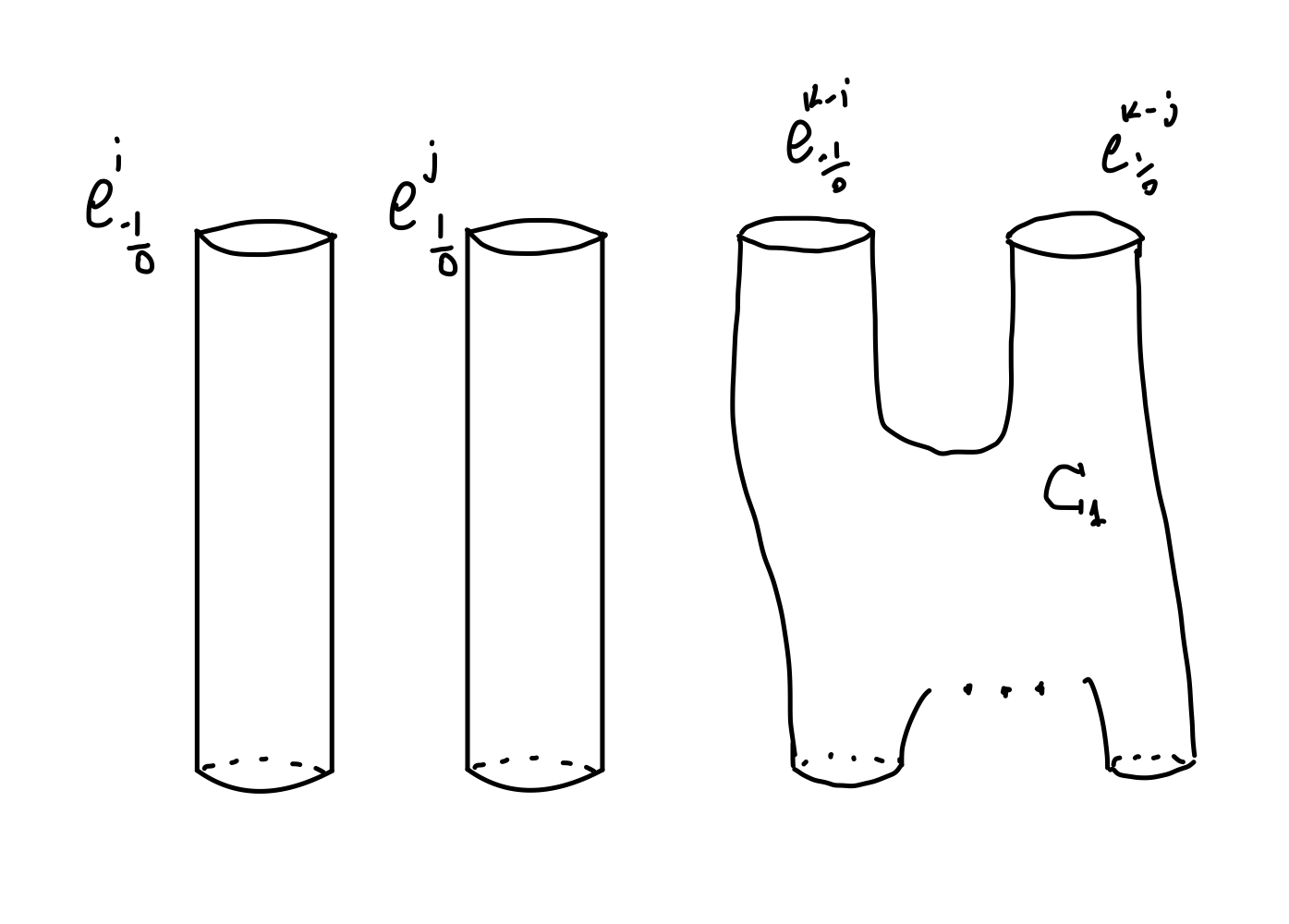}
    \caption{\label{fig2} The current counted by U map}
    \end{figure}
    In particular, $g(C_1)=n_{\beta_0}=h_{\beta_0}=0$, so we can write $\beta_0=\prod_{l=1}^k e_{(q_l,p_l))}$. Take any $\theta_0<\pi/2$ such that $p_l/q_l<\textrm{tan}(\theta_0)$ for all $l=1,...,k$. We can compute the homology $[C_{\theta_0}]$ and obtain
    $$
    [C_{\theta_0}]=\sum_{l=1}^k (q_l,p_l)+(0,j-i)\in \mathbb{Z}^2,
    $$
    By the local energy inequality we have
    $$
    \textrm{det}\begin{bmatrix}
        \cos(\theta_0) & \sum_{l=1}^k q_l\\
        \sin(\theta_0) & \sum_{l=1}^k p_l +j-i\\
    \end{bmatrix}
    \geq 0,
    $$
    that is,
    $$
    \cos(\theta_0)(\sum_{l=1}^k p_l +j-i)-\sin(\theta_0)(\sum_{l=1}^k q_l)\geq 0.
    $$
    Now observe that, this is true always that tan$(\theta_0)\geq p_l/q_l$, because there exist finite orbits in $\beta_0$ we can take $\theta_0\rightarrow \pi/2$, we obtain
    $$
    -\sum_{l=1}^k q_l\geq 0.
    $$
    but by definition, any orbit set has $q_l\geq 0$ for all $l$, this implies that
    $$
    \sum_{l=1}^k q_l=0.
    $$
    Therefore, $q_l=0$ for all $l=1,...,k$, but the unique elliptic orbits with $q=0$ are $e_{(0,-1)}$ and $e_{(0,1)}$, thus $\beta_0=e_{(0,-1)}^a e_{(0,1)}^b$. Now because $\beta=\{(e_{(0,-1)},i+a), (e_{(0,1)},b+j)\}$ and $I(e_{(0,-1)}^k e_{(0,1)}^k,\beta)=2$ implies that $i+a=k-1$ and $b+j=k-1$, thus $\beta=e_{(0,-1)}^{k-1}e_{(0,1)}^{k-1}$.
    To finish, we can apply Taubes results in \cite{Taubes} to concluded that this count of index two $J$-holomorphic currents is always $1$.
\end{proof}

\begin{corollary}\label{U-map-Isomorphism}
    The $U$ map in homology, for $k>0$
    $$
    U:ECH_{2k}(U^*K,\xi_{std},0)\longrightarrow ECH_{2k-2}(U^*K,\xi_{std},0),
    $$
    is an isomorphism.
\end{corollary}

\begin{proof}
    By Theorem \ref{Homologythem}, $ECH_{2k}(U^*K,\xi_{std},0)$ is isomorphic to $\mathbb{Z}_2$ for $k\geq 0$, the results follows by the Proposition \ref{U-map-ee}.
\end{proof}

\section{ECH differential}

In this section we show that the bijection defined in the section \ref{bijectioncomb} preserves differential in each complex. 

\subsection{Geometric framework implies the combinatorial}

\begin{proposition}\label{prop: geometric--->combinatorial}
    Let $\alpha$ and $\beta$ be ECH generators with $[\alpha]=[\beta]=0$. If $\langle \partial \alpha,\beta\rangle=1$, then $\langle \delta P_{\alpha},P_{\beta}\rangle=1$.
\end{proposition}

\begin{proof}
    Let $C=C_0\cup C_1$ be the $J$-holomorphic current counted by the ECH differential $\partial$, where:
    \begin{itemize}
        \item $C_0$ consists of trivial cylinders,
        \item $C_1\in \mathcal{M}_1^J(\alpha',\beta')$ is the nontrivial component embedded and of Fredhom index equal to $1$.
    \end{itemize}
    %%%%%%%%%%%%%%%%%%%%
    By the Fredholm index formula \ref{eq: fredind}, we have
    \begin{equation}\label{eq: ind=1g=1}    1=ind(C_1)=2(g(C_1)+e(\alpha')-1)+n_{\alpha'}+n_{\beta'}+h(\alpha')+h(\beta').
    \end{equation}
    \textbf{Assumption.} We will assume that $C_1$ is not a Morse-Bott cylinder between $e_{p/q}$ and $h_{p/q}$, because they appear in pairs, so do not contribute to the ECH differential.

    We consider the first case to analyze.

    \textbf{Case: genus one.} $(g(C_1)=1)$.

    The equation \ref{eq: ind=1g=1} simplifies to
    $$
    1=2e(\alpha')+n_{\alpha'}+n_{\beta'}+h(\alpha')+h(\beta').
    $$
    In particular, $e(\alpha')=0$, otherwise the (RHS) of the above equation is greater or equal to two. Moreover $n_{\alpha'}+n_{\beta'}$ is an even number, therefore $n_{\alpha'}=n_{\beta'}=0$. Thus
    $h(\alpha')+h(\beta')=1$, since $\alpha'\neq \emptyset$, we obtain
    $$
    \alpha'=h_{(q,p)}\quad and \quad \beta'=\prod_i e_{(q_i,p_i)}.
    $$

    \textbf{Claim 1.} $\alpha'$ corresponds to an extremal edge of $P_{\alpha}$, and $\beta'$ matches an extremal lattice path of $P_{\beta}$.

    \textbf{Proof of Claim 1.} By Lemma \ref{lem: oneorbit} the $J$-holomorphic curve $C_1$ is not contained in $(-\pi/2,\pi/2)\times T^2$, that is, $C_1$ must intersect $\mathbb{R}\times K_-$ or $\mathbb{R}\times K_+$. 
    
    If $\alpha'$ was not an extremal edge of $P_{\alpha}$, then $C_1$ must intersect some trivial cylinder in $C_0$, we have the following inequality by positivity of intersection of $J$-holomorphic curves
    $$
    1=I(C_0\cup C_1)\geq I(C_1)+2\#(C_0\cap C_1)\geq 3,
    $$
    which is a contradiction. The same argument shows that $\beta'$ matches a initial or final $K$-lattice path of $P_{\beta}$, without loss of generality we suppose that $\alpha'$ correspond to the first edge of $P_{\alpha}$.
    
    \textbf{Claim 2.} $p/q\in (-1,0)$.

    \textbf{Proof of Claim 2.}
    Suppose for contradiction that $p/q\leq -1$, we have
    $$
    1=I(\alpha,\beta)=2(L(\alpha)-L(\beta))-(x(\alpha)-x(\beta))-1.
    $$
    Write $L(\alpha)=1+x(\alpha)+L^+(\alpha)$, where $L^+(\alpha)$ are the count of lattice point on $P_{\alpha}$ strictly below the $x$-axis. Then this equation simplifies to
    $$
    2=2(L^+(\alpha)-L^+(\beta))+(x(\alpha)-x(\beta)).
    $$
    By the Proposition \ref{prop: pathcantcross}, we must have $L^+(\alpha)-L^+(\beta)\in \{0,1\}$.

    \textbf{Case 1.} $L^+(\alpha)-L^+(\beta)=1$
    
    The condition $L^+(\alpha)-L^+(\beta)=1$ is impossible, since in this case $x(\alpha)=x(\beta)$, which forces $r_-=r_+=0$ in Proposition \ref{prop: pathcantcross}, but then $P_{\beta}$ is a convex lattice path above the line of slope $p/q$ passing through origin with initial point $(0,0)$ and final point $(q,p)$, the only possibility is that $\beta=e_{(q,p)}$, but this yields $I(\alpha,\beta)=-1$, a contradiction.

    \textbf{Case 2.} $L^+(\alpha)-L^+(\beta)=0$

    The condition $L^+(\alpha)-L^+(\beta)=0$ is also impossible, in this situation $r_-$ and $-r_+$ are nonnegative and even integer numbers, we have
    $$
    -r_++r_-= 2.
    $$
    So we must have $(r_-,-r_+)=(2,0)$ or $(r_-,-r_+)=(0,2)$. We necessary have $r_+=0$, because $[h_{p/q}]=[\beta']$, implies that the end point of $P_{\beta'}$ is in $y=p$, and the because the lattice path $P_{\beta'}$ is above the line passing through origin of slope $p/q$, we must have $p_i/q_i<0$ for any $i$, by local energy inequality we have $r_+=0$.

    The condition $p/q\leq -1$ implies that $L^+(\alpha)>L^+(\beta)$ and we obtain again a contradiction. Therefore, both cases yields a contradiction, and this proves the claim $2$.

    Similarly, if $\alpha'=h_{(q,p)}$ is the last edge of $P_{\alpha}$, then a similar proof shows that $p/q\in (0,1)$. Therefore, this implies the $C$ operation in the case $|p/q|<1$ and extreme edges labeled with $``h"$.

    \textbf{Case: genus zero.} $(g(C_1)=0)$.

    The Fredhom index formula simplifies to
    $$    3=2e(\alpha')+n_{\alpha'}+n_{\beta'}+h(\alpha')+h(\beta').
    $$
    This gives two subcases based on $e(\alpha')\in \{0,1\}$. 
    
    \textbf{Subcase 1:} $e(\alpha')=0$.
    
    The equation above becomes 
    $$
    3=n_{\alpha'}+n_{\beta'}+h(\alpha')+h(\beta').
    $$
    Since $n_{\alpha'}+n_{\beta'}$ must be even, we have:
    $$
    n_{\alpha'}+n_{\beta'}\in \{0,2\}
    $$

    \textbf{Case 1.1.} $n_{\alpha'}+n_{\beta'}=0$.

    This implies 
    $$
    h(\alpha')+h(\beta')=3.
    $$
    We can not have $h(\alpha')=1$, otherwise
    $$
    1=I(\alpha,\beta)=2(L^{+}(\alpha)-L^{+}(\beta))+(x(\alpha)-x(\beta))+1.
    $$
    Therefore, $L^{+}(\alpha)=L^{+}(\beta)$ and $x(\alpha)=x(\beta)$, but this implies that the unlabeled $K$-lattice path associated to $\alpha$ and $\beta$ are equal, this is impossible unless $C_1$ is a Morse-Bott cylinder between some positive hyperbolic orbit and an elliptic orbit in the same Morse-Bott torus, but we are assuming that $C_1$ is not such curve.

    It is also impossible have $h(\alpha')=3$. As explained in the Remark \ref{remark: nodoublerounding}, we consider some generic pertubations of $J_{std}$ and by a argument with SFT-compactness we derive a contradiction. Now suppose that $h(\alpha')=2$. In this case we have $h(\beta')=1$.
    
    \textbf{Claim.} $\alpha'$ and $\beta'$ are contained and homologous in $(-\pi/2,\pi/2)\times T^2$.

    Indeed, if this not holds then $\alpha'$ must be extremal in $\alpha$, that is, $\alpha'$ is the two first arrows or the two last arrows in $P_{\alpha}$, because if not then $C_1$ must intersect some trivial cylinder in $C_0$ which is a contradiction. Without loss of generality, suppose that $\alpha'$ correspond to the two first arrows on $P_{\alpha}$.

    In this case, $C_1$ must intersect $\mathbb{R}\times K_-$, the condition on $I(\alpha,\beta)=1$, implies
    $$
    2(L^{+}(\alpha)-L^{+}(\beta))+(x(\alpha)-x(\beta))=2.
    $$
    If $L^{+}(\alpha)-L^{+}(\beta)=0$, we have $r_-=2$, by Proposition \ref{prop: pathcantcross}, $P_{\alpha'}$ and $P_{\beta'}+(2,0)$ has the same end point and the condition $L^{+}(\alpha)-L^{+}(\beta)=0$ implies that the last edge of $P_{\beta'}$ coincides with the last edge of $P_{\alpha'}$. Write 
    $$
    \alpha'=h_{(q_1,p_1))}h_{(q_2,p_2)}\quad \textrm{and}\quad \beta'=\prod_{i=1}^{k-1}e_{(s_i,r_i)}\cdot h_{(s_k,r_k))}\cdot \prod_{i=k+1}^n e_{(s_i,r_i)}\gamma_{(q_2,p_2)}.
    $$
    Now, take $\theta_0\in (-\pi/2,\pi/2)$, such that $\tan(\theta_0)=p_2/q_2$, for $\epsilon>0$ sufficiently small we have
    $$
    [C_{\theta_0-\epsilon}]=(x(\beta),y(\beta))-(x(\alpha),y(\alpha))+(2,-m_1)=(0,0),
    $$
    by local energy inequality $C$ do not intersect the slice $\theta=\theta_0-\epsilon$, but $C$ has ends in $\theta=\theta_0$, thus a contradiction.

    If $L^{+}(\alpha)-L^{+}(\beta)=1$, then $r_-=r_+=0$. Therefore, $P_{\alpha'}$ and $P_{\beta'}$ has same start and end points, but this implies that $\beta'$ has only Reeb orbits in $(-\pi/2,\pi/2)\times T^2$ because the first edge of $P_{\alpha'}$ is not vertical, and we concluded that $\alpha'$ and $\beta'$ are homologous in $(-\pi/2,\pi/2)\times T^2$, a contradiction.

    Thus, we concluded the proof that in $h(\alpha')=2$, these orbit sets are homologous in $(-\pi/2,\pi/2)\times T^2$. In this case, by local energy inequality the only possibility for the $J-$holomorphic curve $C_1$ is from Taubes results in \cite{Taubes} and this correspond to the operation interior rounding the corners.

    \textbf{Case 1.2.} $n_{\alpha'}+n_{\beta'}=2$.

    In this case, we have three possibilities $(n_{\alpha'},n_{\beta'})=(2,0),(0,2)$ or $(1,1)$. The case $(1,1)$ is ruled out by the following lemma.

    \begin{lemma}\label{onehyperbolic}
    Let $h_{(q,p)}$ a positive hyperbolic orbit. Suppose that $\beta$ is a toric orbit set and that the lattice path $P_{\beta}$ is above the lattice path $P_{h_{p/q}}$ and that $[\beta]=[h_{(q,p)}]$ in $H_1(U^*(K),\mathbb{Z})$, then the action of $\beta$ is greater than or equal to $\sqrt{q^2+p^2}$.
    \end{lemma}

    \begin{proof}
    The equality in the homology class implies that the lattice path $P_{\beta}$ has end point on the line $y=p$, the fact that $P_{\beta}$ is above of $P_{h_{(q,p)}}$ necessary implies that its total lenght has to be greater than or equal to the lenght of $P_{h_{(q,p)}}$.
    \end{proof}

    The Fredholm index condition gives 
    $$
    h(\alpha')+h(\beta')=1
    $$
    The only possibility is $h(\alpha')=1$ and $h(\beta')=0$, due to the action obstruction.

    Suppose that $n_{\alpha'}=2$ and $n_{\beta'}=0$. Without loss of generality, suppose that
    $$
    \alpha'=h_{(0,-1)}^1h_{(0,-1)}^2h_{(q,p)}\quad \textrm{and}\quad \beta'=\prod_{i=1}^n e_{(q_i,p_i)}
    $$
    The $K$-lattice path $P_{\alpha'}$ matches to the extremal initial of $P_{\alpha}$, otherwise $C_1$ must intersect a trivial cylinder in $C_0$. We compute ECH index to obtain:
    $$
    1=I(\alpha,\beta)=2(L^+(\alpha)-L^+(\beta))+(x(\alpha)-x(\beta))-2,
    $$
    which simplifies to
    $$
    3=2(L^+(\alpha)-L^+(\beta))+(x(\alpha)-x(\beta)).
    $$
    We concluded that, $L^+(\alpha)-L^+(\beta)\in\{0,1\}$. We now show that $L^+(\alpha)-L^+(\beta)=0$ leads to a contradiction.
    
    Indeed, if $L^+(\alpha)-L^+(\beta)=0$, then $3=x(\alpha)-x(\beta)$, in this case $r_-$ (from the Proposition \ref{prop: pathcantcross}) is an odd number, in any case $r_-=1$ or $r_-=3$ by convexity is impossible the existence of such $K$-lattice path. 

    Therefore, we must have $L^+(\alpha)-L^+(\beta)=1$, which implies $x(\alpha)-x(\beta)=1$. This forces $r_-=1$ and $r_+=0$ by the condition on the parity of $r_-$ and $r_+$. In this case $P_{\beta'}=D(P_{\alpha'})$, that is, $P_{\beta'}$ start in $(1,0)$ has same end point of $P_{\alpha'}$ and there are no lattice point in the region interior between these $K$-lattice paths. This is the case of the $D$ operation when the ECH generator start or ends with two negative hyperbolic orbits and a positive hyperbolic orbit.

    Now suppose that $n_{\alpha'}=0$ and $n_{\beta'}=2$. We compute the ECH index and obtain:
    $$
    1=2(L^+(\alpha)-L^+(\beta))+(x(\alpha)-x(\beta)).
    $$
    We concluded that, $L^+(\alpha)-L^+(\beta)=0$ and $x(\alpha)-x(\beta)=1$. In this case, without loss of generality $\alpha'=h_{(q,p)}$ and $\beta'=h_{(0,-1)}^1 h_{(0,-1)}^2\prod_{i=1}^n e_{(q_i,p_i)}$. In this case $r_-=1$ and $r_+=0$ and we concluded that $P_{\beta'}=C(P_{\alpha'})$, that is, $P_{\beta'}$ is obtained by doing rounding in the starting point of $P_{\alpha'}$, but the two first arrows are semi-arrows starting in $(1,0)$ and there is no lattice point in the region between these two $K-$lattice path. This finish the case $e(\alpha')=0$. 

    \textbf{Subcase 2:} $e(\alpha')=1$.
    
    In this case, the Fredhom index equation for $C_1$ is
    $$
    1=n_{\alpha'}+n_{\beta'}+h(\alpha')+h(\beta').
    $$
    We know that $n_{\alpha'}+n_{\beta'}$ is an even number, therefore $n_{\alpha'}=n_{\beta'}=0$. The ECH index equation give
    $$
    0=2(L^+(\alpha)-L^+(\beta))+(x(\alpha)-x(\beta)),
    $$
    so that $L^+(\alpha)-L^+(\beta)=0$ and $x(\alpha)-x(\beta)=0$. In this case, $r_-=0=r_+$ and by local energy inequality $C_1$ is contained in $\mathbb{R}\times (-\pi/2,\pi/2)\times T^2$. In particular $h(\alpha')=1$ and $h(\beta')=0$. By local energy inequality, the only possibility for $C_1$ is the $J-$holomorphic curve produced by Taubes results in \cite{Taubes}.  This is the case of interior rounding the corners, when the vertex has adjacent edges labeled by $``e"$ and $``h"$.
\end{proof}

\subsection{Combinatorial framework implies the geometric}

We now prove that the combinatorial differential model recovers the geometric embedded contact homology differential.

\begin{proposition}
    Let $\alpha$ and $\beta$ be ECH generators with $[\alpha]=[\beta]=0$. If $\langle \delta P_{\alpha},P_{\beta} \rangle=1$, then $\langle \partial \alpha,\beta \rangle=1$.
\end{proposition}

The proof proceds in some steps. First, we show that interior rounding in $K$-lattice path's contributes to ECH differential. In fact we will prove that if the first or last ``toric" arrow of $P_{\alpha}$ is elliptic, then $\partial(\alpha)=\delta(P_{\alpha})$.

\begin{proposition}\label{prop: introunding}
    Let $\alpha$ be an admissible orbit set, and $P_{\alpha}$ denote its corresponding K-lattice path. Suppose that $\beta$ is an admissible orbit set such that $P_{\beta}$ is obtained from $P_{\alpha}$ by interior rounding operation. Then $\langle \partial(\alpha),\beta \rangle=1$.
\end{proposition}

\begin{proof}
    Suppose that $P_{\beta}$ is obtained from $P_{\alpha}$ by interior rounding the corner, follows from local energy inequality that a $J$-holomorphic current $C=C_0\cup C_1$ of ECH index one between $\alpha$ and $\beta$ satisfies the following:
    \begin{itemize}
        \item $C_1$ has ends exactly on the orbit sets where $\alpha$ and $\beta$ are not equal.
    \end{itemize}
    Therefore, the ends of $C_1$ are homologous on $(-\pi/2,\pi/2)\times T^2$. Applying local energy inequality, we deduce that $C_1$ is confined as follows: if $\theta_{\min}$ and $\theta_{\max}$ are real numbers in $(-\pi/2,\pi/2)$ such that, the orbits in positive or negative ends of $C_1$ are in $\{\theta\}\times T^2$ torus with $\theta_{\min}\leq \theta\leq\theta_{\max}$, then 
    $$
    C_1\subset \mathbb{R}\times [\theta_{\min},\theta_{\max}]\times T^2.
    $$

    Follows from Taubes results in \cite{Taubes}, that there exist a unique $J-$holomorphic curve between the positive ends of $C_1$ and the negative ends of $C_1$, this shows that $\langle \partial \alpha,\beta\rangle=1.$
    \end{proof}

    \begin{corollary}\label{Cor: Diff1}
    Supppose that $P_{\alpha}$ has first or last ``toric edge" elliptic, then $\partial \alpha=\delta P_{\alpha}$.
    \end{corollary}

    \begin{proof}
    Let $\alpha$ be an ECH generator such that $P_{\alpha}$ is a $K$-lattice path with first or last ``toric edge" is not positive hyperbolic, follows from Proposition \ref{prop: geometric--->combinatorial}, that if $\langle \partial \alpha,\beta\rangle=1$, then $P_{\beta}$ is obtained from $P_{\alpha}$ by interior rounding the corner. By Proposition \ref{prop: introunding} follows the result.
    \end{proof}

    \subsection{C operation for $p/q\in \mathbb{Q}\setminus (-1,1)$}

    We now examine the situation where the first or last ``toric edge" of $P_{\alpha}$ is labeled by ``h" (corresponding to a positive hyperbolic Reeb orbit). In this section, we consider two cases
    \begin{itemize}
        \item The first edge of $P_{\alpha}$ is labeled by ``h" with slope in $
        (-\infty,-1]$.
        \item The last edge of $P_{\alpha}$ is labeled by ``h" with slope in $[1,+\infty)$.
    \end{itemize}
%%%%%%%%%%%%%
    Recall that for a positive hyperbolic Reeb orbit $h_{(q,p)}$ with $p<0$, the $C$-operation $C(h_{(q,p)})$ is the unique $K$-lattice path starting in $(1,0)$ and with end point in $(q,p)$, with two first down semi-arrows and all toric orbits labeled by ``e", such that there is no interior lattice point in the region defined by $P_{h_{p/q}}$ and $C(h_{p/q})$.

    \begin{center}
\begin{tikzpicture}
    % First lattice path
    \begin{scope}
        \foreach \x in {0,1,2} {
            \foreach \y in {0,-1,-2,-3,-4,-5} {
                \fill (\x,\y) circle (1pt);
            }
        }
        \foreach \x/\y in {0/0, 0/-2, 2/-1, 2/0} {
            \fill (\x,\y) circle (1pt);
        }
        \draw[->, thick] (0,0) -- node[left]{$h$} (2,-5) ;
        \node[right] at (3.5,-2.5) {\large $\underset{\longrightarrow}{C}$};
    \end{scope}
    
    % Second lattice path (shifted right by 4 units)
    \begin{scope}[xshift=5cm]
        \foreach \x in {0,1} {
            \foreach \y in {0,-1,-2,-3,-4,-5} {
                \fill (\x,\y) circle (1pt);
            }
        }
        \foreach \x/\y in {0/0, 0/-2} {
            \fill (\x,\y) circle (1pt);
        }
        \draw[->, thick] (0,0) -- node[left]{$h^{-}$} (0,-0.5);
        \draw[->, thick] (0,-0.5) -- node[left]{$h^{-}$} (0,-1);
        \draw[->, thick] (0,-1) -- node[left]{$e$} (0,-2);
        \draw[->, thick] (0,-2) -- node[right]{$e$} (1,-5);
    \end{scope}
\end{tikzpicture}
\end{center}

    \begin{proposition}\label{Prop: CoperationJ}
        Let $h_{(q,p)}$ be a positive hyperbolic orbit where $p/q\in \mathbb{Q}\setminus (-1,1)$. Then, we have
        $$
        \#_2\mathcal{M}_1^J(h_{(q,p)},C(h_{(q,p)}))=1
        $$
        Moreover, if $p/q\leq -1$ and $C\in \mathcal{M}_1^J(h_{(q,p)},C(h_{(q,p)}))$. Then $C\subset \mathbb{R}\times \{\theta\leq \arctan{(p/q)}\}\times T^2$. The similar result holds for $p/q\geq 1$.
    \end{proposition}

\begin{proof}

    \textbf{Base case:} For $p/q=\pm 1$, the result follows directly from the Theorem \ref{thm:1}, which in particular establishe that 
    $$
    \#_2\mathcal{M}_1^J(h_{(1,1)},h_{(0,-1)}^1 h_{(0,-1)}^2)=\#_2\mathcal{M}_1^J(h_{(1,-1)},h_{(0,1)}^1 h_{(0,1)}^2)=1.
    $$
    \textbf{Induction:} Now, we proceed by induction on the ``height" $p$, using the fact that $\partial^2=0$. Indeed, suppose that $q>0$ and $p<0$ are integers and $gcd(q,p)=1$. Let $\alpha$ be the toric generator given by $\alpha=h_{(q,p)}h_{(q,-p)}$. Assume, for contradiction, that the result does not hold for $p$. Then, the differential of $\alpha$ is only the interior rounding by the Proposition \ref{prop: introunding} and Proposition \ref{prop: geometric--->combinatorial}, that is
    $$
    \partial(\alpha)=R_{int}(\alpha),
    $$
    which corresponds to performing rounding while losing a "h" in the vertex $(q,p)$. Therefore, we can express 
    $$
    \partial{\alpha}=\alpha_i+\sum_k \alpha_k+\alpha_f,
    $$ 
    where 
    \begin{itemize}
        \item $\alpha_i$ is the unique generator in $\partial \alpha$, whose first edge is labeled by ``h".
        \item $\alpha_f$ is the unique generator in $\partial \alpha$, whose last edge is labeled by ``h".
        \item $\alpha_k$ are generators in $\partial \alpha$ such that the edge labeled by ``h" is not the first or the last.
    \end{itemize}
    Applying the ECH differential again, we obtain the following result
    $$
    \partial^2(\alpha)=\partial(\alpha_i)+\sum_k \partial(\alpha_k)+\partial(\alpha_f).
    $$
    Follows from Corollary \ref{Cor: Diff1} that the differential of each $\alpha_k$ involves only a sum of toric generators with only elliptic orbits. In contrast, by induction the first term $\partial(\alpha_i)$ contains exactly one generator summand  involving both negative hyperbolic orbits $h_{(0,-1)}^1h_{(0,-1)}^2$, and the last term $\partial(\alpha_f)$ contains a generator summand involving both negative hyperbolic orbits $h_{(0,1)}^1h_{(0,1)}^2$. Therefore, $\partial^2(\alpha)\neq 0$, which leads to a contradiction. 
    
    By induction the result holds for any $p$, and the proof of existence is complete. Now by local energy inequality any $J$-holomorphic curve between $h_{p/q}$ and $C(h_{p/q})$, must lie in the region describe in the proposition. We conclude for $p>0$ using the equation \ref{bijmoduli}.
\end{proof}

By extremal edges of a $K$-lattice path $P$, we referee by the first or last edge of the $K$-lattice path $P$.

\begin{theorem}\label{thm; ECHDIFF1}
    Let $\alpha$ be a toric orbit set with $[\alpha]=0$ such that the extremal edges of $P_{\alpha}$, if positive hyperbolic, satisfies $p/q\leq -1$ or $p/q\geq 1$. Then
    $$
    \partial(\alpha)=\delta P_{\alpha}=R_{int}(\alpha)+C(\alpha).
    $$
\end{theorem}

\begin{proof}
    By hypothesis, $P_{\alpha}$ does not start with two down semi-arrows together with a toric edge labeled by ``h" or end with a toric edge labeled by ``h" together two up semi-arrows. The proposition \ref{prop: geometric--->combinatorial} implies that 
    $$
    \partial \alpha=aR_{int}(\alpha)+bC(\alpha)
    $$
    with $a,b\in \{0,1\}$. By Proposition \ref{prop: introunding} we must have $a=1$. Now the Proposition \ref{Prop: CoperationJ} implies that $b=1$.
\end{proof}

\subsection{$D$ operation for $p/q\in \mathbb{Q}\setminus (-1,1)$}

We now establishe the $D$ operation in the ECH differential, when the first or last toric orbit of $P_{\alpha}$ is positive hyperbolic with slope in $p/q\in \mathbb{Q}\setminus (-1,1)$ and is adjacent two negative hyperbolic orbit. This configuration was described in the introduction (see Section \ref{descrip: differential}). 

Unlike the case of the $C$ operation, the proof here relies on the Weyl Law for the ECH spectrum (see Corollary \ref{asymptotic}), to establish the existence of certain pseudoholomorphic curves.

\begin{lemma}
    The condition $\#_2\mathcal{M}_1^J(h_{(0,-1)}^1h_{(0,-1)}^2h_{(1,-1)},e_{(0,-1)}^2)=1$ holds if and only if we have the $D$ operation in the ECH differential for $|p/q|\geq 1$.
\end{lemma}

\begin{proof}
    ($\Rightarrow$) Suppose that $\#_2\mathcal{M}_1^J(h_{(0,-1)}^1h_{(0,-1)}^2h_{(1,-1)},e_{(0,-1)}^2)=1$. We prove by induction that on $-p$ (for $p<0$) that $\#_2 \mathcal{M}_1^J(h_{-1/0}^1 h_{-1/0}^2 h_{p/q}, D(h_{p/q})) = 1$.

    \textbf{Base case:} In the case $p=-1$, the claim holds by assumption.

     \textbf{Induction step:} Assume the result holds for $-1,...,p+1$ and prove to $p$ (recall $p$ is negative).
    Suppose by contradiction that it fails for $p$. We derive a contradiction using $\partial^2=0$. 
    
    If $q>0$ is even, consider the generator
    $$
    \alpha=h_{(0,-1)}^1h_{(0,-1)}^2h_{(q,p)}h_{(1,1)}e_{(0,1)}^{-p},
    $$
    which satisfies $[\alpha]=0$ in $H_1(U^*(K),\mathbb{Z})$. By induction hypothesis, we compute
    $$
    \partial(\alpha)=h_{(0,-1)}^1h_{(0,-1)}^2h_{(q',p')}\alpha_0e_{(0,1)}^{-p}+\sum_{i}h_{(0,-1)}^1h_{(0,-1)}^2\alpha_ie_{(0,1)}^{-p}+h_{(0,-1)}^1h_{(0,-1)}^2h_{(q,p)}e_{(1,2)}e_{(0,1)}^{-p-1},
    $$
    where the orbit sets $\alpha_0$ is an orbit set with only elliptic orbits, while $\alpha_i$ has a unique positive hyperbolic orbit which is not the first (in the order given by tangency). The first generator is the unique generator obtained by doing rounding in the vertice $(q,p-1)$ so that the hyperbolic orbit is the first in the order given by tangency $p'/q'$, the sum are only orbit sets obtained by doing rounding and such that the unique new hyperbolic orbit is not the first in the order of tangency, the last generator is obtained by doing rounding in the vertice $(q+1,p+1)$.

    In that case, we compute the second ECH differential of $\alpha$ and obtain
    $$
    0=\partial^2(\alpha)=D(h_{(q',p')})\alpha_0e_{(0,1)}^{-p+1}+\sum_i\partial(h_{(0,-1)}^1h_{(0,-1)}^2\alpha_ie_{(0,1)}^{-p+1}),
    $$
    but observe that the first summand is a toric generator, that is, there is no negative hyperbolic orbit, but the second summand has always the two negative hyperbolic orbits $h_{(0,-1)}^1h_{(0,-1)}^2$ from the Thereom \ref{thm; ECHDIFF1}, also by our contradiction assumption, the third summand is obtained only by doing rounding in the vertice $(1,p-1)$, that is, is a summand of orbit sets which always have the two negative hyperbolic orbits $h_{-1/0}^1h_{-1/0}^2$, but then a toric generator is a combination of generators with negative hyperbolic orbits, a contradiction. Now if $q$ is odd just take the generator $\alpha=h_{-1/0}^1h_{-1/0}^2h_{p/q}h_{1/2}e_{1/0}^{-p-1}$, with a similar argument we derive the same contradiction.
    Therefore, the result must holds to $p$ and this prove that we have $D$ operation to $p/q\leq -1$. The same argument applies to $p/q\geq 1$ or we just apply the symmetry condition \ref{bijmoduli}.
\end{proof}

\begin{lemma}\label{lemma: Doplema}
    Let $\alpha$ be an ECH generator such that $\langle \partial \alpha,e_{(0,-1)}^ne_{(0,1)}^n\rangle=1$ with $n\geq 2$. Then 
    $$
    \alpha=h_{(0,-1)}^1h_{(0,-1)}^2h_{(1,-n+1)}e_{(0,1)}^n\quad \textrm{or}\quad \alpha=e_{(0,-1)}^nh_{(1,n-1)}h_{(0,1)}^1h_{(0,1)}^2.
    $$
\end{lemma}

\begin{proof}
    Suppose that $n_{\alpha}=0$. Let $C=C_0\cup C_1$ be an index one $J$-holomorphic current between $\alpha$ and $e_{(0,-1)}^ne_{(0,1)}^n$. Let $\alpha'$ be the orbit set which is the positive asymptotic of $C_1$. The Fredholm index is
    $$
    1=ind(C_1)=2(g(C_1)+e(\alpha')-1)+h(\alpha').
    $$
    If $g(C_1)=1$, then $e(\alpha')=0$ and $h(\alpha')=1$. Let $\beta=e_{(0,-1)}^ie_{(0,1)}^j$ be the orbit set which is the negative asymptotic of $C_1$.

    \textbf{Claim.} $j=0$.

    In fact, if $0<\theta_0<\pi/2$, write $[C_{\theta}]=(r,s)$ the slice class of $C_1$. We can compute the slice class and obtain
    \begin{equation}\label{eq: nat}
    s=-j,
    \end{equation}
    now by the Lemma \ref{lemma: localenergy}, we have $s\cos(\theta_0)-r\sin(\theta_0)\geq 0$, because $\theta_0$ is arbitrary, we obtain $-j\geq 0$, but $j\geq 0$, therefore, $j=0$.

    \textbf{Claim.} $i=n$.

    In this case, we write $\alpha=h_{(q,p)}e_{(0,1)}^n$, but then
    $$
    I(\alpha)=qn+n=n(q+1).
    $$
    However, $2n+1=I(\alpha)=n(q+1)$, this implies that $q\leq 2$, but if $q=2$, then $n=1$, and if $q\in\{0,1\}$ we obtain contradictions. 
    
    If $n_{\alpha}=2$, let $C=C_0\cup C_1$ be an index one $J$-holomorphic current between $\alpha$ and $e_{(0,-1)}^ne_{(0,1)}^n$. Let $\alpha'$ be the orbit set which is the positive asymptotic of $C_2$. By the Fredholm index formula we have
    $$
    1=ind(C_1)=2(g(C_1)+e(\alpha')-1)+2+h(\alpha').
    $$
    Therefore, $g(C_1)=e(\alpha')=0$, and $h(\alpha')=1$. This implies that $\alpha'=h_{(0,-1)}^1h_{(0,-1)}^2 h_{(q,p)}$ or $\alpha'=h_{(q,p)}h_{(0,1)}^1h_{(0,1)}^2$.

    Let $\beta=e_{(0,-1)}^ie_{(0,1)}^j$, the orbit set which is the negative asymptotic of $C_1$. Without loss of generality, suppose that $\alpha'=h_{(0,-1)}^1h_{(0,-1)}^2 h_{(q,p)}$, with $p/q\leq0$.

    \textbf{Claim.} $j=0$.

    In fact, if $0<\theta_0<\pi/2$, write $[C_{\theta}]=(r,s)$ the slice class of $C_1$. We can compute the slice class and obtain
    \begin{equation}\label{eq: nat}
    s=-j,
    \end{equation}
    now by the Lemma \ref{lemma: localenergy}, we have $s\cos(\theta_0)-r\sin(\theta_0)\geq 0$, because $\theta_0$ is arbitrary, we obtain $-j\geq 0$, but $j\geq 0$, therefore, $j=0$.

    \textbf{Claim.} $i=n$.

    Write $\alpha=h_{(0,-1)}^1 h_{(0,-1)}^2 e_{(0,-1)}^{n-i}h_{(q,p)}e_{(0,1)}^n$, and the trivial part of the $J$-holomorphic current $C_0$ is the trivial cylinder in $e_{-1/0}$ with multiplicities $n-i$. We can compute explicitly the ECH index of $\alpha$ and obtain
    $$
    I(\alpha)=2(q(n-i+1)+\frac{q(i-1)}{2})+2n-i.
    $$
    Because $I(\alpha)=2n+1$, we obtain
    $$
    i+1\geq q(2n-i+1)\geq 2n-i+1\geq n+1,
    $$
    thus, $i=n$, $q=1$ and $p=1-n$.

    Follows that $C_0$ is the trivial cylinder in $e_{-1/0}$ with multiplicity $n$ and $C_1$ is an embedded $J$-holomorphic curve in $\mathcal{M}_1^J(h_{(0,-1)}^1h_{(0,-1)}^2h_{(1,1-n)},e_{(0,1)}^n)$, this finish the proof of the lemma.
\end{proof}

\begin{proposition}
    The equality $\#_2\mathcal{M}_1^J(h_{(0,-1)}^1h_{(0,-1)}^2h_{(1,-1)},e_{(0,-1)}^2)=1$ holds.
\end{proposition}

\begin{proof}
    Suppose by contradiction that not, then the generator $e_{(0,-1)}^ne_{(0,1)}^n$, which has ECH grading equal to $2n$, is the unique ECH generator whose class in ECH homology is nonzero, this follows from \ref{U-map-ee}. Now by Theorem \ref{Homologythem}, the homology in nonnegative gradings is $\mathbb{Z}_2$, and so if this equality does not holds follows from \ref{lemma: Doplema} that the generator $e_{(0,-1)}^ne_{(0,1)}^n$ do not satisfies $\langle \partial\beta,e_{(0,-1)}^ne_{(0,1)}^n\rangle=1$ for any $\beta$, but we know that $e_{(0,-1)}^ne_{(0,1)}^n$ is a cycle. In this case, the ECH spectrum is given by the action of $e_{(0,-1)}^ne_{(0,1)}^n$, that is
    $$
    c_{n}(U^*(K),\lambda_{st})=2n,
    $$
    but in this case we have, 
    $$
    \lim_{n\rightarrow +\infty}c_{n}(U^*(K),\lambda_{st})^2/n=\lim_{n\rightarrow +\infty}4n^2/n=+\infty,
    $$
    this contradicts the Corollary \ref{asymptotic}.
\end{proof}

\begin{corollary}\label{Cor: D-oper1}
    The $D$ operation is on the ECH differential if $\alpha$ is an ECH generator such that 
    \begin{itemize}
        \item $P_{\alpha}$ starts with $P_{h_{(0,-1)}^1h_{(0,-1)}^2h_{p/q}}$, with $p/q\leq -1$.
        \item $P_{\alpha}$ ends with $P_{h_{p/q}h_{(0,1)}^1h_{(0,1)}^2}$, with $p/q\geq 1$.
    \end{itemize}
\end{corollary}

\subsection{C and D operations with $|p/q|<1$}

To finish the description of the ECH differential, remains to show the $C$ and $D$ operations when:
\begin{itemize}
    \item $P_{\alpha}$ starts with $P_{h_{(0,-1)}^1h_{(0,-1)}^2h_{(q,p)}}$ with $p/q>-1$.
    \item $P_{\alpha}$ end with $P_{h_{(q,p)}h_{(0,1)}^1h_{(0,1)}^2}$, with $p/q<1$.
    \item $P_{\alpha}$ starts with $P_{h_{(q,p)}}$ with $p/q\in (-1,0)$.
    \item $P_{\alpha}$ ends with $P_{h_{(q,p)}}$ with $p/q\in (0,1)$.
\end{itemize}

Turns out that all information to end the description of the ECH differential follows from the information that we already have.

\begin{proposition}
    The $D$ operation for $p/q\leq -1$ implies the $D$ operation for $p/q>-1$.
\end{proposition}

\begin{proof}
    We first show that we have $D$ operations for $\alpha$ such that
        \begin{center}
        $
        P_{\alpha}$ starts with $P_{h_{(0,-1)}^1h_{(0,-1)}^2h_{p/q}}$ with $p/q>-1$.
        \end{center}

    To do that, we need some induction steps to derive a general proof.
    
    \textbf{Claim.} $\partial h_{(0,-1)}^1h_{(0,-1)}^2h_{(2,-1)}e_{(1,2)}=e_{(0,-1))}e_{(1,-1)}e_{(1,2)}+h_{(0,-1)}^1h_{(0,-1)}^2e_{(1,0)}^2e_{(1,1)} $. In particular 
    $$
    \#_2\mathcal{M}_1^J(h_{(0,-1)}^1h_{(0,-1)}^2 h_{(2,-1)},D(h_{(0,-1)}^1h_{(0,-1)}^2 h_{(2,-1)}))=1
    $$

    This follows from take $\partial^2\beta=0$, where $\beta=h_{(0,-1)}^1h_{(0,-1)}^2 h_{(1,-1)}h_{(1,0)}e_{(1,2)}$. Indeed, note that 
    $$
    \partial \beta=e_{(0,-1)}h_{(1,0)}e_{(1,2)}+h_{(0,-1)}^1h_{(0,-1)}^2h_{(2,-1)}e_{(1,2)}+h_{(0,-1)}^1h_{(0,-1)}^2h_{(1,-1)}e_{(1,1)}^2.
    $$
    Therefore, because $\partial^2=0$, we have
    \begin{eqnarray*}
    \partial h_{(0,-1)}^1h_{(0,-1)}^2h_{(2,-1)}e_{(1,2)}&=&\partial e_{(0,-1)}h_{(1,0)}e_{(1,2)}+\partial h_{(0,-1)}^1h_{(0,-1)}^2h_{(1,-1)}e_{(1,1)}^2\\
    &=& e_{(0,-1)}e_{(1,-1)}e_{(1,2)}+e_{(0,-1)}^2e_{(1,1)}^2+e_{(0,-1)}^2e_{(1,1)}^2+h_{(0,-1)}^1h_{(0,-1)}^2e_{(1,0)}^2e_{(1,1)}\\
    &=& e_{(0,-1)}e_{(1,-1)}e_{(1,2)}+h_{(0,-1)}^1h_{(0,-1)}^2e_{(1,0)}^2e_{(1,1)}.
    \end{eqnarray*}

    \textbf{Claim.} $\#_2\mathcal{M}_1^J(h_{(0,-1)}^1h_{(0,-1)}^2h_{(1,0)},e_{(0,-1)})=1$. 
    
    This follows as the last case computing now $\partial h_{(0,-1)}^1h_{(0,-1)}^2h_{(2,-1)}h_{(1,1)}e_{(0,1)}$.

    \textbf{Claim.} By induction we have 
    $$
    \#_2\mathcal{M}_1^J(h_{(0,-1)}^1h_{(0,-1)}^2h_{(q,-1)},D(h_{(0,-1)}^1h_{(0,-1)}^2h_{(q,-1)}))=1.
    $$

    Indeed, by the last claim we can assume that $q>2$. Suppose that $q$ is odd ($q$ even is a similar argument). Then, consider the generator
    $$
    \alpha=h_{(0,-1)}^1h_{(0,-1)}^2h_{(q-1,-1))}h_{(1,0)}e_{(0,1)}^2,
    $$
    its ECH differential by the induction hypothesis, can be computed, and appears the generator $h_{(0,-1)}^1h_{(0,-1)}^2h_{(q,-1)}e_{(0,1)}^2$ plus other terms which if the start toric orbit is hyperbolic, then is of the form $h_{(r,-1)}$ with $r\leq q-1$. Therefore, by induction hypothesis and $\partial^2=0$, we can compute the ECH differential of $h_{(0,-1)}^1h_{(0,-1)}^2h_{(q,-1)}e_{(0,1)}^2$, which is
    $$
    D(h_{(0,-1)}^1h_{(0,-1)}^2h_{(q,-1)}e_{(0,1)}^2)+R_{int}(h_{(0,-1)}^1h_{(0,-1)}^2h_{(q,-1)}e_{(0,1)}^2),
    $$
    which proves the claim.

    \textbf{Claim.} By induction we have, $\partial \alpha=\delta P_{\alpha}$, when $P_{\alpha}$ starts with $P_{h_{(0,-1)}^1h_{(0,-1)}^2h_{(q,p)}}$ with $p/q\leq 0$.

    This follows by the last claim and the Corollary \ref{Cor: D-oper1}.

    Now if the generator $\alpha$ is such that $P_{\alpha}$ starts with $P_{h_{(0,-1)}^1h_{(0,-1)}^2 h_{(q,p)}}$ with $p/q>0$, because $[\alpha]=0\in H_1(U^*K,\mathbb{Z})$. We must have $p=1$, $q$ is an odd number and $\alpha=h_{(0,-1)}^1h_{(0,-1)}^2 h_{(q,1)}$. To show that we need the following claim. Note that
    $$
    \alpha=h_{(0,-1)}^1h_{(0,-1)}^2 h_{(q,1)} \Longrightarrow D(\alpha)=e_{(1,0)}^{q-1}.
    $$
    \textbf{Claim.} $\#_2\mathcal{M}_1^J(h_{(0,-1)}^1h_{(0,-1)}^2h_{(q,1)},e_{(1,0)}^{q-1})=1$.

    We start by observing that
    $$
    \#_2\mathcal{M}_1^J(h_{(2,-1)},e_{(0,-1)})=1.
    $$
    This follows by compute $\partial^2 h_{-1}h_0e_{1/0}=0$. Now by induction (using that $\partial^2=0$), we can show that 
    \begin{eqnarray}\label{C-op1/q}
    \#_2\mathcal{M}_1^J(h_{(q,-1)},C(h_{(q,-1)}))=\#_2\mathcal{M}_1^J(h_{(q,1)},C(h_{(q,1)}))=1, \forall q>0,
    \end{eqnarray}
    where the first equality follows from symmetry equation \ref{bijmoduli}. Now we show the claim, consider the following ECH generator
    $$
    \alpha=h_{(0,-1)}^1h_{(0,-1)}^2h_{(1,0)}h_{(q-1,1)},
    $$
    Computing its ECH differential we obtain
    \begin{eqnarray*}
        \partial \alpha &=& e_{(0,-1)}h_{1/(q-1,1)}+h_{(0,-1)}^1h_{(0,-1)}^2h_{(q,1)}+h_{(0,-1)}^1h_{(0,-1)}^2h_{(1,0)}e_{(q-3,1))}\\
    \end{eqnarray*}
    Therefore, by using that $\partial^2=0$, we obtain
    \begin{eqnarray*}
        \partial h_{(0,-1)}^1h_{(0,-1)}^2h_{(q,1)} &=& \partial e_{(0,-1)}h_{(q-1,1))}+\partial h_{(0,-1)}^1h_{(0,-1)}^2h_{(1,0)}e_{(q-3,1))}\\
        &=& e_{(1,0)}^{q-1}+e_{(0,-1)}e_{(q-3,1))}+e_{(0,-1)}e_{(q-3,1)}+h_{(0,-1)}^1h_{(0,-1)}^2e_{(q-2,1))}\\
        &=& e_{(1,0)}^{q-1}+h_{(0,-1)}^1h_{(0,-1)}^2e_{(q-2,1)}.
    \end{eqnarray*}
    This implies the claim.

    In this case we prove that we have the $D$ operation in the ECH differential, always that $P_{\alpha}$ starts with $P_{h_{(0,-1)}^1h_{(0,-1)}^2h_{(q,p)}}$, with $p/q>-1$. By similar argument, or just by using the equation \ref{bijmoduli}, we can also concluded that the $D$ operation appears in the ECH differential when $P_{\alpha}$ ends with $P_{h_{(q,p)}h_{(0,1)}^1 h_{(0,1)}^2}$, with $p/q<1$ and this concludes the proof of the Proposition. 
\end{proof}

To concluded the description of the ECH differential in terms of combinatorics, we have to prove that the ECH differential involves the $C$ operation when the ECH generator $\alpha$ satisfies:

\begin{itemize}
    \item $P_{\alpha}$ starts with $P_{h_{p/q}}$ with $p/q\in (-1,0)$.
    \item $P_{\alpha}$ ends with $P_{h_{p/q}}$ with $p/q\in (0,1)$.
\end{itemize}

\begin{proposition}
    Let $\alpha$ be an ECH generator such that $P_{\alpha}$ satisfies one of the two conditions above. Then $\partial \alpha=\delta P_{\alpha}$.
\end{proposition}

\begin{proof}
By equation \ref{C-op1/q}, the result holds for $p=-1$. We now apply induction in $-p$, without loss of generality suppose that $P_{\alpha}$ starts with $P_{h_{(q,p)}}$ for some $q>0$. Consider $(r,s)$ the point below the line $y=(p/q)x$ in the fourth quadrant of the plane which is closest to this line subject to the constraint $1\leq r\leq q-1$.

\textbf{Claim.} $\gcd(r,s)=1$, if $(r,s)\neq (1,-1)$.

Indeed, if this not holds, then there exist an positive integer $d$ such that $r/d$ and $s/d$ are integers, but the point $(r/d,s/d)$ is in the line $y=(s/r)x$, therefore is below the line $y=(p/q)x$, but is closest to this line than $(r,s)$ which contradicts the minimality of $(r,s)$.

Note that by also by minimality of $(r,s)$ we have $\gcd(p-s,q-r)=1$. Assume first that $q$ is an even number. Therefore, we can consider the ECH generator
$$
\alpha=h_{(r,s)}h_{(q-r,p-s)}e_{(0,1)}^{-p},
$$
If $(r,s)=(1,-1)$, we can compute the ECH differential of $\alpha$ by
$$
\partial \alpha=h_{(q,p)}e_{(0,1)}^{-p}+h_{(0,-1)}^1 h_{(0,-1)}^2h_{q-1,p+1}e_{(0,1)}^{-p}+\sum_i\Theta_i
$$
where each $\Theta_i$ is obtained by doing rounding in the vertex $(q,p)$ and losing a ``h". Therefore, using that $\partial^2=0$, we obtain
\begin{eqnarray*}
    \partial h_{(q,p)}e_{(0,1)}^{-p}&=& \partial h_{(0,-1)}^1 h_{(0,-1)}^2h_{(q-1,p+1)}e_{(0,1)}^{-p}+ \sum_i\partial \Theta_i.
\end{eqnarray*}
By applying the $D$ operation we have
$$
\langle \partial h_{(q,p)}e_{(0,1)}^{-p}, C(h_{(q,p)})e_{(0,1)}^{-p}\rangle=1
$$
In particular, $\#_2\mathcal{M}_1^J(h_{(q,p)},C(h_{(q,p)}))=1$ and any $J-$holmorphic curve in this moduli space stays in the set $\{\theta<\min\{p/q,\theta_{\max}(C(h_{(q,p)}))\}\}$, where $\theta_{\max}(C(h_{(q,p)}))$ is the maximum $\theta$ where $C(h_{(q,p)})$ has Reeb orbits in the Morse-Bott torus $\{\theta\}\times T^2$.

Now if $(r,s)$ is not $(1,-1)$, by the last claim $r$ and $s$ are coprimes. In this case we have $0>r/s>-1$, then we use induction hypothesis and compute the ECH differential of $\alpha$
$$
\partial\alpha=C(h_{(r,s)})h_{q-r,p-s}e_{(0,1)}^{-p}+h_{(q,p)}e_{(0,1)}^{-p}+\sum_i\Theta_i
$$
Therefore, using that $\partial^2=0$, we can compute
$$
\partial h_{(q,p)}e_{(0,1)}^{-p}=\partial (C(h_{(r,s)})h_{(q-r,p-s)}e_{(0,1)}^{-p})+\sum_i\partial \Theta_i.
$$
By doing interior rounding in the vertex $(r,s)$ in the first summand we obtain
$$
\langle \partial h_{(q,p)}e_{(0,1)}^{-p}, C(h_{(q,p)})e_{(0,1)}^{-p}\rangle=1,
$$
and we concluded the result in the case that $P_{\alpha}$ starts with $P_{h_{p/q}}$ with $0>p/q>-1$. By similar argument we can concluded the $C$ operation when $P_{\alpha}$ ends with $P_{h_{p/q}}$ with $p/q\in (0,1)$.
\end{proof}

\begin{remark}
    What happens when $p=0$, that is, $P_{}\alpha$ is contained in the $x$-axis?. In this case, $P_{\alpha}=e_{(1,0)}^n$ for some $n$, or $P_{\alpha}=h_{(1,0)}e_{(1,0)}^n$. In the first case the differential is obviously zero by local energy inequality and ECH index equal to one. In the second case, we also have differential zero, because in this case the unique way of $J$-holomorphic current is trivial cylinder with multiplicity $n$ on $e_{(1,0)}$ and some index one $J$-holomorphic curve between $h_{(1,0)}$ and $e_{(1,0)}$, but this is impossible, since $I(h_{(1,0)},e_{(1,0)})=-1$.
\end{remark}

\section{Combinatorial ECH spectrum of $(U^*K,\lambda_{std})$}

In this section, we show a combinatorial formula for the ECH spectrum of $(U^*K,\lambda_{std})$ in terms of $K$-convex lattice path.

By Proposition \ref{U-map-ee}, we know that the ECH generator $e_{-1/0}^ne_{1/0}^n$, which has grading $2n$, satisfies $U^n[e_{-1/0}^ne_{1/0}^n]_{ECH}=[\emptyset]_{ECH}$. Since the homology is $\mathbb{Z}_2$ in nonnegative gradings, we must have the following formula for the ECH spectrum.

\begin{proposition}
    The ECH spectrum of $U^*(K)$ with the standard contact form is given by
    $$
    c_{k}(U^*(K),\lambda)=\min\{A(\alpha)\mid [\alpha]_{\textrm{ECH}}=[e_{-1/0}^ke_{1/0}^k]_{\textrm{ECH}}\}
    $$
    where $[\alpha]_{\textrm{ECH}}$, means the homology class of the generator $\alpha$ in the ECH-homology.
\end{proposition}

Recall that for a $K$-lattice path $P$, its action $A(P)$ is just the sum of euclidean length of its edges. In this section, we prove the Theorem \ref{combcapacity}, which give a combinatorial description of ECH spectrum for $(U^*K,\lambda_{std})$. To prove the Theorem \ref{combcapacity}, we need prove two lemmas.

\begin{lemma}\label{lemma: eECHhomologous}
    Let $\alpha$ and $\beta$ be ECH generators with even grading and such that they do not have positive hyperbolic orbits. Then $\alpha$ is ECH homologous to $\beta$.
\end{lemma}

\begin{proof}
    By the description of the ECH differential is enough to prove that $P_{\alpha}$ is $\delta$-homologous to the $K$-lattice path $P_{e_{(1,0)}^{2k}}$. First observe that using interior rounding operation we can show that $P_{\alpha}$ is $\delta$-homologous a $K$-lattice contained in $[0,+\infty]\times [-1,0]$. Now we can apply again interior rounding operation to show that it is $\delta$-homologous a $K$-lattice path which touches the line $y=-1$ once. We can apply $C$ operations if necessary to show that this $K$-lattice path is $\delta$-homologous to a toric generator, that is, having only elliptic orbits comming from $(-\pi/2,\pi/2)\times T^2$. Now we apply interior rounding operation to show that it is $\delta$-homologous to 
    $$
    P_{e_{(2k-2,-1)}e_{(0,1)}}.
    $$ 
    To finish we consider the ECH generator $\beta=e_{(2k-2,-1)}h_{(2,1)}$, then we have
    $$
    \delta P_{\beta}=P_{e_{(1,0)}^{2k}}+P_{e_{(2k-2,1)}e_{(0,1)}}.
    $$
    Therefore, $\alpha$ is ECH homologous to $e_{(1,0)}^{2k}$.
\end{proof}

\begin{lemma}
    Let $\alpha$ be an ECH generator of grading $2k$. If $\alpha$ have some positive hyperbolic orbit, then there is an ECH generator $\alpha'$ of index $2k$ such that $A(\alpha')<A(\alpha)$.
\end{lemma}

\begin{proof}
    If $\alpha$ have even grading and also some positive hyperbolic orbit, by the parity property of the ECH index, $\alpha$ must have at least two positive hyperbolic orbit. Consider $P_{\alpha}$ its associated $K$-lattice path. Consider the new $K$-lattice path $Q$ construct as follows: Let $(a,b)$ a vertex of $P_{\alpha}$ which has some edge labeled by ``h" with start or end point in $(a,b)$. There is two cases, the first is when both edges adjacents to $(a,b)$ are labeled by ``h", in this case we doing rounding in this vertex but all new edges are labeled by ``e", of course this new generator has strictly less action, moreover, of course, its combinatorial grading is the same as of $P_{\alpha}$ by Pick's formula.

    Suppose now that not the two adjacent edges to $(a,b)$ are labeled by ``h", in this case we make rounding in the vertice $(a,b)$, while we change the label of some other edge labeled by ``h" in $p_{\alpha}$ which is not adjacent to $(a,b)$, as in the last case this strictly decrease the action, but conserves the ECH grading.

    Because the ECH grading is conserved, in each step of this procedure we have a even number of positive hyperbolic orbits. Therefore, we can apply until have a generator with only elliptic orbits or negative hyperbolic.
\end{proof}

\begin{proof} of Theorem \ref{combcapacity}

Follows from the two previous lemmas that, to compute the $k$-ECH spectrum we just need consider generators without positive hyperbolic orbits. Therefore, the (LHS) is greater or equal to (RHS) of the equation \ref{formulaspecECH}.

In other hand, given a $K$-lattice path $P$ which attain the minimum in the (RHS) of the equation \ref{formulaspecECH}, consider $\alpha$ the ECH generator such that $P_{\alpha}=P$, then $\alpha$ has ECH grading equal to $2k$. Since there are no positive hyperbolic orbits, follows from the combinatorial description of the ECH differential that $\partial \alpha=\emptyset$. By Lemma \ref{lemma: eECHhomologous}, $\alpha$ induces a non-trivial homology class in $ECH_{2k}(U^*K,\xi_{std},0)$. By definition of ECH capacity $c_k^{ECH}(U^*K,\lambda_{std})\leq A(P)$, but we take a arbitrary $K$-lattice path $P$ without edges labeled by ``h" and with combinatorial grading equal to $2k$, then the (LHS) is less than or equal to the (RHS) in the equation \ref{formulaspecECH}.
    
\end{proof}

\section{Combinatorial obstruction to symplectic embeddings and Gromov width}

\subsection{Results about symplectic cobordism}

Consider $(Y_+,\lambda_+)$ and $(Y_-,\lambda_-)$ closed three-manifolds with contact forms. A \textit{strong symplectic cobordism} from $(Y_+,\lambda_+)$ to $(Y_-,\lambda_-)$ is a compact symplectic four-manifold $(W,\omega)$ such that
$$
\partial W=Y_+ -Y_-\quad and \quad \omega|_{Y_{\pm}}=d\lambda_{\pm}.
$$

Let $(W,\omega)$ be a strong symplectic cobordism as above, one can choose a neighborhood $N_+$ of $Y_+$ in $W$ and an identification of it with $(-\epsilon,0]\times Y_+$ for some $\epsilon>0$, such that $\omega=e^s\lambda_+$, where $s$ is the real coordinate in $(-\epsilon,0]$. Likewise one can choose a neighborhood $N_-$ of $Y_-$ in $W$, identified with $[0,\epsilon)\times Y_-$ on which $\omega=e^s\lambda_-$. Fix this choices, we define the ``completion" of $(W,\omega)$ to be the four-manifold
$$
\overline{W}=((-\infty,0]\times Y_-)\cup_{Y_-}W\cup_{Y_{+}}([0,\infty)\times Y_+)
$$
glued using the neighborhood identifications above.

An almost complex structure $J$ on $\overline{W}$ is ``cobordism-admissible" if satisfies
\begin{itemize}
    \item $J$ is $\omega$-compatible on $W$.
    \item On the sets $(-\infty,0]\times Y_-$ and $[0,\infty)\times Y_+$, $J$ agrees with $\lambda_{\mp}$-compatible almost complex structure $J_{\mp}$.
\end{itemize}

Let $J$ be a cobordism-admissible almost complex structure. If $\alpha_{\pm}$ are orbit sets in $Y_{\pm}$, we consider the set $\mathcal{M}^J(\alpha_+,\alpha_-)$ of $J$-holomorphic currents in $\overline{W}$. We can defined in the same way as in the symplectization the Fredholm index, ECH index, $J_0$ index for such curves, for more details see \cite{Hutchings2}.

\begin{definition}
    Fix a cobordism-admissible almost complex structure $J$ on $\overline{W}$ which restricts to $\lambda_{\pm}$-compatible almost complex structures $J_{\pm}$ on the ends of the completion. Fix orbit sets $\alpha_{\pm}$ in $Y_{\pm}$. A broken $J$-holomorphic current from $\alpha_+$ to $\alpha_-$ is a tuple
    $$
    \mathcal{C}=(\mathcal{C}(N_-),\mathcal{C}(N_{-}+1),...,C(N_+)),
    $$
    with $N_-\leq 0\leq N_+$, such that there exist orbit sets $\alpha_-=\alpha_{-}(N_-),...,\alpha_-(0)$ in $Y_-$, and orbit sets $\alpha_+(0),...,\alpha_{+}(N_+)=\alpha_+$ in $Y_{+}$, satisfying:
    \begin{itemize}
        \item $\mathcal{C}(i)\in \mathcal{M}^{J_-}(\alpha_{-}(i+1),\alpha_{-}(i))/\mathbb{R}$ for $i=N_-,...,-1$.
        \item $\mathcal{C}_0 \in \mathcal{M}^J(\alpha_{+}(0),\alpha_{-}(0))$.
        \item $\mathcal{C}(i)\in \mathcal{M}^{J_-}(\alpha_{+}(i),\alpha_{+}(i-1))/\mathbb{R}$ for $i=1,...,N_+$.
        \item If $i\neq 0$, then there exist a component of $\mathcal{C}(i)$ which is not a trivial cylinder.
    \end{itemize}
    The holomorphic currents $\mathcal{C}(i)$ are called the ``levels" of the broken $J$-holomorphic current $\mathcal{C}$. The ECH index of the broken $J$-holomorphic current $\mathcal{C}$ is defined as
    $$
    I(\mathcal{C})=\sum_{i=N_-}^{N_{+}} I(\mathcal{C}(i)).
    $$
\end{definition}

\subsection{L-tame cobordisms}

Let $(Y,\lambda)$ be a nondegenerate contact three-manifold. Suppose that $\gamma$ is an elliptic Reeb orbit with rotation angle $\theta$.

\begin{definition}
    Let $L>0$ and let $\gamma$ be an embedded elliptic Reeb orbit with action $\mathcal{A}(\gamma)<L$.
    \begin{itemize}
        \item $\gamma$ is called $L$-positive if its rotation angle $\theta\in (0,\mathcal{A}(\gamma)/L) \mod 1$.
        \item We say that $\gamma$ is $L$-negative if its rotation angle $\theta\in (-\mathcal{A}(\gamma)/L,0) \mod 1$.
    \end{itemize}
\end{definition}

Let $(W,\omega)$ be a strong symplectic cobordism from $(Y_+,\lambda_+)$ to $(Y_-,\lambda_-)$, where the boundary are nondegenerate contact three-manifold.

\begin{definition}
    Suppose $\alpha_{\pm}$ are orbit sets for $\lambda_{\pm}$, define $e_{L}(\alpha_+,\alpha_-)$ to be the total multiplicity of all elliptic orbits in $\alpha_+$ that are $L$-negative, plus the total multiplicity of all elliptic Reeb orbits in $\alpha_-$ that are $L$-positive. If $\mathcal{C}\in \mathcal{M}^J(\alpha_+,\alpha_-)$, we define $e_{L}(\mathcal{C})=e_{L}(\alpha_+,\alpha_-)$.
\end{definition}

Fix $J$ a cobordism-admissible almost complex structure on $\overline{W}$, and let $L\in \mathbb{R}$ positive real number. If $C$ is an irreducible $J$-holomorphic curve from $\alpha_+$ to $\alpha_-$, let $g(C)$ denote the genus of $C$, and $h(C)$ denote the number of ends of $C$ at hyperbolic Reeb orbits.

\begin{definition}
    The strong symplectic cobordism $(W,\omega,J)$ is $L$-tame if for every $C$ embedded irreducible $J$-holomorphic curve in $\mathcal{M}^J(\alpha_+,\alpha_-)$, such that there exists a positive integer $d$ such that $\mathcal{A}(\alpha_{\pm})\leq L/d$ and $I(dC)\leq C$, then
    \begin{equation}\label{eq: L-tame}
        2g(C)-2+ind(C)+h(C)+2e_{L}(C)\geq 0.
    \end{equation}
\end{definition}

In general, we can not to guarantee that multiply covered curves in cobordism have nonnegative ECH index, the significance of $L$-tameness is the following proposition.

\begin{proposition}\label{prop: cobLtame}{(See \cite{Hutchings0}, Proposition 4.6.)} Suppose that $J$ is generic and $(W,\omega,J)$ is $L$-tame. Let
    $$
    \mathcal{C}=\sum_{k} d_k\mathcal{C}_k\in \mathcal{M}^J(\alpha_+,\alpha_-).
    $$
    be a $J$-holomorphic current in $\overline{W}$ with $\mathcal{A}(\alpha_{\pm})<L$. Then:
    \begin{enumerate}
        \item $I(\mathcal{C})\geq 0$.
        \item If $I(\mathcal{C})=0$, then:
            \begin{itemize}
                \item $I(C_k)=0$ for each $k$.
                \item If $i\neq j$, then $C_i$ and $C_j$ do not both have positive ends at covers of the same $L$-negative elliptic Reeb orbit, and $C_i$ and $C_j$ do not both have negative ends at covers of the same $L$-positive elliptic Reeb orbits.
                \item If $d_k'$ are integers with $0\leq d_k'\leq d_k$, then
                $$
                I(\sum_{k} d_k'C_k)=0.
                $$
            \end{itemize}
    \end{enumerate}
\end{proposition}

\begin{definition}
    We call the strong symplectic cobordism $(W,\omega)$ weakly exact if there exists a $1$-form $\lambda$ in $X$ such that $d\lambda=\omega$.
\end{definition}

\begin{theorem}\label{thm: cobordism map}{(See \cite{Hutchings0}, Theorem 3.5.)}
    Let $(Y_+,\lambda_+)$ and $(Y_-,\lambda_-)$ be closed oriented three-maniflds with contact forms, and let $(X,\omega)$ be a weakly exact symplectic cobordism from $(Y_+,\lambda_+)$ to $(Y_-,\lambda_-)$. Then there are canonical maps
    $$
    \Phi^L:ECH^L(Y_+,\lambda_+,0)\longrightarrow ECH^L(Y_-,\lambda_-,0)
    $$
    for each $L\in \mathbb{R}$ with the following properties:
    \begin{enumerate}
        \item If $L<L'$, then the diagram
        %%%%%%%%%%%%%%%%%%DIAGRAMA
        \[
        \begin{tikzcd}
        ECH^L(Y_+,\lambda_+,0) \arrow[r, "\Phi^{L}"] \arrow[d, ""'] & ECH^L(Y_-,\lambda_-,0) \arrow[d, ""] \\
        ECH^{L'}(Y_+,\lambda_+,0\arrow[r, "\Phi^{L'}"]                  & ECH^{L'}(Y_-,\lambda_-,0)
        \end{tikzcd}
        \]
        %%%%%%%%%%%%%%%%%%DIAGRAMA
        commutes. In particular,
        $$
        \Phi:\lim_{L\rightarrow +\infty} \Phi^L:ECH(Y_+,\lambda_+,0)\rightarrow ECH(Y_-,\lambda_-,0)
        $$
        is well-defined.

        \item $\Phi^L$ commutes with the $U$-map.
        
        \item If $J$ is a cobordism-admissible almost complex structure on $\overline{W}$, restricting to generic $\lambda_{\pm}$-compatible almost complex structures $J_{\pm}$ on the ends of the completion, then $\Phi^L$ is induced by a chain map 
        $$
        \phi:ECC^L(Y_+,\lambda_+,0,J_+)\longrightarrow ECC^L(Y_-,\lambda_-,0,J_-),
        $$
        such that:
            \begin{itemize}
                \item If $\alpha_{\pm}$ are admissible orbit sets for $\lambda_{\pm}$, with $[\alpha_{\pm}]=0$ and $\mathcal{A}(\alpha_{\pm})<L$, and if the coefficient $\langle \phi(\alpha_+),\alpha_-\rangle\neq 0$, then there exists a broken $J$-holomorphic current $\mathcal{C}$ from $\alpha_+$ to $\alpha_-$.
                \item The broken $J$-holomorphic current $\mathcal{C}$ in the first bullet satisfies $I(\mathcal{C})=0$.
            \end{itemize}
    \end{enumerate}
\end{theorem}

\subsection{Toric domains in $\mathbb{C}^2$}

In this section, we describe a well know class of symplectic manifolds called toric domains in $\mathbb{C}^2$.

\begin{definition}
    Let $\Omega$ be a region in the first quadrant of the plane. The toric domain $X_{\Omega}$ associated to it is the set
    $$
    X_{\Omega}=\{(z_1,z_2)\in\mathbb{C}^2\mid \pi(|z_1|^2,|z_2|^2)\in \Omega\},
    $$
    with the restriction of the standard symplectic form on $\mathbb{C}^2$.
    $$
    \omega=\sum_{i=1}^4dx_i\wedge dy_i.
    $$
\end{definition}

\begin{example}
    The ellipsoid $E(a,b)$ is the toric domain $X_{\Omega}$, where $\Omega$ is the triangle with vertices in $(0,0)$, $(a,0)$ and $(0,b)$. In particula, the ball $B(a)$ is defined as the ellipsoid $E(a,a)$.
    
\begin{figure}[h]
    \centering
    \begin{tikzpicture}[scale=0.6]
    % Axes
    \draw[->] (0,0) -- (6,0) node[right] {\small $x$};
    \draw[->] (0,0) -- (0,4) node[above] {\small $y$};
    
    % Points
    \coordinate (O) at (0,0);
    \coordinate (A) at (5,0); % a = 5
    \coordinate (B) at (0,3); % b = 1
    
    % Convex hull (shaded region)
    \fill[blue!20,opacity=0.6] (O) -- (A) -- (B) -- cycle;
    
    % Draw edges
    \draw[thick,blue] (O) -- (A) -- (B) -- cycle;
    
    % Labels
    \filldraw[black] (O) circle (1.5pt) node[anchor=north east] {\small $(0,0)$};
    \filldraw[black] (A) circle (1.5pt) node[anchor=north west] {\small $(a,0)$};
    \filldraw[black] (B) circle (1.5pt) node[anchor=south east] {\small $(0,b)$};
\end{tikzpicture}
\caption{Ellipsoid $E(a,b)$}
\end{figure}
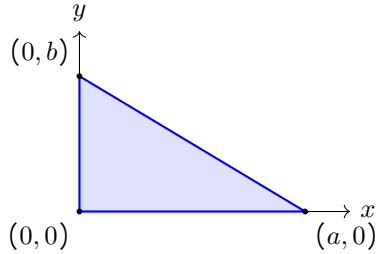
\end{example}

\begin{definition}
    A toric domain $X_{\Omega}$ is convex if
    $$
    \Omega=\{(x,y)\in \mathbb{R}^2\mid 0\leq x\leq A, 0\leq y\leq f(x)\},
    $$
    where $f:[0,A]\rightarrow \mathbb{R}_{\geq 0}$ is a nonincreasing concave function with $f(0)>0$ and $f(A)=0$.
\end{definition}

We now describe the combinatoric generators of the boundary of a convex toric domain.

\begin{definition}
    A convex integral path $\Lambda$ is a path in the plane such that:
    \begin{itemize}
        \item The initial point is $(0,y(\Lambda))$ and the end point is $(x(\Lambda,0))$, where $x(\Lambda)$ and $y(\Lambda)$ are non-negative integers.
        \item The path $\Lambda$ is the graph of a piecewise linear concave function $f:[0,x(\Lambda)]\rightarrow [0,y(\Lambda)]$, with $f'(0)\leq 0$, possibly with a vertical line segment at the right.
        \item The non-smooth points of $f$ are lattice points.
    \end{itemize}
\end{definition}

\begin{definition}
    A convex generator is a convex integral path $\Lambda$ such that:
    \begin{itemize}
        \item Each edge of $\Lambda$ is labeled by ``e" or ``h".
        \item horizontal or vertical is always labeled by ``e".
    \end{itemize}
\end{definition}

Following \cite{Hutchings0} we adopt a notation for a convex generator as follows.

Consider $a$ and $b$ nonnegative integers with $\gcd(a,b)=1$. For $m\geq 1$ an integer, $e_{(a,b)}^m$ denotes the edge whose displacement vector is $(ma,-mb)$, labeled by ``e". Moreover, $h_{(a,b)}$ denotes the edge whose displacement vector is $(a,b)$. A convex generator is a commutative formal product of the symbols $e_{(a,b)}$ and $h_{(a,b)}$, where the factor $h_{(a,b)}$ can not repeat.

\begin{definition}
    Let $\Lambda$ be a convex generator, its ECH index is given by
    $$
    I(\Lambda)=2(L(\Lambda)-1)-h(\Lambda),
    $$
    where $L(\Lambda)$ is the number of lattice point in the region enclosed by $\Lambda$ an the axes, and $h(\Lambda)$ is the number of edges labeled by ``h".
\end{definition}

\begin{example}
    We now list some small ECH index convex generators:
    \begin{itemize}
        \item There is only one $I=0$ convex generator, the integral path constant $(0,0)$;
        \item There are no convex generator with $I=1$;
        \item $I=2$; $e_{1,0}$ and $e_{0,1}$;
        \item $I=3$; $h_{1,1}$;
        \item $I=4$; $e_{1,0}^2$, $e_{1,1}$ and $e_{0,1}^2$;
        \item $I=5$; $h_{2,1}$ and $h_{1,2}$;
        \item $I=6$; $e_{1,0}^3$, $e_{0,1}^3$, $e_{2,1}$, $e_{1,2}$ and $e_{1,0}e_{0,1}$.
    \end{itemize}
\end{example}

\begin{definition}
    If $\Lambda$ is a convex generator and $X_{\Omega}$ is a convex toric domain, define the symplectic action of $\Lambda$ with respect of $X_{\Omega}$ by
    $$
    A_{X_{\Omega}}(\Lambda)=A_{\Omega}(\Lambda)=\sum_{\nu \in \textrm{Edge}(\Lambda)}\vec{\nu} \times p_{\Omega,\nu}.
    $$
    Here for a edge $\nu$ the vector $\vec{\nu}$ denotes the lower right point of $\nu$ minus the left upper left point. The point $p_{\Omega,\nu}$ denotes a point in the tangent line to $\partial \Omega$ parallel to $\nu$. The operation $\times$ is the determinant of a pair of vectors in the plane.
\end{definition}

\subsubsection{Combinatorial ECH of boundary of convex toric domains}

Let $X_{\Omega}$ be a smoth convex toric domain with boundary $Y=\partial X_{\Omega}$, this is a smooth star-shaped hypersurface in $\mathbb{R}^4$ diffeomorphic to $S^3$, therefore the $1$-form $\lambda_{std}$ restricts to a contact form on $Y$. This contact form in general is degenerate, we can perturb the contact form as in the case of the unit cotangent bundle of the Klein bottle, in this case there exist a bijection of a convex generator and a ECH generator. For a convex generator $\Lambda$, denote $i(\Lambda)$ the corresponding ECH generator.

\begin{lemma}\label{lemma: pertubtoricdomain}(See \cite{Hutchings0}, Lemma 5.4.)
    Let $X_{\Omega}$ be a smooth convex toric domain with boundary $Y$. Then for every $\epsilon,L>0$, there is a contact form $\lambda$ in $Y$ with the following properties:
    \begin{enumerate}
        \item $\lambda$ is nondegenerate.
        \item $\lambda=f(\lambda_{std}|_{Y})$, where $f:Y\rightarrow \mathbb{R}_{>0}$ is a smooth function with $||f-1||_{C^0}<\epsilon$.
        \item All hyperbolic orbits with action less than $L$ are positive hyperbolic, and all embedded elliptic Reeb orbits with action less than $L$ are $L$-positive.
        \item There is a bijection $i$, from the set of extended convex generators $\Lambda$ with action less than $L$ to the set of orbit sets $\alpha$ for $\lambda$ with action less than $L$, such that if $\alpha=i(\Lambda)$, then:
            \begin{itemize}
                \item $\alpha$ is admissible (in the ECH sense) if and ony if $\Lambda$ is a convex generator
                \item $|\mathcal{A}(\alpha)-\mathcal{A}_{\Omega}(\Lambda)|<\epsilon$.
                \item $I(\alpha)=I(\Lambda)$.
                \item $J_0(\alpha)=I(\Lambda)-2(x(\Lambda)+y(\Lambda))-e(\Lambda)$.
            \end{itemize}
    \end{enumerate}
\end{lemma}

We have moreover a similar lemma for the cotangent disk bundle of the flat Klein bottle with boundary $U^*(K)$. In fact, we use our combinatorial description, we see that all embedded elliptic Reeb orbit in our pertubations of $(U^*(K),\lambda_{std}|_{U^*(K)})$ can be taken to be $L$-positive.

The next lemma is important to ensure that in the decomposition of a generator in the Theorem \ref{thm: combbeyond}, each factor is a $K$-lattice path.,

\begin{lemma}\label{lemma: 0homology}
    Let $\alpha$ be an admissible orbit set in $U^*(K)$ which has no negative hyperbolic orbits comming from different Klein bottles $K_-$ or $K_+$. If $\alpha$ is nullhomologous in $D^*(K)$, then it is nullhomologous in $U^*(K)$.
\end{lemma}

\begin{proof}
    Let $\pi:T^*K\rightarrow K$ be the natural projection. The induced map 
    $$
    \pi_{*}:H_1(T^*K,\mathbb{Z})\rightarrow H_1(K,\mathbb{Z}),
    $$
    satisfies the following:
    \begin{itemize}
        \item $\pi_{*}([\gamma_{p/q}])=(\overline{q},2p)$;
        \item $\pi_{*}([h_{1/0}^1])=(\overline{0},1)$;
        \item $\pi_{*}([h_{1/0}^2])=(\overline{1},1)$;
        \item $\pi_{*}([h_{-1/0}^1])=(\overline{0},-1)$;
        \item $\pi_{*}([h_{-1/0}^2])=(\overline{1},-1)$.
    \end{itemize}
    Therefore, if the orbit set $\alpha$ has zero homology in $D^*K$ we consider three cases:

    \textbf{Case 1}. $\alpha$ has only toric orbits.

    In this case, write $\alpha=\prod_{i}\gamma_{p_i/q_i}$. Since $[\alpha]=0\in H_1(D^*K,\mathbb{Z})$, we have $\pi_*([\alpha])=0$. Therefore
    $$
    2\sum_i p_i=0\quad \textrm{and}\quad \sum_i q_i\equiv 0\pmod{2}.
    $$
    Therefore, $[\alpha]=0\in H_1(U^*K,\mathbb{Z})$.

    \textbf{Case 2}. $\alpha$ has at least one negative hyperbolic orbit from $K_+$.
    
    In this case $\alpha$ must have both negative hyperbolic orbit $h_{1/0}^1$ and $h_{1/0}^2$, since the second coordinate of $\pi_{*}([h_{1/0}^1]$ and $\pi_{*}([h_{1/0}^2]$ are odd number, while the second coordinate of $\pi_{*}([\gamma_{p/q}])$ is an even number. Therefore, $\pi_*([\alpha])=0$, implies that
    $$
    2\sum_i p_i+2=0\quad \textrm{and}\quad \sum_i q_i\equiv 1\pmod{2},
    $$
    and this is exactly the condition $[\alpha]=0\in H_1(U^*K,\mathbb{Z})$.

    \textbf{Case 3.} $\alpha$ has at least one negative hyperbolic orbit from $K_-$.

    This case is analogous to the last case, and this finish the proof.
\end{proof}

\subsection{Combinatorial obstruction theorem}

In this section we show a combinatorial obstruction theorem for a symplectic embedding of a convex toric domain $X_{\Omega}$ into $D^*K$. 

\begin{definition}
    Let $X_{\Omega}$ be a convex toric domain. Consider $\Lambda$ and $\Lambda'$ a convex generator for $X_{\Omega}$ and a $K$-convex lattice path for $D^*K$. We write $\Lambda\leq_{X_{\Omega,D^*K}} \Lambda'$, if the following conditions hold:

    \begin{itemize}
        \item $I(\Lambda)= I(\Lambda')$.
        \item $A_{\Omega}(\Lambda)\leq A_{D^*k}(\Lambda')$.
        \item $x(\Lambda)+y(\Lambda)-h(\Lambda)/2\geq n_{\Lambda'}/2+m(\Lambda')-1$.
    \end{itemize}
\end{definition}

\begin{theorem}\label{thm: combbeyond}
    Suppose that $X_{\Omega}\hookrightarrow D^*K$ is a symplectic embedding, where $X_{\Omega}$ is a convex toric domain. Let $\Lambda'$ be a $K$-latice path (not of type IV) with no edges labeled by $``h"$. Then there exist a convex generator $\Lambda$ and a nonnegative integer $n$, and a product decomposition $\Lambda=\Lambda_1\cdots \Lambda_n$ and $\Lambda'=\Lambda_1'\cdots \Lambda_n'$ such that:
    \begin{itemize}
        \item Each $\Lambda_i$ is a $K$-convex lattice path and $\Lambda_i$ is a convex generator.
        \item $\Lambda\leq_{X_{\Omega,D^*K}} \Lambda'$.
        \item Given $i,j\in \{1,...,n\}$, if $\Lambda_i\neq \Lambda_j$, or $\Lambda_i'\neq \Lambda_j'$, then $\Lambda_i$ and $\Lambda_j$ have no elliptic orbit in common.
        \item If $S$ is any subset of $\{1,..,n\}$, then $I(\prod_{i\in S}\Lambda_i)=I(\prod_{i\in S}\Lambda_i')$.
    \end{itemize}
\end{theorem}

\begin{proof}

    We can assume by shringking $X_{\Omega}$ that there is a symplectic embedding $\varphi:X_{\Omega}\rightarrow int(D^*K)$. (The theorem for the original toric domain follows by a limiting argument).

    We have that $W=D^*K\setminus \varphi(\textrm{int}(X_{\Omega})$ is a weakly exact symplectic cobordism from $(U^*(K),\lambda_{std}|_{U^*(K)})$ to $(\partial X_{\Omega},\lambda_{std}|_{\partial X_{\Omega}})$.

    Let $L>\mathcal{A}_{U^*(K)}(\Lambda)$ and $\epsilon>0$. By the Lemma \ref{lemma: pertubtoricdomain}, we have $\lambda'$ a contact form in $\partial X_{\Omega}$ and we have $\lambda$ the contact form in $U^*(K)$ such that the filtred ECH complex is generated by $K$-convex lattice path.

    Let $J_+$ be a generic $\lambda$-compatible almost complex structure on $\mathbb{R}\times U^*(K)$, let $J_-$ be a generic $\lambda'$-compatible almost complex structure in $\mathbb{R}\times \partial X_{\Omega}$, and let $J$ be a generic cobordism-admissible almost complex structure on $\overline{W}$ restricting to $J_+$ and $J_-$ on the ends of the completion.

    We now prove that the cobordism $(W,\omega,J)$ is $L$-tame. Consider $d$ a positive integer, let $\alpha_+$ be an orbit set for $\lambda$ with $\mathcal{A}_{U^*(K)}(\alpha_+)<L/d$, and let $C$ be an irreducible embedded curve in $\mathcal{M}^J(\alpha_+,\alpha_-)$ with $I(dC)\leq 0$. To prove the inequality \ref{eq: L-tame}, we suppose by contradiction that it does not hold, then we have
    $$
    2g(C)+ind(C)+h(C)+2e_{L}(C)\leq 1.
    $$
    Since $J$ is generic, $ind(C)\geq 0$. The parity of $ind(C)$ is equal to the number of ends of $C$ at positive hyperbolic orbits. Write $h(C)=h^+(C)+h^-(C)$, where $h^+(C)$ is the number of ends of $C$ at positive hyperbolic orbits, and $h^-(C)$ the number of ends of $C$ at negative hyperbolic orbits. This implies that
    $$
    g(C)=ind(C)=h^+(C)=e_{L}(C)=0,
    $$
    and $h^-(C)\leq 1$.

    By Lemma \ref{lemma: pertubtoricdomain}$(3)$, every simple elliptic Reeb orbit in $\alpha_-$ is $L$-positive, therefore $\alpha_-$ must be the empty set. Suppose first that $h^-(C)=1$, this implies that $\alpha_+$ is an orbit set with exactly one negative hyperbolic orbit and all other are elliptic orbits with action less than $L/d$, but in this case because $\alpha_-=\emptyset$, this implies that $\alpha_+$ is nullhomologous in $D^*K$, this is impossible by the Lemma \ref{lemma: 0homology}. Therefore, we can suppose that $h^-(C)=0$ and $\alpha_+$ is a toric orbit set with only elliptic orbits with $[\alpha_+]=0$. In this case, $d[C]$ represents a relative homology class in $H_2(D^*K,d \alpha_+,\emptyset)$, and because $H_2(D^*(K),\mathbb{Z})=0$, we have $I(\alpha_{+}^d)=I(dC)$. Let $\Lambda_+$ be the $K$-lattice path associated to $\alpha_+^d$, we have $I(dC)=I(\alpha_+^d)=I(\Lambda_+)\geq 0$, with equality only if $\Lambda_+$ is the empty set or the orbit set formed by the four negative hyperbolic orbits, because $\alpha_+$ has no negative hyperbolic orbits we must have $\alpha_+=\emptyset$, this is a contradiction since $C$ is nonempty. Thus $I(dC)>0$, contradicting our hypothesis. 

    By the combinatorial description of the ECH differential, we have that $\Lambda$ is a cycle which is non-exact, that is, $[\Lambda]_{ECH}\neq 0$. Let $\Phi$ the ECH cobordism map, it is well know that $\Phi$ commutes with the $U$-map in homology. Follows from Corollary \ref{U-map-Isomorphism} that $U^k[\Lambda]_{ECH}=[\emptyset]_{ECH}$. Therefore, we have
    $$
    U^k\Phi[\Lambda]_{ECH}=\Phi U^k[\Lambda]_{ECH}=\Phi[\emptyset]_{ECH}=[\emptyset]_{ECH}\in ECH(\partial X_{\Omega},\xi_{std},0).
    $$
    Since the $U$-map in the ECH of $\partial X_{\Omega}$ is an isomorphism, we must have $\Phi[\Lambda]_{ECH}\neq 0$.

    Consider the chain map which induce the cobordism map given by the Theorem \ref{thm: cobordism map}
$$
\begin{array}{cccc}
\phi \ : & \! ECC(U^*K,\lambda_{\epsilon},0,J_+) & \! \longrightarrow
& \! ECC(\partial B(r),\lambda_{\epsilon},0,J_-) \\
\end{array}
$$

By Theorem \ref{thm: cobordism map}$(3)(i)$, there is an admissible orbit set $\alpha'$ for $\lambda'$ in $\partial X_{\Omega}$ with $\mathcal{A}_{\Omega}(\alpha)<L$ and a broken $J$-holomorphic current
$$
\mathcal{C}=(\mathcal{C}(N_-),\mathcal{C}(N_{-}+1),...,C(N_+))
$$
from $\Lambda$ to $\alpha'$. By Theorem \ref{thm: cobordism map}$(3)$ we have $I(\mathcal{C})=0$. Moreover, because $J$ is generic, $I(\mathcal{C}(i))>0$ if $i\neq 0$. Follows from Proposition \ref{prop: cobLtame}$(1)$ that $I(\mathcal{C}(0))\geq 0$. Since 
$$
0=I(\mathcal{C})=\sum_{i}I(\mathcal{C}(i))=0,
$$
it follows that $N_-+N_+=0$.

Write $\mathcal{C}=\mathcal{C}(0)=\sum_k d_k\mathcal{C}_k$. Let $\Lambda_k$ and $\Lambda_k'$ denote the $K$-lattice path and the convex lattice path respectively for which $C_k\in \mathcal{M}^J(\Lambda_k,\Lambda_k')$. We have
$$
I(\Lambda_k)=I(\Lambda_k')\quad and\quad A(\Lambda_k')\leq A(\Lambda_k)+2\epsilon,
$$
because $I(C_k)=0$ by the Proposition \ref{prop: cobLtame}$(2)$. We now prove the inequality 
$$
x(\Lambda_k')+y(\Lambda_k')-h(\lambda_k)/2\geq n_{\Lambda_k}/2+m(\Lambda_k)-1.
$$
where $n_{\Lambda_k}$ is the number of negative hyperbolic orbits in $\Lambda_k$ and $m(\Lambda_k)$ is the total multiplicity of the elliptic Reeb orbits in $\Lambda_k$.

To prove the claim we use the $J_0$ index. Observe that, because $H_2(D^*(K)\setminus \varphi(X_{\Omega}),\mathbb{Z})=0$, we can compute
$$
J_0(C_k)=J_0(\Lambda_k)-J_0(\Lambda_k').
$$
%%%%%%%%%%%%%%
The first term is computed as follows, 
\begin{eqnarray*}
J_0(\Lambda_k)-I(\Lambda_k)&=&-2c_{\tau}(Z)+\mu_{\tau}'(\Lambda_k)-\mu_{\tau}(\Lambda_k)\\
&=& -n_{\Lambda_k}-\sum_i CZ_{\tau}(\gamma_i^{m_i})\\
&=& -n_{\Lambda_k}+n_{\Lambda_k}-e(\Lambda_k).
\end{eqnarray*}
%%%%%%%%%%%%%%%%%%%
Therefore, $J_0(\Lambda_k)=-e(\Lambda_k)+I(\Lambda_k)$. The term $J_0(\Lambda_k')$ was computed by Hutchings and we obtain
$$
J_0(\Lambda_k')=I(\Lambda_k')-2x(\Lambda_k')-2y(\Lambda_k')-e(\Lambda_k').
$$
Thus, we obtain
\begin{eqnarray*}
    J_0(C_k)&=&-e(\Lambda_k)+I(\Lambda_k)-I(\Lambda_k')+2x(\Lambda_k')+2y(\Lambda_k')+e(\Lambda_k')\\
    &=& 2x(\Lambda_k')+2y(\Lambda_k')+e(\Lambda_k')-e(\Lambda_k)\\
    &\geq& 2g(C_k)-2+\sum_{i}2n_i^{+}-1+\sum_j 2n_j^{-}-1\\
    &\geq& -2+n_{\Lambda_k}+2m(\Lambda_k)-e(\Lambda_k)+e(\Lambda_k')+h(\Lambda_k').
\end{eqnarray*}
This implies the claim. Now taking $\epsilon\rightarrow 0$, we concluded the prove of the theorem.
\end{proof}

\subsection{Computation on Gromov width}

We prove now the Theorem \ref{thm: GROMOVWDITH}.

\begin{proof} of Theorem \ref{thm: GROMOVWDITH}. Consider the $K$-lattice path $\Lambda'=h_{(0,-1)}^1 h_{(0,-1)}^2 e_{(0,-1)}^{k}e_{(1,0)}e_{(0,1)}^{k+1}$ to apply the Theorem \ref{thm: combbeyond}, the unique possible factorization of $\Lambda'$ into generators with zero homology repeat the elliptic orbit $e_{(0,1)}$. Therefore, by the the third bullet of Theorem \ref{thm: combbeyond}, there is no factorization and we obtain a lattice path $\Lambda$ with the following properties.
$$
A(\Lambda)\leq A(\Lambda'),\quad I(\Lambda)=4(k+1)\quad and\quad x(\Lambda)+y(\Lambda)\geq 2(k+1).
$$
%%%%%%%%%%%%%%
\textbf{Claim.} The constrains above implies that $\Lambda$ is one of the following generators
$$
e_{1,0}^{2(k+1)},\quad e_{0,1}^{2(k+1)},\quad e_{2(k+1),1}\quad \textrm{or}\quad e_{1,2(k+1)}.
$$

Indeed, we have $I(\Lambda)=2(L(\Lambda)-1)-h(\Lambda)$. We can write this as follows. Observe that
$$
L(\Lambda)=x(\Lambda)+y(\Lambda)+h(\Lambda)+m_{e}^{+}(\Lambda)+i(\Lambda),
$$
where $m_{e}^+(\Lambda)$ is the total multiplicity of elliptic orbits which are not in the $x$-axis and $i(\Lambda)$ is the number of interior lattice points in the region defined by $\Lambda$. Therefore, we have
\begin{eqnarray*}
    4(k+1)&=& -2+2i(\Lambda)+2x(\Lambda)+2y(\Lambda)+h(\Lambda)+2m_{e}^+(\Lambda)\\
    &\geq& 4(k+1)+2i(\Lambda)+2m_{e}^+(\Lambda)+h(\Lambda)-2.
\end{eqnarray*}
Thus we obtain that 
$$
2(i(\Lambda)+m_{e}^+(\Lambda))+h(\Lambda)\leq 2.
$$
However, if $i(\Lambda)\geq 1$, then $m_{e}^+(\Lambda)\geq 1$ or $h(\Lambda)\geq 1$. Therefore, we can suppose $i(\Lambda)=0$. In this case $m_{e}^+(\Lambda)\in \{0,1\}$. If $m_{e}^+(\Lambda)=0$, we must have the generators
$$
h_{2(k+1),1},  \quad h_{1,2(k+1)},\quad e_{1,0}^{2(k+1)}\quad and\quad e_{0,1}^{2(k+1)}
$$
If $m_{e}^+(\Lambda)=1$, we obtain the following generators
$$
e_{2(k+1),1}\quad \textrm{and}\quad e_{1,2(k+1)}.
$$
The generator must have a even number of positive hyperbolic Reeb orbits, since its ECH index is even. Therefore, the unique possibilities are the elliptic ones. This proves the claim.

Thus we obtain the inequality
$$
2(k+1)r=A_{X_{\Omega}}(\Lambda)\leq A_{D^*K}(\Lambda')=2(k+1)+1
$$
for any $k$, so that
$$
r\leq \frac{2(k+1)+1}{2(k+1)},
$$
taking the limit $k\rightarrow +\infty$ we obtain $r\leq 1$. Therefore, $c_{Gr}(D^*K,\omega_{std})\leq 1$. From Remark \ref{remark; Gromovwidth} we also have $c_{Gr}(D^*K,\omega_{std})\geq 1$.
\end{proof}

\end{document}